\documentclass[11pt]{amsart}

\title[A Bernstein theorem for two-valued minimal graphs]%
{A Bernstein theorem for two-valued\\ minimal graphs in dimension four}

\author{Fritz Hiesmayr}
\address{University College London, 25 Gordon Street, London WC1H 0AY}
\email{f.hiesmayr@ucl.ac.uk}

\usepackage[utf8]{inputenc}
\usepackage[T1]{fontenc}	
\usepackage{amsmath}		
\usepackage{amssymb}		
\usepackage{enumerate}
\usepackage{amsthm}
\usepackage{mathrsfs}
\usepackage{mathtools}
\mathtoolsset{showonlyrefs,showmanualtags}

\usepackage{cite}

\usepackage{bm}
\usepackage{bbm} 
\usepackage[titletoc,title]{appendix}

\usepackage{enumitem,linegoal}
\usepackage{calc}

\usepackage{stmaryrd}

\usepackage{tikz-cd}

\usepackage{tabularx}
\newcolumntype{Y}{>{\centering\arraybackslash}X}

\usepackage{caption} 
\usepackage{booktabs}

\usepackage{chngcntr}
\usepackage{apptools}
\AtAppendix{\counterwithin{thm}{section}}

\usepackage{url}

\usepackage[hidelinks]{hyperref}

\newcommand\restr[2]{{
	\left.\kern-\nulldelimiterspace 
	#1 
	\vphantom{\big|} 
	\right|_{#2} 
}}

\newcommand\srestr[2]{{
	\left.\kern-\nulldelimiterspace 
	#1 
	\right|_{#2} 
}}

\newcommand{\mres}{\mathbin{\vrule height 1.6ex depth 0pt width
0.13ex\vrule height 0.13ex depth 0pt width 1.3ex}}

\DeclareFontFamily{U}{MnSymbolC}{}
\DeclareSymbolFont{MnSyC}{U}{MnSymbolC}{m}{n}
\DeclareFontShape{U}{MnSymbolC}{m}{n}{
    <-6>  MnSymbolC5
   <6-7>  MnSymbolC6
   <7-8>  MnSymbolC7
   <8-9>  MnSymbolC8
   <9-10> MnSymbolC9
  <10-12> MnSymbolC10
  <12->   MnSymbolC12}{}
\DeclareMathSymbol{\intprod}{\mathbin}{MnSyC}{'270}

\newcommand{\indic}{\mathbf{1}}

\newcommand{\calA}{\mathcal{A}}
\newcommand{\calB}{\mathcal{B}}
\newcommand{\calH}{\mathcal{H}}

\newcommand{\calG}{\mathcal{G}}
\newcommand{\calK}{\mathcal{K}}

\newcommand{\calR}{\mathcal{R}}
\newcommand{\calS}{\mathcal{S}}

\newcommand{\calC}{\mathcal{C}}
\newcommand{\calV}{\mathcal{V}}

\newcommand{\calU}{\mathcal{U}}
\newcommand{\calZ}{\mathcal{Z}}

\newcommand{\bC}{\mathbf{C}}

\newcommand{\bP}{\mathbf{P}}
\newcommand{\bR}{\mathbf{R}}

\newcommand{\bL}{\mathbf{L}}
\newcommand{\bN}{\mathbf{N}}
\newcommand{\bE}{\mathbf{E}}

\newcommand{\bZ}{\mathbf{Z}}

\newcommand{\orig}{0}

\newcommand{\ct}[1]{\text{\rmfamily\upshape #1}}

\newcommand*\intdiff{\mathop{}\!\mathrm{d}}

\newcommand{\abs}[1]{\lvert#1\rvert}
\newcommand{\norm}[1]{\lVert#1\rVert}

\newcommand{\cur}[1]{\llbracket #1 \rrbracket}

\DeclareMathOperator{\IV}{\mathbf{IV}}

\DeclareMathOperator{\I}{\mathbf{I}}

\newcommand{\Gr}{Gr}

\newcommand{\vartan}{\mathrm{VarTan}}

\newcommand*\diff{\mathop{}\!\mathrm{d}}

\newcommand{\eps}{\epsilon}

\newcommand{\bdary}{\partial}

\newcommand{\clos}[1]{\overline{#1}}

\DeclareMathOperator\Lip{Lip}

\DeclareMathOperator{\linspan}{span}

\DeclareMathOperator\spt{spt}
\DeclareMathOperator\reg{reg}

\DeclareMathOperator\sing{sing}

\DeclareMathOperator{\dist}{dist}

\DeclareMathOperator{\inter}{int}

\DeclareMathOperator{\Per}{Per}

\let\div\relax 
\DeclareMathOperator{\div}{div}
\DeclareMathOperator{\cpcty}{cap}

\let\Re\relax
\DeclareMathOperator{\Re}{Re}

\DeclareMathOperator{\graph}{graph}

\theoremstyle{plain}
\newtheorem{thm}{Theorem}[section]
\newtheorem*{thm*}{Theorem}
\newtheorem{lem}[thm]{Lemma}
\newtheorem*{lem*}{Lemma}
\newtheorem{prop}[thm]{Proposition}
\newtheorem{cor}[thm]{Corollary}
\newtheorem*{cor*}{Corollary}

\theoremstyle{definition}

\theoremstyle{remark}
\newtheorem{rem}[thm]{Remark}
\newtheorem*{rem*}{Remark}

\newtheorem{claim}{Claim}
\newtheorem*{claim*}{Claim}
\newtheorem*{notation*}{Notation}

\newtheorem*{quest*}{Question}

\theoremstyle{plain}
\newtheorem{innercustomthm}{Theorem}
\newenvironment{customthm}[1]
{\renewcommand\theinnercustomthm{#1}\innercustomthm}
{\endinnercustomthm}

\numberwithin{equation}{section}

\begin{document}

\begin{abstract}
	
We prove a Bernstein-type theorem for two-valued minimal graphs in the
four-dimensional Euclidean space $\bR^4$.
This states that two-valued functions defined on the entire $\bR^3$,
and whose graph is a minimal surface, must necessarily be linear.
This is a two-valued analogue of the classical Bernstein theorem, which 
asserts that in dimensions up to $n + 1 \leq 8$, an entire single-valued
minimal graph is linear.
The main contrast with the single-valued theory is the presence of a large
set of singularities in the graphs of two-valued functions.
Indeed two-valued minimal graphs are neither area-minimising,
nor is the regularity theory of elliptic PDE directly available in this setting.
We obtain structure results for the blowdown cones of two-valued minimal graphs,
valid in dimension $n + 1 \leq 7$, proving in particular that they are 
smoothly immersed away from an $(n-2)$-rectifiable set that includes
its branch points.
In dimension four we go further, and completely classify the possible blowdown
cones using a combinatorial argument. We show that they must be a union of two
three-dimensional planes: this is the key to the proof of the Bernstein
theorem.

\end{abstract}

\maketitle

\section*{Introduction}

The Bernstein theorem is one of the most emblematic results in the theory
of minimal surfaces. It answers the following question. Let $u: \bR^n \to \bR$
be a globally defined function whose graph $G$ is minimal. Must $u$ be affine linear?
This question was first studied by Bernstein~\cite{Bernstein_1916}, who answered
the question affirmatively when $n = 2$, using methods from complex analysis.
However his argument proved hard to generalise, and Fleming~\cite{Fleming_Oriented_Plateau}
opened the way to higher-dimensional extensions by framing the problem in
measure-theoretic terms.
We summarise the novelty of Fleming's approach, which culminated in a proof of Bernstein's
theorem in all dimensions up to $n + 1 = 8$.
Fleming~\cite{Fleming_Oriented_Plateau} showed that $G$ is weakly asymptotic
at large scales to an $n$-dimensional area-minimising cone $\bC$. Moreover if $u$
is non-linear then $\bC$ must be singular.
%
De Giorgi~\cite{deGiorgi_Frontiere_orientate} next observed that additionally
the blowdown cone $\bC$ of a non-linear $G$ must be cylindrical: there is 
an $(n-1)$-dimensional area-minimising cone $\bC_0$ in $\bR^n$ so that
$\bC = \bC_0 \times \bR e_{n+1}$.
Consequently if none of the $(n-1)$-dimensional area-minimising cones in $\bR^n$ are
singular then there cannot be any non-linear minimal graphs in $\bR^{n+1}$ either.
Fleming~\cite{Fleming_Oriented_Plateau} showed that this is true when $n = 3$;
next this was extended to $n = 4$ by
Almgren~\cite{Almgren_Some_Interior_Regularity_and_Bernstein}, and finally by
Simons~\cite{Simons_Minimal_varieties} to $n \leq 7$.
This is sharp, as Bombieri--de Giorgi--Giusti~\cite{BombierideGiorgiGiusti69}
showed that the cones $\bC_{p} = \{ (X,Y) \in \bR^p \times \bR^p \mid
\abs{X} = \abs{Y} \}$ are area-minimising for all $p \geq 4$.
They also constructed, for every one of these cones, a  minimal graph
in $\bR^{2p + 1}$ which at infinity is asymptotic to $\bC_{p} \times
\bR e_{2p + 1}$, thus proving that the Bernstein theorem could not hold
in dimensions larger than $n + 1 \geq 9$.

Here we propose to study the analogous question for two-valued functions:
must entire two-valued minimal graphs be affine linear?
Historically multi-valued functions were introduced to geometric measure theory
by Almgren~\cite{Almgren1981} in his monumental regularity theory
for area-minimising currents, valid in any codimension $k$.
To this end Almgren used functions taking values in the set of unordered
$Q$-tuples $\calA_Q(\bR^k)$.
(Almgren's regularity theory has been streamlined
by De Lellis--Spadaro in a series of papers, including an initial paper
revisiting the theory of $Q$-valued functions~\cite{DLS_Q_Valued_Functions,
DLS_RegularityI,DLS_RegularityII,DLS_RegularityIII}.)

Here we work exclusively in codimension one, with two-valued functions
defined to be those taking values in $\calA_2(\bR)$, the set of
unordered pairs of real numbers $\{ a_1,a_2 \} = \{ a_2, a_1 \}$ with $a_1,a_2 \in \bR$.
The graph $G$ of a two-valued function is called minimal if it is a
critical point for the $n$-dimensional area functional. 
Such two-valued minimal graphs provide the local picture of stable minimal
hypersurfaces near multiplicity-two branch points.
Indeed, a result of Wickramasekera~\cite{Wic_MultTwoAllard} proves
the following: near a point where it has a multiplicity two tangent hyperplane,
a stationary codimension one varifold with stable regular part is equal to the graph of a
two-valued function.
Their inherent interest aside, this provides an additional motivation for the
study of entire two-valued minimal graphs $G$, defined in terms of some
$u: \bR^n \to \calA_2(\bR)$,
because they arise as blow-up models along degenerating sequences.

Unlike their single-valued counterparts, two-valued minimal graphs are not
area-minimising. Indeed, they contain branched and immersed singularities,
neither of which exist in area-minimising hypersurfaces.
Although a two-valued minimal graph $G$ is still asymptotic at infinity
to a stationary cone $\bC$, this is not area-minimising either.
In short, the regularity theory of area-minimising cones in $\bR^{n+1}$ is of no use here,
even when $n + 1 \leq 8$.
Nonetheless, two-valued minimal graphs can be shown to be stable (see 
Lemma~\ref{lem_graph_ambient_stability}): the second variation of the
area functional is non-negative.
The presence of a large singular set (that includes branch points)
may therefore come as a surprise at a first sight. Indeed, perhaps
the most celebrated regularity theorems for stable stationary hypersurfaces%
---respectively obtained by Schoen--Simon~\cite{SchoenSimon81}
and Wickramasekera~\cite{Wickramasekera14}---give that their singular
set is small, and has codimension at least seven.
However neither result applies here, because both impose \emph{a priori}
hypotheses on the singular set that two-valued minimal graphs fail to satisfy: 
Schoen--Simon ask that it be $\calH^{n-2}$-negligible, and
Wickramasekera imposes a (significantly weaker) structural hypothesis,
forbidding so-called classical singularities.
In particular, all the results we obtain are completely independent of the
regularity theory developed by Wickramasekera~\cite{Wickramasekera14}.

There are additional difficulties associated to working with two-valued
minimal graphs. A real-valued function $u$ whose graph $G$ is minimal 
satisfies an elliptic quasilinear PDE: the minimal surface equation
$(1 + \abs{Du}^2) \Delta u - D_i u D_j u D_{ij} u = 0$.  
This makes the elliptic regularity theory available. For example, assuming only
that $u$ is $C^{1,\alpha}$-regular for some $\alpha \in (0,1)$ one finds
that in fact $u$ must be smooth.
(In fact via de Giorgi--Nash--Moser theory it suffices to consider
Lipschitz-regular $u$ to obtain the same conclusion.)
This is not applicable in the two-valued setting.
Simon--Wickramasekera~\cite{SimonWickramasekera07} constructed two-valued
minimal graphs defined on bounded subsets of $\bR^n$ by solving an adapted
Dirichlet problem. Their examples have an $(n-2)$-dimensional branch set,
along which they look approximately like functions of the form
$(x_1,\dots,x_n) \mapsto \Re(x_1 + \ct{i} x_2)^{k/2}$, where $k \geq 3$ is an odd integer.
(See also the subsequent work of Krummel~\cite{Krummel_Multivalued_Dirichlet}
for an extension of these constructions to higher codimensions.)
Simon--Wickramasekera~\cite{SimonWickramasekera16}
later showed that a two-valued minimal graph of class $C^{1,\alpha}$
for some $\alpha \in (0,1)$ is necessarily $C^{1,1/2}$-regular.
The examples above with $k = 3$ show that this is optimal.

We overcome these difficulties and prove the following theorem,
which is our main result; see Theorem~\ref{thm_bernstein_dim_four_chapter}.

\begin{customthm}{1}
\label{thm_bernstein_dim_four_intro}
Let $\alpha \in (0,1)$ and $u \in C^{1,\alpha}(\bR^3;\calA_2(\bR))$ be a two-valued
function whose graph $G$ is minimal. Then $u$ is linear, and its graph is a
union of two three-dimensional planes.
\end{customthm}

Rosales~\cite{Rosales_Two_valued_minimal_graphs} has proved a similar result in dimension
$n + 1 = 3$. However the proofs bear few similarities, because Rosales obtains
curvature estimates via a logarithmic cut-off trick,
an argument unavailable in larger dimensions.

It is crucial here that no bound is assumed for the growth of $u$ or its derivatives.
Notice in particular the strong contrast with two-valued harmonic functions, of which
an abundance of non-linear examples exists without such a hypothesis:
take for example the functions
$(x_1,\dots,x_n) \mapsto \Re (x_1 + \ct{i} x_2)^{k/2}$ for odd $k \geq 3$ we
already mentioned above.
(This mirrors a phenomenon already present in the single-valued setting, where the
Liouville theorems for harmonic functions require a bound on $u$ or its derivatives.)

When one does in fact assume a bound for the growth of $u$, then the conclusion
of Theorem~\ref{thm_bernstein_dim_four_intro} extends to all dimensions;
see Theorem~\ref{thm_grad_bounds_vert_line}.

\begin{customthm}{2}
	\label{thm_bernstein_bounded_growth_intro}
Let $\alpha \in (0,1)$ and $u \in C^{1,\alpha}(\bR^n;\calA_2(\bR))$ be a
two-valued function whose graph $G$ is minimal, and for which there exists
a constant $C > 0$ so that $\sup_{D_R} \{ \abs{u_1} + \abs{u_2} \} \leq C R$
for all $R \geq 1$.
Then $u$ is linear, and its graph is a union of two $n$-dimensional planes.
\end{customthm}

The method of proof of Theorem~\ref{thm_bernstein_bounded_growth_intro}
diverges from the classical, single-valued theory.
There the corresponding result---that a single-valued minimal graph
$u: \bR^n \to \bR$ with bounded growth is necessarily linear---can essentially
be proved using standard Harnack inequalities. We cannot follow this line
of argument here, and instead develop an alternative proof.
The key here is a regularity lemma, where we prove that Lipschitz
two-valued minimal graphs must in fact be $C^{1,\alpha}$-regular for some $\alpha \in (0,1)$.
This result is not completely new, as Becker-Kahn and Wickramasekera
have obtained the same conclusion via a frequency function argument%
~\cite{BeckerKahn_Wic_Perso}.
Our approach here is different and original: it combines a
geometric, inductive argument with anterior results of
Becker-Kahn~\cite{Spencer_Two_valued_graphs_arbitrary_codimension};
see Section~\ref{sec_regularity_lemma_for_two_valued_minimal_graphs}.

Of the two results, Theorem~\ref{thm_bernstein_dim_four_intro} is by far the more
difficult to prove. However both results require \emph{a priori} estimates,
among which we cite the stability inequality (see Lemma~\ref{lem_graph_ambient_stability}),
an area estimate (see Proposition~\ref{prop_area_estimate})
and an interior gradient estimate (see Lemmas~\ref{lem_interior_gradient_bounds}
and~\ref{lem_int_grad_bounds_form_in_small_ball}). 
It is imperative that all three hold in the presence of branch point singularities.
This means that, although we can adapt the proofs used in the single-valued theory,
these modifications introduce significant technical intricacies.

Broadly speaking, the proof of Theorem~\ref{thm_bernstein_dim_four_intro} follows a
similar argument as the proof of the single-valued Bernstein theorem.
Let $G = \graph u$ be a minimal two-valued graph,
corresponding to some function $u \in C^{1,\alpha}(\bR^n;\calA_2(\bR))$.
Next let $G_j = \graph u_j$ be a sequence of two-valued minimal graphs obtained
by homothetically rescaling $G$ along some sequence of factors $\lambda_j \to + \infty$:
$u_j(x) = \lambda_j^{-1} u(\lambda_j x)$ for all $x \in \bR^n$.
The area bounds imply that a subsequence of these is convergent, with weak limit
a so-called blowdown cone $\bC$ of $G$. 
Using the well-known monotonicity formula for area,
it suffices to prove that $\norm{\bC}(B_1) = 2 \omega_n$---%
the same area as the sum of two $n$-dimensional planes---%
to deduce that $u$ is linear.
Because $\norm{\bC}(B_1) = \lim_{j \to \infty} \calH^n(G_j \cap B_1)$, 
one might harbour hope that the initial area estimates can be sharpened to
obtain this.
Although improvements on the initial bounds%
---see Section~\ref{subsec_improved_area_estimates}---%
indeed play an important role in our argument, these do not directly lead to a proof. 
Instead we classify the cones that can arise from such a blowdown construction,
and ultimately prove that in dimension $n + 1 = 4$ the cone $\bC$ must be a sum of two
three-dimensional planes $\Pi_1,\Pi_2 \in \Gr(3,4)$, with multiplicity one:
$\bC = \abs{\Pi_1} + \abs{\Pi_2}$.

In this sense our argument is similar to the single-valued theory, where 
one instead shows that, provided $n + 1 \leq 8$, the only possible blowdown
cones are $n$-dimensional planes. This is a direct consequence of the
regularity of area-minimising currents.
As this is unavailable here, even in low dimensions the blowdown cone $\bC$
could in principle contain a large set of singularities $\sing \bC$,
which can be immersed, branched or more complicated yet.
The singular set $\sing \bC$ can be divided into strata $\calS^0(\bC),
\dots,\calS^n(\bC)$, where $\calS^k(\bC)$ is the set of points
whose tangent cones have a spine of dimension at most $k$, along which it
is translationally invariant.
The Almgren--Federer stratification theorem~\cite{Almgren1981} gives that
$\dim_{\calH} \calS^k(\bC) \leq k$; this was later improved by
Naber--Valtorta~\cite{NaberValtorta_rectifiability_stationary_varifolds},
who proved that in fact $\calS^k(\bC)$ is countably $k$-rectifiable using
quantitative stratification methods.
The lower strata can thus be gathered into a set with codimension at least two,
which can be excised using capacity arguments.

The top stratum $(\calS^n \setminus \calS^{n-1})(\bC)$ is the set of branch points
$\calB(\bC)$, where at least one tangent cone is of the form $Q \abs{\Pi}$,
where $Q \geq 2$ and $\Pi \in \Gr(n,n+1)$. 
It turns out to be relatively easy to prove that the multiplicity at these points is
$Q = 2$. This allows the application of the results of Wickramasekera~\cite{Wic_MultTwoAllard}
and Krummel--Wickramasekera~\cite{KrumWic_FinePropsMinGraphs}, which together
imply that $\calB(\bC)$ 
too is $(n-2)$-rectifiable and can be excised.
A large portion of the text is therefore dedicated to
$(\calS^{n-1} \setminus \calS^{n-2})(\bC)$, the only stratum
for which these excision arguments are impossible on account of its size.
This is composed of those points $X \in \sing \bC$ that have at least one
so-called classical tangent cone: this is a cone $\bP \in \vartan(\bC,X)$ 
of the form $\bP = \sum_i Q_i \abs{\pi_i}$, where $Q_i \geq 1$ and
$\pi_i$ and $n$-dimensional half-planes which meet along an $(n-1)$-dimensional
axis $L$ say.
Via a diagonal extraction argument, one obtains a new sequence $(G_j \mid j \in \bN)$
of two-valued minimal graphs so that $\abs{G_j} \to \bP$ as $j \to \infty$.
(We use the same notation as for our original sequence for sake of simplicity.)
The aim is therefore to classify the classical cones that can arise as weak
limits of two-valued minimal graphs. We prove that, provided $n + 1 \leq 7$,
there exist two $n$-dimensional planes $\Pi_1,\Pi_2 \in \Gr(n,n+1)$ so that
$\bP = \abs{\Pi_1} + \abs{\Pi_2}$: Sections~\ref{sec_classical_limit_cones}
to~\ref{sec_classical_cones_vertical} are devoted to this classification.
(The dimension restriction is related to our improved area estimates, see
Corollary~\ref{cor_qualitative_estimates}. When $n + 1 \leq 7$, these
exclude \emph{a priori} the possibility that $\bP = 2( \abs{\Pi_1} + \abs{\Pi_2}
+ \abs{\Pi_3})$ for example. 
When $n + 1 \geq 8$, and this initial reduction is unavailable, our classification
arguments come up short.)
This being established, the results of Wickramasekera~\cite{Wic_MultTwoAllard}
then demonstrate that the support of $\bC$ is smoothly immersed away from
an $(n-2)$-dimensional rectifiable set that includes its branch points.
In fact this is true for any varifold that arises as a weak limit from a sequence
of two-valued minimal graphs. For blowdown cones---those that arise from a
sequence of homothetic rescalings of $G = \graph u$ for some
$u \in C^{1,\alpha}(\bR^n;\calA_2(\bR))$---we obtain the following stronger
conclusion; see Corollary~\ref{cor_sing_of_limit_cones} and
Theorem~\ref{thm_degiorgi_splitting_in_low_dimensions}.

\begin{customthm}{3} 
	\label{thm_structure_blowdown_cones_intro}
Let $\alpha \in (0,1)$ and $2 \leq n \leq 6$.
Let $u \in C^{1,\alpha}(\bR^n;\calA_2)$ be
an entire two-valued minimal graph and $\bC$ be a blowdown cone of
$\abs{G}$ at infinity.
Then $\bC$ is stationary, stable, smoothly immersed away from an $(n-2)$-rectifiable set,
and in fact
\begin{enumerate}[label = (\roman*), font = \upshape]
	\item \label{item_blowdown_cylindrical}
		either $\bC$ is cylindrical of the form
		$\bC = \bC_0 \times \bR e_{n+1}$,
	\item \label{item_blowdown_cylindrical_plus_plane}
		or $\bC = \abs{\Pi} + \bC_0 \times \bR e_{n+1}$
		where $\Pi \in \Gr(n,n+1)$,
	\item \label{item_blowdown_sum_of_planes}
		or $\bC$ is the sum of two planes $\Pi_1,\Pi_2 \in \Gr(n,n+1)$,
		$\bC = \abs{\Pi_1} + \abs{\Pi_2}$.
\end{enumerate}
\end{customthm}

Finally, in dimension $n + 1 = 4$ we complete the classification of blowdown
cones: we first show that actually the blowdown cone $\bC$ is equal to a sum
of at most three planes, each of which has multiplicity one, and then apply
a combinatorial argument to a sort of dual cellular decomposition to conclude that
there are two three-dimensional planes $\Pi_1,\Pi_2 \in \Gr(3,4)$ so that
$\bC = \abs{\Pi_1} + \abs{\Pi_2}$.
By the monotonicity formula for the area, this will conclude the proof of
Theorem~\ref{thm_bernstein_dim_four_intro}.

\subsection*{Overview}
In Section~\ref{sec_two_valued_functions} we set notation for two-valued functions.
In Section~\ref{sec_two_valued_minimal_graphs} we first define two-valued
minimal graphs, next quote results of Simon--Wickramasekera~\cite{SimonWickramasekera16}
and Krummel--Wickramasekera~\cite{KrumWic_FinePropsMinGraphs}, and finally establish
some of their basic properties, including stability.
In the next two sections we prove estimates for two-valued minimal graphs:
bounds for their area in Section~\ref{sec_area_estimates} and an interior
gradient estimate in Section~\ref{sec_grad_estimates}.
In Section~\ref{sec_regularity_lemma_for_two_valued_minimal_graphs}
we prove a regularity lemma, which establishes that Lipschitz two-valued minimal graphs
are in fact $C^{1,\alpha}$-regular.
In Section~\ref{sec_single_valued_minimal_graphs}
we prove a lemma for single-valued minimal graphs in the style of Jenkins--Serrin:
this is used in our classification of classical limit cones;
see Section~\ref{sec_classical_cones_vertical}.
In Section~\ref{sec_multiplicity_limit_cones} we show that the branch points of limits of
two-valued minimal graphs have multiplicity two, and establish some preliminary
lemmas in anticipation of Section~\ref{sec_classical_cones_vertical}.
The following three sections are devoted to the classification of classical limit cones.
In Section~\ref{sec_classical_limit_cones} we adapt an argument of Schoen--Simon
for a preliminary analysis.
Section~\ref{sec_non_vertical_cone} treats limit cones that
are not vertical, and Section~\ref{sec_classical_cones_vertical} concludes the
analysis by considering vertical cones. 
In Section~\ref{sec_blowdown_cones_and_asymptotic_analysis}
we prove Theorems~\ref{thm_bernstein_bounded_growth_intro}
and~\ref{thm_structure_blowdown_cones_intro}. Finally we give a proof of the main result,
Theorem~\ref{thm_bernstein_dim_four_intro}
in Section~\ref{sec_bernstein_theorem_in_four_dimensions}.
In Appendix~\ref{sec_appendix_regularity_and_singularities} we quote some results
from geometric measure theory, notably the statement of
the branched sheeting theorem of Wickramasekera~\cite{Wic_MultTwoAllard} 
in Section~\ref{subsec_wic_mult_two_allard}.

\subsection*{Acknowledgements}

We are grateful for comments from Neshan Wickramasekera, Costante Bellettini,
Spencer Becker--Kahn, Ben Sharp and Edriss Titi;
we would also expressly like to thank Neshan Wickramasekera
and Costante Bellettini for their support and encouragement.
This work was supported by the UK Engineering and Physical Sciences Research
Council (EPSRC) under grants EP/S005641/1 and EP/L016516/1.


\section{Two-valued functions}

\label{sec_two_valued_functions}

\subsection{Unordered pairs}

Let $\calA_2(\bR)$ be the set of unordered pairs of real numbers,
abbreviated $\calA_2$ when no confusion is possible. An element of $\calA_2(\bR)$
is written $\{ x , y \} = \{ y , x \}$.
One can define $\calA_2(\bR)$ as the set obtained
by taking the quotient of $\bR^2$ under the action by the group $\bZ_2$ which
transposes the two elements. The quotient map is then
$\bR^2 \to \calA_2(\bR): (x,y)  \mapsto \{ x, y \}$.
Alternatively one could define $\calA_2(\bR)$ as the non-empty subsets of $\bR$
counting at most two elements.
More generally we can also take unordered pairs of elements of any set $X$,
thus forming $\calA_2(X)$.
Here $X$ is usually either a finite-dimensional vector space $\bE$ say
or a subset thereof.
Apart from $\bE = \bR$ we also use $\bE = \bR^{n+1}$
and $\bE = \bL(\bR^n;\bR)$, the space of linear functions $\bR^n \to \bR$.
Consider the example where $X = \bE$ is a real vector space.
Scalar multiplication is unambiguously defined, but
$\calA_2(\bE)$ is not naturally a vector space, because the sum (or
difference) of two elements $u,v \in \calA_2(\bE)$ cannot be made sense
of as an element in $\calA_2(\bE)$.
Any map $\Phi: X \to Y$ between two sets can be extended to a two-valued
analogue by setting $\mathbf{\Phi}: \calA_2(X) \to \calA_2(Y): \{ x,y \} \mapsto
\{ \Phi(x),\Phi(y) \}$.
For example we define $\abs{ \{ u_1,u_2\} } = \{ \abs{u_1} , \abs{u_2} \}
\in \calA_2(\bR_{\geq 0})$.
Given a pair $\{ x_1, x_2 \} \in \calA_2(\bR)$ of real numbers, we define
its \emph{average} and \emph{symmetric difference} by
$x_a = \frac{1}{2}(x_1 + x_2)$ and $x_s = \{ \pm (x_1 - x_2) \}$
respectively. 
Moreover write $x_+ = \max \{ x_1 , x_2 \} = x_1 \vee x_2$ and
$x_- = \min \{ x_1 , x_2 \} = x_1 \wedge x_2$.

Let $\bE$ be a vector space with norm $\norm{\cdot}$. For a pair
$\{ x , y \} \in \calA_2(\bE)$ we write $\norm{ \{ x , y \}}
= \norm{x} + \norm{y}$.
We define a metric on $\calA_2(\bE)$ by
$\calG: \calA_2(\bE) \times \calA_2(\bE) \to \bR_{\geq 0}$ with
$\calG(u,v) = \min ( \norm{u_1 - v_1} + \norm{u_2 - v_2},
\norm{u_1 - v_2} + \norm{u_2 - v_1} )$.
The analogous construction works for arbitrary metric spaces.

\subsection{Definition and function spaces}

Let $A \subset \bR^n$ be an arbitrary set. A \emph{two-valued function} on $A$ is
a function $A \to \calA_2(\bR)$.
To $u$ we can associate its average
$u_a = \frac{1}{2}(u_1 + u_2)$ and symmetric difference
$u_s = \{ \pm \frac{1}{2} (u_1 - u_2) \}$.
Similarly we write $u_+ = \max \{u_1,u_2\} = u_1 \wedge u_2$ and
$u_- = \min \{ u_1 ,u_2 \} = u_1 \vee u_2$.

More generally we consider two-valued maps $A \to \calA_2(X)$, where $X$ is
an arbitrary set. For use most often $X$ will either be $\bR$ or a real vector space,
say $\bR^n$ or $\bL(\bR^n;\bR)$. Those taking values in $\calA_2(\bR^n)$ we shall call 
\emph{two-valued vector fields}.
Consider a map $\Phi : X \to \bR$. Given $u: A \to \calA_2(X)$ we can
compose to $\Phi \circ u: A \to \calA_2(\bR)$. If we are given
two different functions $\Phi,\Psi: X \to \bR$ then we may take the sum
$(\Phi \circ u) + (\Psi \circ u) := (\Phi + \Psi) \circ u$, although
one the sum of two-valued functions cannot be defined as a two-valued function.
The same of course holds for other binary operations.

As $\calA_2 = \calA_2(\bR)$ is a metric space, endowing it with the corresponding topology
and Borel $\sigma$-algebra allows us to define \emph{measurable} and \emph{continuous}
two-valued functions.
Given $\alpha \in (0,1)$, a two-valued function on $A$ is called
\emph{$\alpha$-H\"{o}lder continuous} if $\limsup_{x \neq y \in A}
\abs{x - y}^{-\alpha} \calG(u(x),u(y)) < \infty$, and that $u$ is
\emph{Lipschitz continuous} if this holds with $\alpha = 1$.
Let $\Omega \subset \bR^n$ be open. We write $C^{0,\alpha}(\Omega;\calA_2)$ and
$\Lip(\Omega;\calA_2)$ for the functions that are locally $\alpha$-H\"{o}lder continuous
or locally Lipschitz respectively.
Both notions are again defined for functions taking values in $\calA_2(X)$
for any metric space $(X,d)$, for example $X = \bL(\bR^n;\bR)$.

We say that a function $l: \bR^n \to \calA_2$ is \emph{linear} if there exist
two single-valued linear functions $l_i \in \bL(\bR^n;\bR)$ so that $l = \{l_1,l_2 \}$.
A function $u: \Omega \to \calA_2$ is called
\emph{differentiable at} $x$ if there is a two-valued linear function $l = \{l_1,l_2\}$
so that for all $v \in \bR^n$, $t^{-1} \calG(u(x+ tv),\{ u_i (x) + l_i(tv) \}) \to 0$
as $t \to 0$.
If this exists, we write $Du(x) = l$ and call this the \emph{derivative} of $u$
at $x$. This defines a two-valued function $Du : \Omega \to \calA_2(\bL(\bR^n;\bR))$.
Moreover we write $u \in C^1(\Omega;\calA_2)$ if it is differentiable at all points
$x \in \Omega$ and the function $Du$ is continuous, and for $\alpha \in (0,1)$
we write $u \in C^{1,\alpha}(\Omega;\calA_2)$ if
$Du \in C^{0,\alpha}(\Omega;\calA_2(\bL(\bR^n;\bR)))$.

The corresponding H\"{o}lder norms are defined in the usual way; given 
$\alpha \in (0,1)$ we write 
$C^0(\clos{\Omega};\calA_2), C^{0,\alpha}(\clos{\Omega};\calA_2),
C^1(\clos{\Omega};\calA_2),C^{1,\alpha}(\clos{\Omega};\calA_2)$ 
for the spaces formed by those functions on $\Omega$ for which the respective
norms are finite.

\subsection{Selections and singularities}

Consider a two-valued function $u: \Omega \to \calA_2$, and let $\Omega' \subset \Omega$
be an open subset. We say that two functions $u_1,u_2: \Omega \to \bR$
define a \emph{selection} for $u$ on $\Omega'$ if
$u = \{ u_1,u_2 \}$ on $\Omega'$.
Most often one seeks selections with favourable properties:
for example a selection is called continuous if $u_1,u_2 \in C^0(\Omega')$,
of class $C^1$ if $u_1,u_2 \in C^1(\Omega')$, and smooth if $u_1,u_2 \in C^\infty(\Omega')$.

The \emph{touching set} of a function $u \in C^1(\Omega;\calA_2)$
is $\calZ_u = \{ x \in \Omega \mid u_1(x) = u_2(x) \}$,
and the \emph{critical set} of $u$ is
$\calK_u = \{ x \in \Omega \mid u_1(x) = u_2(x), Du_1(x) = Du_2(x) \}$.
We also set $\calC_u = \calZ_u \setminus \calK_u =
\{ x \in \Omega \mid u_1(x)= u_2(x), Du_1(x) \neq Du_2(x) \}$.
This is a relatively open subset of $\calZ_u$, and
$\calB_u,\calK_u \subset \calZ_u$ are relatively closed.
For every point $y \in \Omega \setminus \calK_u$ there is $\sigma > 0$ so
that on $D_\sigma(y)$ there are two functions $u_{1,y},u_{2,y} \in C^1(D_\sigma(y))$
so that $u = \{ u_{1,y} , u_{2,y} \}$.
A point $x \in \Omega$ is called a \emph{branch point}
if there is no radius $\sigma > 0$ for which
$u$ admits a $C^1$ selection on $D_\sigma(x)$.
They form the set $\calB_u \subset \calK_u$.
Points in $\calK_u \setminus \calB_u$ are sometimes called \emph{false branch points}.

Let $u : \Omega \to \calA_2$ be a two-valued function. Its 
\emph{graph} $G \subset \Omega \times \bR$ is the set
\begin{equation}
	\graph u =
	G = \{ (x,X^{n+1}) \in \Omega \times \bR \mid 
	X^{n+1} = u_1(x) \text{ or } u_2(x) \}.
\end{equation}
This may be considered as a varifold inside $\Omega \times \bR$, in which case
we write $\abs{G} = \abs{\graph u} \in \IV_n(\Omega \times \bR)$ as is customary.
(We emphasise that throughout we do not consider the graph as a subset of $\Omega
\times \calA_2(\bR)$, as one might expect by interpreting the term more literally.)
In general of course $G = \graph u$ of $u \in C^1(\Omega;\calA_2)$
is not regular and instead has an immersed set of singularities $\calC(G)$,
where its tangent cone is a union of two distinct $n$-dimensional planes,
and the branch set $\calB(G)$, where its tangent cone is a hyperplane
with multiplicity two.
They are related to the singularities of $u$ via the orthogonal projection $P_0$
onto $\bR^n \times \{ 0 \}$,
$P_0(\sing G \cup \{ X \in \reg G \mid \Theta(\norm{G},X) = 2 \}) = \calZ_u$,
$P_0(\calB(G) \cup \{ X \in \reg G \mid \Theta(\norm{G},X) = 2 \}) = \calK_u$,
and $P_0(\calC(G)) = \calC_u$.

We briefly remark on the case where $n = 1$ and 
$\Omega = I \subset \bR$ is an open interval in the real line.
Let $u \in C^1(I;\calA_2)$. This automatically has
$\calB_u = \emptyset$, even while $\calK_u$ may be non-empty.
In other words, we can always find $u_1,u_2 \in C^1(I)$ so that
$u = \{ u_1,u_2 \}$ on $I$, although some arbitrary choices have
to be made if $\calK_u \neq \emptyset$.
We will later use this elementary observation in the following context.
Let $\Omega \subset \bR^n$ be open, and $u \in C^{1}(\Omega;\calA_2)$.
Let $y \in \Omega$ and $v \in \bR^n$ be arbitrary. Write $l_y
\subset \{ y + tv \mid t \in \bR \} \cap \Omega$ for the connected
component containing $y$. This corresponds to an interval $I \subset \bR$.
Via this identification, the restriction of $u$ to $l_y$ defines
a two-valued function in $C^{1}(I;\calA_2)$. Hence we can find two
functions $u_{1,y},u_{2,y} \in C^1(I)$ so that
$u(y + tv) = \{ u_{1,y}(t),u_{2,y}(t) \}$ even though possibly
$l_y \cap \calB_u \neq \emptyset$.
Moreover if $u \in C^{1,\alpha}(\Omega;\calA_2)$ for some $\alpha \in (0,1)$
then we can impose $u_{1,y},u_{2,y} \in C^{1,\alpha}(I)$ as well.

\subsection{Integrals of two-valued functions}
\label{subsec_two_valued_integrals}

Let $p \in [1,+\infty)$. We write $L^p(\Omega;\calA_2)$ for the space of
two-valued functions $u : \Omega \to \calA_2$ with
$\int_\Omega \norm{u}^p < \infty$, and $L^\infty(\Omega;\calA_2)$
for those functions for which $\norm{u}$ is essentially bounded.
We define the \emph{integral} of a two-valued function $u \in L^1(\Omega;\calA_2)$
by $\int_\Omega u = 2 \int_\Omega u_a = \int_\Omega u_1 + u_2.$
Sometimes we also write $\int_\Omega u(x) \intdiff x = \int_\Omega u$.

Let $u \in C^1(\clos{\Omega};\calA_2)$ and $\Phi: \bR \to \bR$ be so that
$\Phi \circ u \in L^1(\Omega;\calA_2)$. Then we have that 
$\int_\Omega (\Phi \circ u) (1 + \abs{Du}^2) = \int_G \Phi \intdiff \calH^n$,
where on the right-hand side the integral is over $G = \graph u \subset \Omega \times \bR$.

Consider $u \in L^1(\Omega;\calA_2)$,
$A \subset \bR$ be a Borel subset, and $\indic_A$ be its indicator function.
As $u$ takes values in $\calA_2$, the pre-image of $A$ under $u$ is not
defined. However, we can define the integral
$\int_{u \in A} u := \int u (\indic_A \circ u)
= \int u_1 ( \indic_A \circ u_1) + u_2( \indic_A \circ u_2)$.
This can be generalised further if we consider $\Phi: \bR \to \bR$ so that
$\Phi \circ u \in L^1(\Omega;\calA_2)$, by setting
$\int_{u \in A} \Phi \circ u = \int (\Phi \circ u) (\indic_A \circ u)$,
where recall the product of the two functions is a well-defined two-valued
function in this specific context.
We will often use a variant of this, where in fact $\Phi: \Omega \times \bR^n \times \bR
\to \bR$ is so that the two-valued function $x \in \Omega \mapsto \Phi(x,Du(x),u(x))$ is
integrable, and we consider integrals of the form
$\int_{u \in A} \Phi(x,Du(x),u(x)) \intdiff x$, or say
$\int_{u \in A, Du \in B} \Phi(x,Du(x),u(x)) \intdiff x$
where  $B \subset \bR^n$ is another Borel subset.

\subsection{Some results for two-valued functions}

From the usual Rademacher theorem, see for example~\cite[Thm.\ 5.2]{Simon84},
one easily derives a two-valued analogue.

\begin{lem}
\label{lem_rademacher_two_valued}
Let $u \in \Lip(D_1;\calA_2)$ be a two-valued Lipschitz function. Then $u$
is differentiable $\calH^n$-a.e.\ in $D_1$.
\end{lem}

Using a general form of the Arzel\`{a}--Ascoli theorem, for example~%
\cite[Thm.~47.1]{Munkres_Topology}, we obtain the following lemma specialised
to sequences of two-valued Lipschitz functions.

\begin{lem}
\label{lem_arzela_ascoli_two_valued}
Let $(u_j \mid j \in \bN)$ be a sequence of two-valued Lipschitz functions
on $D_1$. If there is $C > 0$ so that $\sup_{D_1} \norm{u_j} + \norm{Du_j} \leq C$
for all $j \in \bN$ then there is a subsequence which converges locally uniformly
to a Lipschitz function $u \in \Lip(D_1;\calA_2)$.
\end{lem}

\section{Two-valued minimal graphs}

\label{sec_two_valued_minimal_graphs}

\subsection{Definition and basic properties}

Let $\Omega \subset \bR^n$ be an open set, and let $\alpha \in (0,1)$.
We say that $u \in C^{1,\alpha}(\Omega;\calA_2(\bR))$ defines a \emph{two-valued
minimal graph} if its graph
\begin{equation}
\abs{G} = \abs{\graph u } \in \IV_n(\Omega \times \bR)
\end{equation}
is stationary as a varifold in the open cylinder $\Omega \times \bR$.
We will often abbreviate this by saying that $u \in C^{1,\alpha}(\Omega;\calA_2)$
is a two-valued minimal graph.

When an open subdomain $\Omega' \subset \Omega \setminus \calB_u$ is simply
connected then there is a selection $u_1,u_2 \in C^2(\Omega')$ so that
$u = \{ u_1, u_2 \}$ in $\Omega'$.
Inside $\Omega' \times \bR$ the graph can be decomposed like
$\abs{G} = \abs{ \graph u_1 } + \abs{ \graph u_2 }$,
which we frequently abbreviate by writing $G_i = \graph u_i$.

The stationarity of $\abs{G}$ is inherited by $\abs{G_1}$ and $\abs{G_2}$.
This in turn means that $u_1,u_2$ are both smooth, that is $u_1,u_2
\in C^\infty(\Omega')$ and they separately solve the minimal surface equation
\begin{equation}
	\label{eq_class_MSE}
	\div T(Du_i) = 0 \text{ in $\Omega$ for $i = 1,2$,}
\end{equation}
where here and throughout we write, for all $p \in \bR^n$
\begin{equation}
	T_k(p) = \frac{p_k}{\sqrt{1 + \abs{p}^2}}
	\text{ for $k = 1,\dots,n$.}
\end{equation}
The vector $T(Du) \in \bR^n$ is the horizontal part of $- \nu$, the downward-pointing
unit normal to the graph $G$. We habitually write $v = \sqrt{1 + \abs{Du}^2}$, so
that also $T(Du) = Du / v$.

If we interpret the equation \eqref{eq_class_MSE} in a weak sense we can show that
for all test functions $\phi \in C_c^1(\Omega \setminus \calB_u)$
\begin{equation}
	\label{eq_integral_MSE}
	\int_{\Omega} \langle T(Du) , D \phi \rangle = 0, 
\end{equation}
using a partition of unity argument to reduce to the case where $\phi$ is
supported in a simply connected domain $\Omega' \subset \Omega \setminus \calB_u$.
Note that as $u$ is a two-valued function, the integral above is understood to be
$\int_{\Omega} \langle T(Du) , D \phi \rangle =
\int_{\Omega} \langle T(Du_1) , D \phi \rangle + \langle T(Du_2) , D \phi \rangle$,
as explained in Section~\ref{subsec_two_valued_integrals}.

This can be generalised to arbitrary test functions $\phi \in C_c^1(\Omega)$
by reasoning as follows. By assumption, there is a compact subset $K \subset \Omega$
so that $\phi$ vanishes identically outside $K$.
Take a sequence $(\eta_j \mid j \in \bN)$ of functions in
$C_c^1(\Omega)$ with 
\begin{enumerate}[label = (\arabic*), font = \upshape]
	\item $0 \leq \eta_j \leq 1$ in $\Omega$ for all $j$,
	\item $\eta_j \equiv 0$ on $\calB_u \cap K$ for all $j$,
	\item $\eta_j \to 1$ $\calH^n$-a.e. in $\Omega$,
	\item $\int_{\Omega} \abs{D \eta_j} \to 0$ as $j \to \infty$.
\end{enumerate}
Such a sequence exists because $\calH^{n-1}(\calB_u) = 0$,
see Section~\ref{sec_properties_of_branch_set} and~\cite[Ch.~5.6]{Evans15}
for example.
Actually by~\cite{KrumWic_FinePropsMinGraphs} the branch set is countably
$(n-2)$-rectifiable, see again Section~\ref{sec_properties_of_branch_set}.
(There we also explain that analogous sequences can be constructed to cut off
the critical set $\calK_u$, provided only that the set $G$ is not in fact
equal to a single-valued minimal graph, with multiplicity two.)

Following the steps in the proof of the following proposition 
one gets 
$\int_\Omega \langle T(Du), D \phi \rangle = 0$ for all $\phi \in C_c^1(\Omega)$.
We skip over the details of this, and move on to the proof of the more general identity
expressed in the proposition, which allows the test function to depend on $u$.

\begin{prop}
\label{prop_test_function}
Let $u \in C^{1,\alpha}(\Omega;\calA_2) \cap C^0(\clos{\Omega};\calA_2)$ define
a two-valued minimal graph and let 
$\Phi \in \Lip(\Omega \times \bR \times \bR^n)$
have $\spt \Phi \subset \Omega' \times \bR \times \bR^n$ for
some $\Omega' \subset \subset \Omega$.
If 
\begin{equation}
\label{eq_condition_weak_twovalued_MSE}
\int_{\Omega \setminus \calB_u} \abs{D(\Phi(x,u,Du))} < +\infty
\end{equation}
then
\begin{equation}
\label{eq_weak_MSE}
\int_\Omega \langle T(Du),  D(\Phi(x,u,Du)) \rangle = 0.
\end{equation}
\end{prop}

\begin{proof}
Define a function $\phi \in C^1_c(\Omega)$ by setting
$\phi(x) = \Phi(x,u,Du)$ for all $x \in \Omega$.
This is smooth away from $\calB_u$ and has support contained inside $\Omega'$.
Let $(\eta_j \mid j \in \bN)$ be a sequence with the same properties
as above.
As $(1 - \eta_j) \phi$ vanishes near $\calB_u$ it is a valid test function
in the integral identity \eqref{eq_integral_MSE}, yielding 
$\int_{\Omega} \langle T(Du), D \phi \rangle (1 - \eta_j)
- \int_\Omega \langle T(Du), D \eta_j \rangle \phi = 0$.

That the second integral goes to zero is a direct application of
H\"{o}lder's inequality as above.
For the first integral, we can bound the integrand like
$\abs{\langle T(Du) , D \phi \rangle (1 - \eta_j)}
\leq \abs{ D \phi}$ almost everywhere in $\Omega$---on $\Omega
\setminus \calB_u$ to be precise.
The bounding function is integrable by assumption,
so that again we can use dominated convergence to let $j \to \infty$
and deduce that $\int_\Omega \langle T(Du) , D( \Phi(x,u,Du)) \rangle = 0$.
\end{proof}

\subsection{Orientation and the current structure}

Let $\Omega \subset \bR^n$ be an open set, $\alpha \in (0,1)$ and 
let $u \in C^{1,\alpha}(\Omega;\calA_2)$ define a two-valued minimal
graph in $\Omega \times \bR$.
At all regular points $X = (x,X^{n+1}) \in \reg G \cap \Omega \times \bR$ we
write $\nu(X)$ for the upward-pointing unit normal. 
In terms of a smooth selection $u = \{ u_{1},u_2 \}$ for $u$ on a small disc
$D_\rho(x)$ we have $\nu(x,u_i(x)) = (1 + \abs{Du_i(x)}^2)^{-1/2}  (-Du_i(x),1)$.
This is also defined at the branch points, and hence $\nu$ defines a continuous
vector field on $(\reg G \cup \calB(G)) \cap \Omega \times \bR$, which is
moreover is smooth on $\reg G$.
However $\nu$ cannot be continuously extended to the set of classical singularities
$\calC(G) \cap \Omega \times \bR$.
As this set has $\calH^n(\calC(G) \cap \Omega \times \bR) = 0$ we can however
still define an integer multiplicity rectifiable current
$\cur{G}$ by integrating over $\reg G \cap \Omega \times \bR$.
Moreover $\bdary \cur{G} = 0$ in $\Omega \times \bR$ and
$\cur{G} \in \I_n(\Omega \times \bR)$.

\subsection{Properties of the branch set}
\label{sec_properties_of_branch_set}

Let $\alpha \in (0,1)$, and let $u \in C^{1,\alpha}(\Omega;\calA_2)$ be an arbitrary
two-valued minimal graph.
Using an approach based on a so-called frequency function,
Simon--Wickramasekera~\cite{SimonWickramasekera16} proved the following.

\begin{thm}[\cite{SimonWickramasekera16}]
Let $\alpha \in (0,1)$ and $u \in C^{1,\alpha}(\Omega;\calA_2)$ be
a two-valued minimal graph.
Then the branch set of $u$ is either empty or $\dim_{\calH} \calB_u = n-2$ and
$\calH^{n-2}(\calB_u) \neq 0$.
\end{thm}
In particular $\calH^s(\calB_u) = 0$ for all $s > n-2$. We use the fact
that $\calH^{n-1}(\calB_u) = 0$ to derive our area estimates.

\begin{cor}[\cite{SimonWickramasekera16}]
Let $\alpha \in (0,1)$ and $u \in C^{1,\alpha}(\Omega;\calA_2)$ be 
a two-valued minimal graph. If $G$ is not equal to a single-valued
minimal graph with multiplicity two
then $\dim_{\calH} \calK_u = n-2$ and $\calH^{n-2}(\calK_u) \neq 0$.
\end{cor}

Starting also from a frequency function, this was taken further
by Krummel--Wickramasekera~\cite{KrumWic_FinePropsMinGraphs},
who proved the following.

\begin{thm}[\cite{KrumWic_FinePropsMinGraphs}]
	\label{thm_krumwic_finepropsmingraphs}
Let $\alpha \in (0,1)$ and $u \in C^{1,\alpha}(\Omega;\calA_2)$ be
a two-valued minimal graph defined on $\Omega \subset \bR^n$.
Then $\calB_u$ is either empty or is countably $n-2$-rectifiable.
Moreover if $G$ is not equal to a single-valued minimal graph with multiplicity two
then $\calK_u$ is countably $n-2$-rectifiable. 
\end{thm}
This represents a significant improvement over~\cite{SimonWickramasekera16}
because it allows the excision of the branch set via capacity arguments
which we now detail.

Following the presentation given by Evans--Gariepy in~\cite[Ch.\ 4.7]{Evans15} we 
let $K^p$ be the space of functions $f: \bR^n \to \bR_{\geq 0}$ with
$f \in L^{p^*}$ and $\abs{Df} \in L^p$, where $p^* = np / (n-p)$ is the
Sobolev conjugate of $p$.
The \emph{$p$-capacity} of a set $A \subset \bR^n$ is then defined to be
$\cpcty_p A = \inf \{ \int \abs{Df}^p \mid f \in K^p, A \subset \inter \{ f \geq 1 \} \}$.
When $A$ is compact then equivalently
$\cpcty_p A = \inf \{ \int \abs{Df}^p \mid f \in C_c^\infty(\bR^n), f \geq \indic_A \}$.
We do not list the basic properties of capacity listed in~\cite{Evans15},
and only record here the following two results. For $p \in (1,n)$, if a set $A \subset \bR^n$
has $\calH^{n-p}(A) < \infty$ then $\cpcty_p A  = 0$, and if $\cpcty_p A = 0$ then
$\dim_{\calH} A \leq n-p$.
From this it follows the countably $(n-p)$-rectifiable sets have $p$-capacity zero.
(When $p = 1$ then more is true, as a subset $A \subset \bR^n$ has 
$\cpcty_1 A = 0$ if and only if $\calH^{n-1}(A) = 0$; see~\cite[Ch.~5.6]{Evans15}.)

Now let $p \in [1,n)$ and $A \subset \bR^n$ be a compact set with $\cpcty_p A = 0$.
With only little effort one establishes the existence of a sequence of cutoff
function $(\eta_j \mid j \in \bN)$ with the following properties for all $j$:
\begin{enumerate}[label = (\roman*)]
	\item \label{item_first_prop_compact_capacity}
		$\eta_j \in C_c^1(\bR^n)$,
	\item $0 \leq \eta_j \leq 1$,
	\item \label{item_third_prop_compact_capacity} 
		$\eta_j \equiv 1$ on $(A)_{r_j}$ for some $r_j \to 0$,
	\item \label{item_fourth_prop_compact_capacity}
		$\eta_j \to 0$ $\calH^n$-a.e.,
	\item $\int_{\bR^n} \abs{D \eta_j}^p \to 0$.
\end{enumerate}
Moreover if $U$ is an open set containing $A$ then we can additionally
impose that $\spt \eta_j \subset U$. For the branch set of two-valued
minimal graphs this yields the following.

\begin{lem}
	\label{lem_EG_cutoff_sequence_two}
Let $\alpha \in (0,1)$, and $u \in C^{1,\alpha}(D_2;\calA_2)$ be a
two-valued minimal graph. Then there is a sequence of functions
$(\eta_j \mid j \in \bN)$ with for all $j$,
\begin{enumerate}[label = (\roman*), font = \upshape]
	\item $\eta_j \in C_c^1(D_2)$,
	\item $0 \leq \eta_j \leq 1$ on $D_2$,
	\item $\eta_j \equiv 1$ on $(\calB_u)_{r_j} \cap D_1$ for some $r_j \to 0$,
	\item $\eta_j \to 0$ $\calH^n$-a.e.\ as $j \to \infty$,
	\item $ \int_{D_2} \abs{D \eta_j}^2 \to 0$ as $j \to \infty$.
\end{enumerate}
\end{lem}

In the same way as for the results cited above, one obtains a version of this
valid for the set $\calK_u$ provided the set $G$ is not equal to a single-valued
minimal graph, with multiplicity two.
In various places it is useful to modify the sequence from 
Lemma~\ref{lem_EG_cutoff_sequence_two} and construct it on the graph itself.

\begin{cor}
	\label{cor_cutoff_sequence_cpcty_2_zero}
Let $\alpha \in (0,1)$, and $u \in C^{1,\alpha}(D_2;\calA_2)$ be a
two-valued minimal graph. If $\calB(G) \neq \emptyset$
then there is a sequence of functions
$(\eta_j \mid j \in \bN)$ with for all~$j$,
\begin{enumerate}[label = (\roman*), font = \upshape]
	\item $\eta_j \in C_c^1(D_2 \times \bR)$,
	\item $0 \leq \eta_j \leq 1$ on $D_2 \times \bR$,
	\item $\eta_j \equiv 1$ on $(\calB(G))_{r_j} \cap D_1 \times \bR$ for some $r_j \to 0$,
	\item $\eta_j \to 0$ $\calH^n$-a.e. on $\reg G \cap D_2 \times \bR$ as $j \to \infty$,
	\item $\int_{\reg G \cap D_2 \times \bR} \abs{\nabla_G \eta_j}^2 \intdiff \calH^n
		\to 0$ as $j \to \infty$.
\end{enumerate}
\end{cor}
\begin{proof}
Let $(\eta^{0}_j \mid j \in \bN)$ be a sequence with the properties
listed in Lemma~\ref{lem_EG_cutoff_sequence_two}, but adapted to the set $\calK_u$
instead, and we additionally impose
that $\spt \eta^0_j \subset  D_{3/2}$ for all $j$.
Inside this disc there is $A > 0$ so that
$-A < \min_{D_{3/2}} u_- \leq \max_{D_{3/2}} u_+ < A$,
where $u_- = \min \{ u_1,u_2 \}$ and $u_+ = \{ u_1,u_2\}$.
Let $\tau \in C_c^1(\bR)$
be a classical cutoff function with $\tau \equiv 1$ on $[-1,1]$,
$\spt \tau \subset [-2,2]$ and $\abs{\tau'} \leq 2$.
For all $j \in \bN$, extend $\eta_j^0$ to $D_2 \times \bR$ by setting
$\eta_j(x,X^{n+1}) = \eta_j^0(x) \tau(X^{n+1}/A)$ at all $X = (x,X^{n+1})
\in D_2 \times \bR$.

To obtain the last two properties, let $\delta > 0$ be given, and $d_G$
be the unsigned distance function to $G \cap D_2 \times \bR$. For any $j$
we may replace $\eta_j$ with $\bar{\eta}_j$, defined by
$\bar{\eta}_{j,\delta}(X) =  \eta_j(X)\tau(d_G(X)/\delta)$ at all $X \in D_2 \times \bR$.
This additionally has $\spt \bar{\eta}_{j,\delta}
\subset (G)_{2 \delta} \cap D_2 \times \bR$.
This function inherits most properties from $\eta_j$, but it is only Lipschitz
regular because of $d_G$. This however can be easily remedied by a standard mollification
argument, taking care to choose the mollification parameter small enough in terms
of $r_j,\delta > 0$.
Finally, we may pick any sequence $s_j \to 0$ and apply the construction
above with $\delta = s_j / 2$, letting $\bar{\eta_j} = \bar{\eta}_{j,s_j/2}$
to conclude.
\end{proof}

Let us conclude by pointing out the following consequence
of~\cite{KrumWic_FinePropsMinGraphs}, obtained by combining it with
the results of~\cite{Wic_MultTwoAllard}, see Theorem~\ref{thm_wic_mult_two_allard}.
\begin{cor}
Let $V \in \IV_n(B_1)$ be a stationary varifold with stable regular part.
Then $\calB(V) \cap \{ \Theta(\norm{V},\cdot) \leq 2 \}$ is countably
$n-2$--rectifiable.
\end{cor}

\subsection{Immersion away from the branch set}
\label{subsec_global_immersion}

Let $\Omega \subset \bR^n$ be a connected open set, and let 
$u \in C^{1,\alpha}(\Omega;\calA_2)$
be so that $\abs{G} = \abs{\graph u} \in \IV_n(\Omega \times \bR)$
is a stationary varifold.
In this section we construct a smooth $n$-dimensional manifold $\Gamma$ and a minimal
immersion $\iota: \Gamma \to \bR^{n+1}$ with image
$\iota(\Gamma) = \reg G \cup \calC(G) = G \setminus \calB(G)$
using a standard gluing construction.

Let $ (U_\alpha \mid \alpha \in A)$ be an open cover of $\Omega \setminus \calB_u$,
chosen so that every $U_\alpha$ is simply connected.
For every open set $U_\alpha$ in this collection we can make a
smooth selection $u_{\alpha,1}, u_{\alpha,2} \in C^1(U_\alpha)$ so that
$u = \{ u_{\alpha,1},u_{\alpha,2} \}$ in $U_\alpha$.
Write $G_{\alpha,1} = \graph u_{\alpha,1}$ and $G_{\alpha,2} = \graph u_{\alpha,2}$,
so that accordingly $G \cap U_\alpha \times \bR	= G_{\alpha,1} \cup G_{\alpha,2}.$
We then consider the disjoint union of these sets
$(G_{\alpha,i} \mid \alpha \in A, i  = 1,2)$, each of which is endowed with
an obvious map $\iota_{\alpha,i}: G_{\alpha,i} \to \bR^{n+1}$, which
may also be composed with the projection $P_0: \bR^{n+1} \to \bR^n \times \{ 0 \}$
to obtain bijections $P_0 \circ \iota_{\alpha,i}: G_{\alpha,i} \to U_{\alpha,i}$.
We glue these together using the equivalence relation $\sim$ defined as follows.
\begin{quotation}
	\emph{Two points $X \in G_{\alpha,i}$ and $Y \in G_{\beta,j}$ are equivalent
	if $P_0 \circ \iota_{\alpha,i}(X) = P_0 \circ \iota_{\beta,j}(Y)$ and there is
	a neighbourhood of this point where $u_{\alpha,i}$ and
	$u_{\beta,j}$ coincide.}
\end{quotation}

Given this we simply set
\begin{equation}
	\Gamma = \bigsqcup_{\substack{\alpha \in A \\ i = 1,2}} G_{\alpha,i} / \sim.
\end{equation}
Write $p$ for the projection $\sqcup G_{\alpha,i} \to \Gamma$.
By construction the map $\sqcup \iota_{\alpha,i}: \sqcup U_{\alpha,i} \to \bR^{n+1}$
passes to the quotient by $\sim$,  thus defining a map
$\iota: \Gamma \to \bR^{n+1}.$
This map has the property that for every set $U_{\alpha,i}$
and all $x \in U_{\alpha,i}$, $\iota \circ p(x) = \iota_{\alpha,i}(x).$
Then it is not hard to see that $\Gamma$ is a smooth $n$-dimensional manifold,
with charts $\{  (p(U_{\alpha,i}),P_0 \circ \iota) \mid \alpha \in A , i = 1,2 \}$.

\begin{lem} \leavevmode
\label{lem_properties_global_immersion}
\begin{enumerate}[label = (\roman*),font=\upshape]
\item The map $\iota$ is a smooth immersion, injective
	away from $\iota^{-1}(\calC(G))$.
\item The immersion can be oriented by the upward unit normal $\nu$.
\item \label{item_proper_immersion}
	The map $\iota$ is proper into $\bR^{n+1} \setminus \calB(G)$,
	but not into $\bR^{n+1}$ unless $\calB(G) = \emptyset$.
\item \label{item_glued_manifold_connected}
	The manifold $\Gamma$ is connected unless $G$ is the union of two
	single-valued graphs.
\end{enumerate}
\end{lem}
\begin{proof}
The first two properties follow by construction.

\ref{item_proper_immersion} Let $K \subset \bR^{n+1} \setminus \calB(G)$ be a compact
set, and consider a sequence of points $(X_j \mid j \in \bN)$ in $\iota^{-1}(K \cap G)$.
Write $Y_j = \iota(X_j)$ for all $j$, and extract a convergent subsequence from this,
with limit say $Y_{j'} \to Y \in \reg G \cup \calC(G)$.
If $Y \in \reg G$ then there is $\rho > 0$ so that the restriction of $\iota$
to $\iota^{-1}(B_\rho(Y) \cap G)$ is a homeomorphism onto $B_\rho(Y) \cap G$,
from whence the property follows.
If instead $Y \in \calC(G)$ then we can decompose $\iota^{-1}(B_\rho(Y) \cap G)
= W_1 \cup W_2$ into two disjoint open sets, so that the restriction of
$\iota$ to either of them is again a homeomorphism. As one of $W_1,W_2$ contains
infinitely many terms in the sequence, the conclusion follows.

\ref{item_glued_manifold_connected}
The map $P_0 \circ \iota:\Gamma \to G$ is a double cover, so 
$\Gamma$ has at most two connected components. When $\Gamma$ is disconnected,
then $\Gamma = \Gamma_1 \cup \Gamma_2$ and the restriction of $P_0 \circ \iota$ to either
of them is a homeomorphism, and their images are two single-valued graphs.
\end{proof}

\subsection{Stability of two-valued minimal graphs}
\label{subsec_stability_of_two_valued_graphs}

Let $i: \Gamma \to \bR^{n+1}$ be the immersion constructed in the previous section,
which maps onto $\reg G \cup \calC(G)$. This admits the positive Jacobi field 
$\langle \nu , e_{n+1} \rangle$, which implies the stability of the immersion
via the following standard argument.
Pick any non-negative test function $\varphi \neq 0 \in C_c^2(\Gamma)$.
On its support $\langle \nu , e_{n+1} \rangle$ is bounded below away from zero,
and thus $T = \max \{ t \in \bR \mid t \varphi \leq \langle \nu , e_{n+1} \rangle \}$
is positive.
By construction $\langle \nu , e_{n+1} \rangle - T \varphi  \geq 0$,
and it has a zero but does not vanish identically. If there were a less
regular test function $\varphi \in C_c^1(\Gamma)$ with
$\int_\Gamma \abs{A_\Gamma}^2 \varphi^2 > \int_\Gamma \abs{\nabla_\Gamma \varphi}^2$
then we could mollify this and follow the same argument as above;
hence the stability of the regular part of $G$ is established.

If $\varphi$ is any function in $C_c^1(\bR^{n+1} \setminus \calB(G))$, then
its pullback $\phi = \varphi \circ \iota$ by the immersion is compactly supported
on $\Gamma$ (because $\iota$ is proper away from the branch set).
If we equate the integrals on $\Gamma$ with ones on $G$ we obtain that
$\int_{\reg G} \abs{A_G}^2 \varphi^2 \leq \int_{\reg G} \abs{\nabla_G \varphi}^2$
for all $\varphi \in C_c^1(\bR^{n+1} \setminus \calB(G))$.
Extending this to arbitrary test functions requires a capacity argument.

\begin{lem}
	\label{lem_graph_ambient_stability}
Let $U \subset \bR^{n+1}$ be an open set, and $G \subset U$
be a two-valued minimal graph.
Then $G$ is ambient stable: for all $\varphi \in C_c^1(U)$, 
\begin{equation}
	\tag{$S_G$}
	\label{eq_stab_ineq}
	\int_{\reg G  \cap U } \abs{A_G}^2 \varphi^2
	\leq \int_{\reg G \cap U} \abs{\nabla_G \varphi}^2.
\end{equation}
\end{lem}
\begin{proof}
Let $\varphi \in C_c^1(U)$ be a test function with $\spt \varphi \cap \calB(G) \neq
\emptyset$, and $(\eta_j \mid j \in \bN)$ be a sequence of functions in $C_c^1(U)$
with properties analogous to those listed in Corollary~\ref{cor_cutoff_sequence_cpcty_2_zero}.
Then $(1 - \eta_j) \varphi \in C_c^1(U \setminus \calB(G))$ and
thus~\eqref{eq_stab_ineq} holds with this test function,
\begin{equation}
	\label{eq_stab_ineq_with_cutoff_branch_set}
	\int_{\reg G \cap U} \abs{A_G}^2 (1 - \eta_j)^2 \varphi^2
	\leq \int_{\reg G \cap U} \abs{\nabla_G  \{(1-\eta_j) \varphi \}}^2.
\end{equation}

The right-hand side can be bounded uniformly in $j$ because
\begin{equation}
	\int_{\reg G \cap U} \abs{ \nabla_G \{ (1 - \eta_j) \varphi \} }^2
	\leq 2 \int_{\reg G \cap U} \abs{\nabla_G \eta_j}^2 \varphi^2
	+ (1-\eta_j)^2 \abs{\nabla_G \varphi}^2
\end{equation}
and as $j \to \infty$ we can separately estimate
\begin{equation}
	\label{eq_grad_to_zero_capa_proof_stability}
	\int_{\reg G \cap U} \abs{\nabla_G \eta_j}^2 \varphi^2
	\to 0
\end{equation}
and by dominated convergence
\begin{equation}
	\label{eq_dom_convergence_proof_stability}
	\int_{\reg G \cap U} (1-\eta_j)^2 \abs{\nabla_G \varphi}^2
	\to \int_{\reg G \cap U} \abs{\nabla_G \varphi}^2.
\end{equation}

In fact we can compute the bounding integral in~\eqref{eq_stab_ineq_with_cutoff_branch_set}
more precisely and show that the cross-term also has
\begin{equation}
	\label{eq_convergence_crossterm_proof_stability}
	2 \int_{\reg G \cap U} (1 -\eta_j) \varphi
	\langle \nabla_G (1-\eta_j) , \nabla_G \varphi \rangle 
	\to 0
	\text{ as $j \to \infty$.}
\end{equation}

On the left-hand side of~\eqref{eq_stab_ineq_with_cutoff_branch_set} we may
pass to the limit by Fatou's lemma, so that letting $j \to \infty$ we obtain
the desired inequality
\begin{equation}
	\int_{\reg G \cap U} \abs{A_G}^2 \varphi^2
	\leq \int_{\reg G \cap U} \abs{\nabla_G \varphi}^2.
	\qedhere
\end{equation}
\end{proof}

Using the results of~\cite{Hutchinson86}---see Proposition~\ref{prop_hutchinson_stability}---%
we have the following corollary for sequences of two-valued minimal graphs.

\begin{cor}
\label{cor_stability_preserved_two_valued_graphs}
Let $(u_j \mid j \in \bN)$ be a sequence of two-valued minimal graphs
$u_j \in C^{1,\alpha}(D_2;\calA_2)$,
and suppose that their graphs converge weakly in the varifold topology
to a limit varifold $V \in \IV_n(D_2 \times \bR)$,
\begin{equation}
	\abs{G_j} = \abs{\graph u_j} \to V
	\text{ as $j \to \infty$}.
\end{equation}
Then $V$ is stationary and ambient stable.
\end{cor}

\section{Area estimates for two-valued minimal graphs}
\label{sec_area_estimates}

\subsection{Area bounds for two-valued minimal graphs}
Here we extend the classical area estimates, which are
well-known for single-valued minimal graphs, to two-valued minimal graphs
by adapting the arguments presented in~\cite[Ch.~16]{GilbargTrudinger98}.

\begin{prop}
\label{prop_area_estimate}
Let $\alpha \in (0,1/2)$, and let 
$u \in C^{1,\alpha}(D_{2r};\calA_2)$ be a two-valued minimal graph.
Then 
\begin{equation}
	\calH^n(G \cap B_r) 
	\leq 2 \omega_n ( 1 + n ) r^n.
\end{equation}
\end{prop}
\begin{proof}
Let $\eta \in C_c^1(D_{2r})$ be a test function with $\eta \equiv 1$
on $D_r$ and $\abs{D \eta} \leq 2/r$.
Next define $\Phi(x,z) = \eta(x) z_r$, where 
\begin{equation}
	\label{eq_def_zr}
z_r =
\begin{cases}
	r &\text{if } z > r \\
	z &\text{if } -r \leq z \leq r \\
	-r &\text{if } z < -r.
\end{cases}
\end{equation}

Then by the (two-valued) chain rule we get 
\begin{equation}
D(\Phi(x,u))
= D \eta(x) u_r + \{ D_z \Phi(x,u_i(x)) Du_i(x) \}.
\end{equation}
This is clearly well-defined at any $x \in D_{2r} \setminus \calB_u$;
in fact the same is true at branch points $x \in \calB_u$ because the
two components of $Du(x) \in \calA_2(\bL(\bR^n;\bR))$ agree there.
We evaluate this expression to be
\begin{equation}
D(\Phi(x,u)) =  D \eta(x) u_r(x) + \eta(x)  \indic_{\abs{u} < r } Du(x).
\end{equation}
As $\eta$ is compactly supported inside $D_{2r}$ we get
$\int_{D_{2r} \setminus \calB_u} \abs{D ( \Phi(x,u))} < +\infty$.
Then Proposition~\ref{prop_test_function} justifies
\begin{equation}
	\label{eq_int_identity_for_area_estimate}
	\int_{D_{2r}} u_r \langle T(Du), D \eta \rangle
	+ \eta \langle T(Du), Du \rangle \indic_{\abs{u} < r} = 0,
\end{equation}
so that
\begin{equation}
	\label{eq_area_estimate_inequality_perimeter_type}
	\int_{D_r} \indic_{\abs{u} < r} \frac{\abs{Du}^2}{v}
	\leq 2r \int_{D_{2r}} \abs{D \eta}.
\end{equation}

The area of the graph is bounded by the integral $\calH^n(G \cap B_r) \leq \int
v \indic_{\abs{u} < r}$, which we split as
\begin{equation}
	\label{eq_split_area_two_integrals}
	\int_{D_r} \frac{1}{v} \indic_{\abs{u} < r}
	+ \int_{D_r} \frac{\abs{Du}^2}{v} \indic_{\abs{u} < r}
	\leq 2 (\calH^n(D_r) + r \calH^{n-1}(\bdary D_r)),
\end{equation}
whence we conclude by noting $\calH^{n-1}(\bdary D_r) = n \omega_n r^{n-1}$.
\end{proof}
A similar argument yields area bounds in the cylinder above the disc $D_r$. 
A detailed proof in the single-valued case is given in~\cite[Ch.16]{GilbargTrudinger98},
to adapt it to two-valued graphs one makes the same modifications as
above.
These will be used in the proof of the gradient
estimates (see the proof of Lemma~\ref{lem_int_grad_bounds_form_in_small_ball}).
\begin{lem}\label{lem_cylindrical_area_estimates}
Let $\alpha \in (0,1)$ and $u \in C^{1,\alpha}(D_{2r};\calA_2)$ be a two-valued minimal
graph. Then
\begin{equation}
\calH^n(G \cap D_r \times \bR) 
\leq 2 \omega_n  r^n(1 + n r^{-1} \sup_{D_{2r}} \norm{u}) .
\end{equation}
\end{lem}

\subsection{Improved estimates for convergent sequences}

\label{subsec_improved_area_estimates}

These estimates can be significantly improved when one considers a sequence
of two-valued minimal graphs that converge weakly to a vertical varifold,
that is of the form $V = V_0 \times \bR e_{n+1} \in \IV_n(D_2 \times \bR)$.
Here we concentrate on the case where additionally $V$ is known to be supported
in a union of planes, that is $V = \bP = \sum_j m_j \abs{\Pi_j^0} \times \bR e_{n+1}$,
where $\Pi_j = \Pi_j^0 \times \bR e_{n+1} \in \Gr(n,n+1)$.
(However we point out that virtually identical bounds are possible without this
hypothesis, at the price of slightly more involved computations.)

\begin{prop}
\label{prop_improved_estimates_sum_of_planes}
Let $u_j \in C^{1,\alpha}(D_2;\calA_2)$ be a sequence
of two-valued minimal graphs with $\abs{G_j} \to
\bP = \sum_j m_j \abs{\Pi_j^0} \times \bR e_{n+1}$.
Then $\sum_j m_j \leq \lfloor n \omega_n / \omega_{n-1} \rfloor$.
\end{prop}
\begin{proof}
For all $j$, $\calH^n(G_j \cap B_1) \leq
\int_{D_1} \frac{1}{v_j} \indic_{\abs{u_j} < 1} + \int_{D_1} \frac{\abs{Du_j}^2}{v_j}
\indic_{\abs{u_j} < 1}$. (In fact this bounds the area of the graph in
$D_1 \times (-1,1)$.)
A quick computation reveals that the varifold convergence $\abs{G_j} \to \bP$
is strong enough to ensure that $\int_{D_1} \frac{1}{v_j} \indic_{\abs{u_j} < 1} \to 0$
as $j \to \infty$, so that we may focus on the second term.
Referring back to~\eqref{eq_area_estimate_inequality_perimeter_type} and~%
\eqref{eq_split_area_two_integrals} we find that this is bounded like
$\int_{D_1} \frac{\abs{Du_j}^2}{v_j} \indic_{\abs{u_j} < 1} 
\leq 2 \calH^{n-1}(\bdary D_1) = 2 n \omega_n$.
Hence given any $\eps > 0$ there is $J(\eps) \in \bN$ so that
$\calH^n(G_j \cap B_1) \leq 2 n \omega_n + \eps$ for all $j \geq J(\eps)$,
and in fact by the remark above
$\calH^n(G_j \cap D_1 \times (-1,1)) \leq 2n  \omega_n + \eps$.
Letting $j \to \infty$ we find that
$\norm{\bP}(D_1 \times (-1,1)) \leq 2n \omega_n$,
whence after replacing the right-hand side by $\sum_j 2 m_j \omega_{n-1}$
we find $\sum_j m_j \leq n \omega_n / \omega_{n-1}$.
The conclusion follows after taking integer values. 
\end{proof}

\begin{table}
\caption{Improved area estimates obtained in
	Proposition~\ref{prop_improved_estimates_sum_of_planes},
for dimensions up to seven.}
\begin{tabularx}{0.8\textwidth}{ c  *{6}{Y} } 
\toprule
$n$ &  2 & 3 & 4 & 5 & 6 & 7 \\ \midrule
$n \omega_n / \omega_{n-1}$  & $\pi$ & $4$ &  $3\pi/2$ & $16/3$ & $15 \pi / 8$ & $32 / 5$ \\
$\lfloor n \omega_n / \omega_{n-1} \rfloor$  & 3 & 4 & 4 & 5 & 5 & 6 \\
\bottomrule
\end{tabularx}
\label{table_values_for_the_improved_estimates}
\end{table}

We obtain a small additional improvement of these estimates, which will eventually
turn out to be essential. Assume for now the validity of its conclusion. In the
special case $n = 3$, using the values from
Table~\ref{table_values_for_the_improved_estimates}
 one finds that $\sum_j m_j \leq (1 - \delta) 3 \omega_3 / \omega_2
= 4 ( 1 - \delta)$, whence after taking integer values $\sum_j m_j \leq 3$.

\begin{cor}
\label{cor_qualitative_estimates}
Let $\alpha \in (0,1)$, and $u_j \in C^{1,\alpha}(D_2;\calA_2)$ be a
sequence of two-valued minimal graphs with
$\abs{G_j} \to \bP = \sum_j m_j \abs{\Pi_j^0} \times \bR e_{n+1}$ as $j \to \infty$.
Then there is $\delta = \delta(\bP) \in (0,1)$ so that
$\sum_j m_j \leq \lfloor (1 - \delta) n \omega_n / \omega_{n-1} \rfloor$.
\end{cor}
\begin{proof}
As the boundary of the disc is mean-convex, one may use an elementary construction
to find an open set $U \subset D_1$ with $\spt \norm{V_0} \cap D_1 \subset U$ and
$\Per(U) < \Per (D_1)$.
Given arbitrarily small $\delta > 0$, there is $\eta \in C_c^1(D_2)$ with
$\eta \equiv 1$ on $U$ and $\int_{D_2} \abs{D \eta} \leq (1 + \delta) \Per(U)$. 
Now first take $\tau > 0$ small enough that $(U)_\tau \subset \{ \eta = 1 \}$,
and next take the index $j \geq J(\tau)$ large enough that 
$G_j \cap D_1 \times (-1,1) \subset (U)_\tau  \times (-1,1)$.
This being fulfilled one proceeds as in the proof of the area bounds,
namely by estimating 
$\int_{D_1} \indic_{\abs{u_j} < 1} \frac{\abs{Du_j}^2}{v_j}
\leq 2 \int_{D_2} \abs{D \eta} \leq 2(1 + \delta) \Per(U)$.
Now choose $\delta = \delta(\bP) > 0$ small enough that $(1 + \delta)\Per(U) 
\leq (1 - \delta) \Per (D_1)$, and use the same value for $\delta$ in  the above.
Letting $j \to \infty$ we find $\norm{\bP}(D_1 \times (-1,1))
\leq  2 ( 1 - \delta) n \omega_n$, whence
$\sum_j m_j \leq \lfloor (1 - \delta) n \omega_n / \omega_{n-1} \rfloor$.
\end{proof}

\section{Gradient estimates for two-valued minimal graphs}
\label{sec_grad_estimates}

Let $\alpha \in (0,1)$ and the dimension $n \geq 1$ be arbitrary.
Let $u \in C^{1,\alpha}(D_2;\calA_2)$ be a two-valued minimal graph. In this
section we derive an interior gradient estimate analogous to the classical
estimates for smooth, single-valued graphs. These can be found for example
in Section~16.2 of~\cite{GilbargTrudinger98}, which we also follow for the
structure of the argument.
These gradient bounds stem from integral estimates for the function $w$,
defined on $\reg G$ by the expression~\eqref{eq_def_w}. To ensure the validity
of these in the presence of branch points (which are absent in the single-valued
case), we rely on the fine properties of the branch set proved by
\cite{KrumWic_FinePropsMinGraphs}, specifically that it has zero $2$-capacity.
The main result in this section is the following.

\begin{lem}
\label{lem_interior_gradient_bounds}
Let $\alpha \in (0,1)$ and let
$u \in C^{1,\alpha}(D_2;\calA_2)$ be a two-valued minimal graph.
There is a constant $C = C(n) > 0$ so that
\begin{equation}
	\max_{D_1} \norm{Du} \leq C \exp(C \sup_{D_2} \norm{u}).
\end{equation}
\end{lem}

We prove the equivalent version below, for discs of arbitrary radius $r > 0$.

\begin{lem}
\label{lem_int_grad_bounds_form_in_small_ball}

Let $u \in C^{1,\alpha}(D_{3r};\calA_2)$ be a two-valued minimal graph.
Then there is a constant $C = C(n) > 0$ so that
\begin{equation}
	\max \{\abs{Du_1(0)},\abs{Du_2(0)} \}
	\leq  C \exp(C \max_{D_{2r}}\norm{u} / r).
\end{equation}
\end{lem}

\subsection{Integral estimates and a mean-value inequality for $w$}

Define a function $w$ at all points  $X \in \reg G \cap D_2 \times \bR$ by
\begin{equation}
	\label{eq_def_w}
	w(X) = \log v(X) = - \log \langle \nu(X) , e_{n+1} \rangle,
\end{equation}
where $\nu(X)$ is the upward-pointing unit normal to $\reg G$ at $X$
and $v(X) = \langle \nu(X), e_{n+1} \rangle^{-1}$.

\begin{lem}
	\label{lem_weak_PDE_for_w}
For all compact $K \subset D_2 \times \bR$,
\begin{equation}
	\label{eq_local_bounds_for_w_and_nabla_w}
	\sup_{K \cap \reg G} w
	+ \int_{K \cap \reg G} \abs{\nabla_G w}^2 < +\infty
\end{equation}
and $w$ satisfies $\Delta_G w = \abs{\nabla_G w}^2 + \abs{A_G}^2$
weakly in the sense that for all $\varphi \in C_c^1(D_2 \times \bR)$,
\begin{equation}
	\label{eq_integral_identity_PDE_for_w}
	- \int_{\reg G \cap D_2 \times \bR} \langle \nabla_G w , \nabla_G \varphi \rangle
	= \int_{\reg G \cap D_2 \times \bR} (\abs{\nabla_G w}^2 + \abs{A_G}^2) \varphi.
\end{equation}
\end{lem}
\begin{proof}
Let $K \subset D_2 \times \bR$ be an arbitrary compact subset.
As the gradient of $u$ is locally bounded, we also find that
$\sup_{\reg G \cap K} w <  +\infty$.
To prove $\int_{\reg G \cap K} \abs{\nabla_G w}^2 < +\infty$
we use the minimal immersion $\iota: \Gamma \to \bR^{n+1}$
with image $\iota(\Gamma) = G \setminus \calB(G) = \reg G \cup \calC(G)$
that we constructed in Section~\ref{subsec_global_immersion}.

Pull back $w$ to $w \circ \iota \in C^\infty(\iota^{-1}(\reg G))$,
which we subsequently extend across $\iota^{-1}(\calC(G))$ to yield a function
in $C^\infty(\Gamma)$, still denoted by $w \circ \iota$.
This function satisfies the PDE
\begin{equation}
	\label{eq_PDE_on_globally_immersed_manifold}
	\Delta_\Gamma(w \circ \iota)
	- \abs{\nabla_\Gamma (w \circ \iota)}^2 - \abs{A_\Gamma}^2 = 0
\end{equation}
pointwise (and thus also weakly) on $\Gamma$.
From this we may deduce the bound 
\begin{equation}
	\label{eq_local_grad_bound}
	\int_\Gamma \abs{\nabla_\Gamma (w \circ \iota)}^2 \phi^2
	\leq 4 \int_\Gamma \abs{\nabla_\Gamma \phi}^2,
\end{equation}
valid for all $\phi \in C_c^1(\Gamma)$.

Indeed if we ignore the curvature term in
\eqref{eq_PDE_on_globally_immersed_manifold}---as we may because it
has a favourable sign---and integrate against an arbitrary $\phi \in C_c^1(\Gamma)$
we see that
\begin{equation}
	\int_\Gamma \abs{\nabla_\Gamma (w \circ \iota)}^2 \phi
	\leq - \int_\Gamma \langle \nabla_\Gamma (w \circ \iota), \nabla_\Gamma \phi \rangle.
\end{equation}
If instead we use $\phi^2$ as a test function, then we find
\begin{multline}
	\int_\Gamma \abs{\nabla_\Gamma (w \circ \iota)}^2 \phi^2  \\
	\leq -2 \int_\Gamma \phi 
	\langle \nabla_\Gamma (w \circ \iota), \nabla_\Gamma \phi \rangle
	\leq 2 \big( \int_\Gamma \phi^2 \abs{\nabla_\Gamma (w \circ \iota)}^2 \big)^{1/2}
	\big( \int_\Gamma \abs{\nabla_\Gamma \phi}^2 \big)^{1/2}.
\end{multline}
Unless $\int_\Gamma \abs{\nabla_\Gamma ( w \circ \iota)}^2 \phi^2 = 0$ we may divide
both sides by its square root, yielding~\eqref{eq_local_grad_bound}.
(If the integral vanishes then the inequality is trivially satisfied.)

In particular, if we take $\varphi \in C_c^1(D_2 \times \bR \setminus \calB(G))$ and
let $\phi = \varphi \circ \iota$ and translate \eqref{eq_local_grad_bound} 
to the graph, we obtain
\begin{equation}
	\label{eq_local_grad_bound_on_G}
	\int_{\reg G \cap D_2 \times \bR} \abs{\nabla_G w}^2 \varphi^2
	\leq 4 \int_{\reg G \cap D_2 \times \bR} \abs{\nabla_G \varphi}^2.
\end{equation}
To extend this through the branch point singularities of $G$, we once
again employ the sequence $(\eta_j \mid j \in \bN)$ with properties as
described in Lemma~\ref{lem_EG_cutoff_sequence_two}.

Then proceed as in the proof of Lemma~\ref{lem_graph_ambient_stability},
namely take $\varphi \in C_c^1(D_2 \times \bR)$ and substitute the test function
$\varphi (1 - \eta_j)$ into \eqref{eq_local_grad_bound_on_G}, yielding
\begin{multline}
	\int_{\reg G \cap D_2 \times \bR} \abs{\nabla_G w}^2 \varphi^2 (1 - \eta_j)^2 \\
	\leq 4 \int_{\reg G \cap D_2 \times \bR} \varphi^2 \abs{\nabla_G \eta_j}^2
	+ 2 \varphi (1-\eta_j) \langle \nabla_G \varphi , \nabla_G (1 -  \eta_j) \rangle
	+ \abs{\nabla_G \varphi}^2 (1 - \eta_j)^2.
\end{multline} 
The terms on the right-hand side are identical to those on the right-hand
side of~\eqref{eq_stab_ineq_with_cutoff_branch_set}, so that we may justify
passing to the limit $j \to \infty$ as in
\eqref{eq_grad_to_zero_capa_proof_stability},\eqref{eq_dom_convergence_proof_stability}
and \eqref{eq_convergence_crossterm_proof_stability}.
We thus obtain
\begin{equation}
	\label{eq_gradient_estimate_for_nabla_w}
	\int_{\reg G \cap D_2 \times \bR} \abs{\nabla_G w}^2 \varphi^2 
	\leq 4\int_{\reg G \cap D_2 \times \bR} \abs{\nabla_G \varphi}^2,
\end{equation}
justifying passing to the limit on the left-hand side 
by an application Fatou's lemma, again as in the proof of
Lemma~\ref{lem_graph_ambient_stability}.
This justifies \eqref{eq_local_bounds_for_w_and_nabla_w}.

To show that $w$ is a weak solution of the PDE $\Delta_G w 
= \abs{\nabla_G w}^2 + \abs{A_G}^2$ we proceed in much the same
way. The integral identity \eqref{eq_integral_identity_PDE_for_w}
is obtained for test functions $\varphi \in C_c^1(D_2 \times \bR \setminus \calB(G))$
by working with the immersed $\Gamma$ instead, as above.
Using the same sequence $(\eta_j \mid j \in \bN)$ of functions, we then
obtain 
\begin{equation}
	- \int_{\reg G \cap D_2 \times \bR}
	\langle \nabla_G w, \nabla \{ (1 - \eta_j) \varphi \} \rangle
	= \int_{\reg G \cap D_2 \times \bR}
	(\abs{\nabla_G w}^2 + \abs{A_G}^2) \varphi (1 - \eta_j).
\end{equation}
For the right-hand side of the identity, we may let $j \to \infty$
by dominated convergence, which we can justify using our previously established
local $L^2$-bounds for $A_G$ and $\abs{\nabla_G w}$ from \eqref{eq_stab_ineq} and 
\eqref{eq_local_bounds_for_w_and_nabla_w} respectively.

For the left-hand side of the identity, we expand
\begin{multline}
	- \int_{\reg G \cap D_2 \times \bR}
	\langle \nabla_G w, \nabla \{ (1 - \eta_j) \varphi \} \\
	= \int_{\reg G \cap D_2 \times \bR}
	\langle \nabla_G w , \nabla_G \eta_j \rangle \varphi
	- \langle \nabla_G w, \nabla_G \varphi \rangle (1 - \eta_j),
\end{multline}
and repeat calculations akin to those in the first part of the proof, noting that
the bounds of \eqref{eq_local_bounds_for_w_and_nabla_w} justify both the limit
\begin{multline}
	\int_{\reg G \cap D_2 \times \bR}
	\langle \nabla_G w , \nabla_G \eta_j \rangle \varphi \\
	\leq \sup_{D_2 \times \bR}
	\abs{\varphi} \big(\int_{\reg G \cap D_2 \times \bR} \abs{\nabla_G w}^2\big)^{1/2}
	\big(\int_{\reg G \cap D_2 \times \bR} \abs{\nabla_G \eta_j}^2\big)^{1/2} 
	\to 0 
	\text{ as $j \to \infty$}
\end{multline}
and the application of dominated convergence to deduce that
\begin{equation}
	\int_{\reg G \cap D_2 \times \bR} \langle \nabla_G w, \nabla_G \varphi \rangle
	(1 - \eta_j) \to \int_{\reg G \cap D_2 \times \bR}
	\langle \nabla_G w, \nabla_G \varphi \rangle
	\text{ as $j \to \infty$}.
\end{equation}
Thus we have derived the identity~\eqref{eq_integral_identity_PDE_for_w},
which concludes the proof of the lemma.
\end{proof}

\begin{rem}
\label{rem_alternative_derivation_bounds_for_w}
We could have derived the integral estimate for $\abs{\nabla_G w}$ 
in~\eqref{eq_local_bounds_for_w_and_nabla_w} differently,
using the local curvature bounds that follow from
the stability inequality~\eqref{eq_stab_ineq}.
Indeed at all regular points $\langle \nu , e_{n+1} \rangle > 0$, hence
\begin{equation}
	\label{eq_ineq_alternative_sobolev_bounds}
	\abs{\nabla_G w}^2 \leq \abs{A_G}^2 (\langle \nu , e_{n+1} \rangle^{-2} - 1)
	\text{ on $\reg G$},
\end{equation}
where we used the fact that $\abs{\nabla_G \langle \nu , e_{n+1} \rangle}^2
\leq \abs{A_G}^2 (1 - \langle \nu , e_{n+1} \rangle^2)$.
Given any compact subset $K \subset D_2 \times \bR$,
the term $\langle \nu,e_{n+1} \rangle$ is bounded below, say
$\langle \nu,e_{n+1} \rangle \geq \delta_K$ on $\reg G \cap K$.
Then integrating~\eqref{eq_ineq_alternative_sobolev_bounds}
we obtain
\begin{equation}
	\label{eq_succinct_estimate_nabla_w}
	\int_{\reg G \cap K} \abs{\nabla_G w}^2
	\leq (\delta_K^{-2} - 1) \int_{\reg G \cap K} \abs{A_G}^2,
\end{equation}
whence we get $\int_{\reg G \cap K} \abs{\nabla_G w}^2 \leq C_K$ for
some $C_K > 0$ using~\eqref{eq_stab_ineq}.

We chose to include the longer derivation in our proof, as it yields the
more precise $\int_{\reg G \cap D_2 \times \bR} \abs{\nabla_G w}^2 \varphi
\leq 4 \int_{\reg G \cap D_2 \times \bR} \abs{\nabla_G \varphi}^2$,
valid for all $\varphi \in C_c^1(D_2 \times \bR)$.
We will also use this in the derivation of the interior gradient estimates
(see the proof of Lemma~\ref{lem_int_grad_bounds_form_in_small_ball} below).
Compare this with the less useful inequality derived by arguing as above,
essentially combining \eqref{eq_succinct_estimate_nabla_w} with the stability
inequality~\eqref{eq_stab_ineq} to yield
$\int_{\reg G \cap D_2 \times \bR} \abs{\nabla_G w}^2 \varphi^2
\leq (\delta_K^{-2} - 1) \int_{\reg G \cap D_2 \times \bR} \abs{\nabla_G \varphi}^2$,
where $\varphi \in C_c^1(D_2 \times \bR)$ and $\spt \varphi \subset K$.
\end{rem}

Let us quickly comment on a slightly subtle point.
The function $w = - \log \langle \nu , e_{n+1} \rangle$ is only defined on the
regular part $\reg G \cap D_2 \times \bR$, and cannot be extended continuously
across $\calC(G) \cap D_2 \times \bR$.
However after pulling $w$ back via the immersion $\iota: \Gamma \to G \setminus
\calB(G)$ we obtain a function $w \circ \iota$ which we can extend smoothly
through $\iota^{-1}(\calC(G))$. This in turn allowed us to integrate by parts,
yielding formulas which translate to $G$.

This way one obtains the following identity,
valid for all $\varphi \in C_c^2(D_2 \times \bR \setminus \calB(G))$:
\begin{equation}
	\label{eq_partial_integration}
	\int_{\reg G \cap D_2 \times \bR} (\Delta_G \varphi) w
	= - \int_{\reg G \cap D_2 \times \bR} \langle \nabla_G \varphi , \nabla_G w \rangle
	= \int_{\reg G \cap D_2 \times \bR} \varphi \Delta_G w.
\end{equation}
From this, we may verify using a capacity argument
that for $\varphi \in C_c^2(D_2 \times \bR)$,
\begin{equation}
	\label{eq_partial_int_through_branch_set}
	- \int_{\reg G \cap D_2 \times \bR} \langle \nabla_G \varphi, \nabla_G w \rangle
	= \int_{\reg G \cap D_2 \times \bR} \varphi \Delta_G w.
\end{equation}
Let again $(\eta_j \mid j \in \bN)$ be a sequence of functions with properties as described
in Lemma~\ref{lem_EG_cutoff_sequence_two}. By a mollification argument for
example, we may additionally impose that $\eta_j \in C_c^\infty(D_2 \times \bR)$ for all $j$.
If $\varphi \in C_c^2(D_2 \times \bR)$ then
$(1 - \eta_j) \varphi \in C_c^2(D_2 \times \bR \setminus \calB(G))$
is a valid test function in~\eqref{eq_partial_integration}.

Focus on the two integrals in~\eqref{eq_partial_int_through_branch_set}.
For the first, we may justify taking the limit 
\begin{equation}
	\int_{\reg G \cap U} \langle \nabla_G \{ (1 - \eta_j)  \varphi \} , \nabla_G w \rangle
	\to \int_{\reg G \cap U} \langle \nabla_G \varphi , \nabla_G w \rangle
	\text{ as $j \to \infty$}
\end{equation}
as usual, whereas for the second note that on $\reg G$,
\begin{equation}
\abs{\varphi (1 - \eta_j) \Delta_G w}
\leq \abs{\varphi} (\abs{\nabla_G w}^2 + \abs{A_G}^2)
\text{ for all $j$}.
\end{equation}
As $K = \spt \varphi$ is compact, by stability
and~\eqref{eq_local_bounds_for_w_and_nabla_w} there is a $C_K > 0$ so that
\begin{equation}
	\int_{\reg G \cap K} \abs{\nabla_G w}^2 + \abs{A_G}^2
	\leq C_K.
\end{equation}
Finally, dominated convergence allows taking the limit
\begin{equation}
	\int_{\reg G \cap U} \varphi ( 1 - \eta_j) \Delta_G w
	\to \int_{\reg G \cap U} \varphi \Delta_G w
	\text{ as $j \to \infty$},
\end{equation}
which confirms the identity~\eqref{eq_partial_int_through_branch_set} claimed above.

By dropping the curvature term from the (weakly satisfied) PDE $\Delta_G w = 
\abs{\nabla_G w}^2 + \abs{A_G}^2$ one sees that $w$ is weakly subharmonic on
$G$ through the branch set. This implies that it satisfies a mean-value
inequality, see Corollary~\ref{cor_mean_value_ineq} below.
Here the difference between the identities expressed in~\eqref{eq_partial_integration}
and~\eqref{eq_partial_int_through_branch_set} becomes significant.
The identity involving the left-most term of~\eqref{eq_partial_integration}
does not naturally extend across the branch set of $G$ via the capacity argument
we just invoked.
This is unfortunate, because it \emph{prima facie} prevents an appeal to the
mean-value inequalities that appear in the literature: see for example those
derived by Michael--Simon~\cite{Michael_Simon_Sobolev} or Simon~\cite[Ch.\ 18]{Simon84}.
For this reason, we give a detailed derivation of the mean-value inequality%
~\eqref{eq_developed_monotonicity_formula} expressed in Corollary~\ref{cor_mean_value_ineq};
for this we essentially use a modification of arguments of Simon~\cite[Ch.\ 17]{Simon84}.

\begin{cor}
	\label{cor_mean_value_ineq}
Let $X = (x,X^{n+1}) \in D_2 \times \bR$.
Then for all $0 < \sigma < \rho < 2 - \abs{x}$
\begin{multline}
	\label{eq_mean_value_ineq}
	\rho^{-n} \int_{\reg G \cap B_\rho(X)} w
	- \sigma^{-n} \int_{\reg G \cap B_\sigma(X)} w \\
	\geq \int_{\reg G \cap B_\rho(X) \setminus \clos{B}_\sigma(X)}
	w \abs{D^\perp r}^2 r^{-n} \geq 0.
\end{multline}
\end{cor}
\begin{proof}
To simplify notation, we may assume without loss of generality that the point
$X$ lies at the origin and $\rho < 2$.
We use a two-parameter family of Lipschitz cutoff functions
$(\gamma_{\delta,s} \mid \delta \in (0,1), s \in (0,\rho))$
constructed by first setting
\begin{equation}
\gamma_\delta(t) = 
\begin{cases}
	1 &\text{ if $t \leq 1 - \delta$}, \\
	(1-t)/\delta &\text{ if $1 - \delta < t < 1$}, \\
	0 & \text{ if $t \geq 1$}
\end{cases}
\end{equation}
and then rescaling $\gamma_{\delta,s}(t) = \gamma_\delta(t/s)$ for all $t \in \bR$.
Moreover we write $r = \abs{X}$ and define the radial functions
$\gamma_{\delta,s}(X) = \gamma_{\delta,s}(r)$ for all $X \in D_2 \times \bR$,
which all have $\spt \gamma_{\delta,s} \subset \subset B_\rho \subset D_2 \times \bR$.
Fix $\delta \in (0,1)$ and $s \in (0,\rho)$, with the eventual aim of
letting $\delta$ tend to zero.

Although the vector field $\gamma_{\delta,s}(r) X w$ is not Lipschitz
we can justify its use in the first variation formula using a quick
capacity argument.
Let $(\eta_j \mid j \in \bN)$ be a sequence of cutoff sequences with
properties essentially as described in Corollary~\ref{cor_cutoff_sequence_cpcty_2_zero},
namely $\eta_j \in C_c^1(D_2 \times \bR \cap \reg G)$ with $0 \leq \eta_j \leq 1$
and $\eta_j \equiv 1$ on $(\calB(G))_{r_j} \cap B_\rho$ for some $r_j \to 0$.
Moreover as $j \to \infty$, $\eta_j \to 0$ $\calH^n$-a.e.\ on $\reg G \cap B_\rho$
and $\int_{\reg G \cap B_\rho} \abs{\nabla_G \eta_j}^2 \to 0$.
As these cut out the branch set, the vector field $(1-\eta_j) \gamma_{s,\delta} X w$
is a valid choice in the first variation formula, and yields
$\int_{\reg G \cap B_\rho} \div_G \big( (1-\eta_j) \gamma_{s,\delta} X w \big) = 0$.
We can expand this expression to get
$\abs{ \int (1-\eta_j) \div_G(\gamma_{s,\delta} X w) }
\leq \int \abs{\gamma_{s,\delta} X w} \abs{\nabla_G \eta_j}$.
The right-hand side tends to zero by the Cauchy--Schwarz inequality.
On the left-hand side we justify the convergence
$\int (1-\eta_j) \div_G (\gamma_{s,\delta} X w)
\to \int \div_G (\gamma_{s,\delta} X w)$ by dominated convergence, after noticing
that $\int \abs{ \div_G (\gamma_{s,\delta} X w)} < \infty$.
Hence we have 
\begin{equation}
	\label{eq_first_variation_with_w}
	\int_{\reg G \cap D_2 \times \bR} \div_G (\gamma_{\delta,s} X w) = 0.
\end{equation}
Following the computations in~\cite[p.~83]{Simon84} we find that
$\div_G(\gamma_{\delta,s} X) = n \gamma_{\delta,s}
+ r \gamma'_{\delta,s} (1 - \abs{D^\perp r}^2)$,
and~\eqref{eq_first_variation_with_w} leads to 
\begin{multline}
	\label{eq_developed_monotonicity_formula}
	n \int_{\reg G \cap D_2 \times \bR } \gamma_{\delta,s} w
	+ \int_{\reg G \cap D_2 \times \bR} r \gamma_{\delta,s}' w \\
	= \int_{\reg G \cap D_2 \times \bR} r \gamma_{\delta,s}' w \abs{D^\perp r}^2
	- \int_{\reg G \cap D_2 \times \bR} \gamma_{\delta,s} \langle  \nabla_G w, X \rangle.
\end{multline}

To somewhat abbreviate this integral identity we define the two functions
$I_\delta,J_\delta: (0,\rho) \to \bR$ by setting, for all $s \in (0,\rho)$,
\begin{align}
	I_\delta(s) & = \int_{\reg G \cap D_2 \times \bR} \gamma_{\delta,s} w, \\
	J_\delta(s) & = \int_{\reg G \cap D_2 \times \bR} \gamma_{\delta,s} w \abs{D^\perp r}^2.
\end{align}
Notice that these are both differentiable in $s$ with respective derivatives
$I_\delta'(s) = \int \frac{\partial \gamma_{\delta,s}}{\partial s}  w$ and $J_\delta'(s)
= \int \frac{\partial \gamma_{\delta,s}}{\partial s}   w  \abs{D^\perp r}^2$.
Note $\gamma_{\delta,s}'(r) = 1/s \gamma_\delta'(r/s) = 1/(rs) \gamma_\delta'(r/s)$,
so that $r \gamma_{\delta,s}'(r) = -s \frac{\partial}{\partial s} \gamma_{\delta,s}(r)$.
Thus we can rewrite~\eqref{eq_developed_monotonicity_formula} as
\begin{equation}
	\label{eq_changed_monotonicity_formula}
	n I_\delta(s) 
	-s I_\delta'(s)
	= -s J_\delta'(s)
	- \int_{\reg G \cap D_2 \times \bR} \gamma_{\delta,s} \langle \nabla_G w , X \rangle
	\text{ for all $s \in (0,\rho)$}.
\end{equation}
Multiply this equation by $s^{-n-1}$ and notice that this is
\begin{equation}
	\frac{\diff}{ \diff s}( s^{-n} I_\delta(s))
	= s^{-n} J_\delta'(s) + s^{-n-1} \int_{\reg G \cap D_2 \times \bR} \gamma_{\delta,s} \langle
	\nabla_G w , X \rangle.
\end{equation}
Integrate this identity for $s \in (\sigma,\rho)$ to get
\begin{equation}
	\label{eq_mean_value_first_identity}
	\rho^{-n} I_\delta(\rho) - \sigma^{-n} I_\delta(\sigma)
	= \int_\sigma^\rho s^{-n} J_\delta'(s)
	+ \int_\sigma^\rho s^{-n-1}
	\int_{\reg G \cap D_2 \times \bR} \gamma_{\delta,s} \langle \nabla_G w , X \rangle.
\end{equation}
The eventual aim is to let $\delta \to 0$ in this identity; this will yield%
~\eqref{eq_mean_value_ineq} and conclude the proof.
Before we do this, we separately integrate both integrals on the right-hand side by
parts. For the first, we obtain
\begin{align}
	\label{eq_mean_value_ineq_first_integral}
	\hspace{3em}	\int_\sigma^\rho
	&s^{-n} J_\delta'(s)  = \rho^{-n} J_\delta(\rho) - \sigma^{-n} J_\delta(\sigma)
	+ n \int_\sigma^\rho s^{-n-1} J_\delta(s) \\
	& = \int_{\reg G \cap D_2 \times \bR} w \abs{D^\perp r}^2
	\Big \{
		\rho^{-n} \gamma_{\delta,\rho} - \sigma^{-n} \gamma_{\delta,\sigma}
		+ n \int_{\sigma}^\rho s^{-n-1} \gamma_{\delta,s}
	\Big \}
	\intdiff \calH^n.
\end{align}

For the second integral, first fix $s \in (0,\rho)$ and notice
that $\nabla_G (r^2 - s^2) = 2 X^T$.
The identity $\int_{D_2 \times \bR \cap \reg G}
\div_G (\gamma_{\delta,s} (r^2 - s^2) \nabla_G w) = 0$,
is equivalent to 
\begin{multline}
	2\int_{ \reg G \cap D_2 \times \bR}
	\gamma_{\delta,s} \langle \nabla_G w , X \rangle \\
	= - \int_{ \reg G \cap D_2 \times \bR } (\Delta_G w) \gamma_{\delta,s} (r^2 - s^2)
	- \int_{ \reg G \cap D_2 \times \bR} (r^2 - s^2) \langle \nabla_G \gamma_{\delta,s} ,
	\nabla_G w \rangle.
\end{multline}
The first integral on the right-hand side is equal to
$- \int_{\reg G \cap D_2 \times \bR}( \abs{A_G}^2 + \abs{\nabla_G w}^2)
\gamma_{\delta,s} (r^2-s^2)$, and is non-negative.
From the identity we only retain the inequality
\begin{equation}
	2\int_{ \reg G \cap D_2 \times \bR} \gamma_{\delta,s} \langle \nabla_G w , X \rangle 
	\geq - \int_{\reg G \cap D_2 \times \bR}(r^2 - s^2)
	\langle \nabla_G \gamma_{\delta,s} , \nabla_G w \rangle.
\end{equation}
Notice that 
\begin{equation}
	\nabla_G \gamma_{\delta,s} = -(s\delta)^{-1} \frac{X^T}{r}
	\indic_{B_s \setminus \clos{B}_{(1 - \delta)s}},
\end{equation}
so using the co-area formula we can estimate the integral on the right-hand side
like
\begin{multline}
	\int_{\reg G \cap D_2 \times \bR} (r^2 - s^2 )
	\langle \nabla_G \gamma_{\delta,s}, \nabla_G w \rangle \\
	= \int_{(1-\delta)s}^s  (s\delta)^{-1} (s^2 - \theta^2)
	\Big \{	\int_{\bdary B_\theta \cap \reg G} 
	\big \langle \nabla_G w , \frac{X}{\theta} \big \rangle
	\Big \}  \intdiff \theta.
\end{multline}
As $\theta \in (s(1 - \delta),s)$ we get
$s^2 - \theta^2 \leq s^2 \delta (2 - \delta)$ and
\begin{multline}
	\Big \lvert \int_{\reg G \cap D_2 \times \bR} \gamma_{\delta,s}
	\langle \nabla_G w, X \rangle  \Big \rvert 
	\leq \Big \lvert \int_{\reg G \cap D_2 \times \bR}
	(s^2 - r^2) \langle \nabla_G \gamma_{\delta,s},\nabla_G w \rangle \Big \rvert  \\
	\leq \int_{(1-\delta)s}^s \int_{\bdary B_\theta \cap \reg G}
	(2 - \delta)s \abs{\nabla_G w}
	= (2-\delta)s \int_{\reg G \cap B_s \setminus \clos{B}_{(1 - \delta)s}}
	\abs{\nabla_G w}.
\end{multline}
Then integrating this over $s \in (\sigma,\rho)$ we obtain a
bound for the second integral in~\eqref{eq_mean_value_first_identity},
\begin{equation}
	\label{eq_second_int_last_before_delta_to_zero}
	\Big \lvert \int_\sigma^\rho
	s^{-n-1} \int_{\reg G \cap D_2 \times \bR} \gamma_{\delta,s}
	\langle \nabla_G w, X \rangle \Big \rvert
	\leq 2\int_\sigma^\rho s^{-n}
	\int_{\reg G \cap B_s \setminus \clos{B}_{(1-\delta)s}} \abs{\nabla_G w}.
\end{equation}

As we announced above, we may now let $\delta \to 0$ at the same time on
the left-hand side of~\eqref{eq_mean_value_first_identity}, in%
~\eqref{eq_mean_value_ineq_first_integral}
and \eqref{eq_second_int_last_before_delta_to_zero},
justifying the convergence each time by dominated convergence.
Thus
\begin{equation}
	\rho^{-n} I(\rho) - \sigma^{-n} I(\sigma)
	\geq \lim_{\delta \to 0} \Big \{ \int_{\sigma}^\rho s^{-n} J_{\delta}'(s) \Big \},
\end{equation}
where we write $I(s) = \int_{\reg G \cap B_s} w$ for all $s \in (0,\rho)$.
We conclude by evaluating the limit on the left-hand side, which is
\begin{multline}
	\int_{\reg G \cap B_\rho} w \abs{D^\perp r}^2
	\Big \{
		\rho^{-n} - \sigma^{-n} \indic_{r < \sigma}
		+ (r \vee \sigma)^{-n} - \rho^{-n}
	\Big \} \\
	= \int_{\reg G \cap B_\rho \setminus \clos{B}_\sigma} w \abs{D^\perp r}^2
	r^{-n}.
\end{multline}
Therefore we obtain
\begin{equation}
	\rho^{-n} I(\rho) - \sigma^{-n} I(\sigma)
	\geq \int_{\reg G \cap B_\rho \setminus \clos{B}_\sigma} w \abs{D^\perp r}^2
	r^{-n},
\end{equation}
which establishes~\eqref{eq_mean_value_ineq} and concludes the proof.
\end{proof}

\subsection{Proof of the gradient bounds}

We place ourselves in the situation described in
Lemma~\ref{lem_int_grad_bounds_form_in_small_ball}.
Let $u \in C^{1,\alpha}(D_{3r};\calA_2)$ be a two-valued minimal graph.
Write $u_1(0),u_2(0)$ for the two values of $u$ at $0 \in \bR^n$, with
corresponding $Du_1(0),Du_2(0) \in \bR^n$. 
Note that we may consider $w$ as a two-valued function defined on $D_{3r}$,
including at the singular points of $u$.

Applying the mean-value inequality at either point $X_i = (0,u_i(0))$
we obtain
$\rho^{-n} \int_{G \cap B_\rho(X_i)} w  - \sigma^{-n} \int_{G \cap B_\sigma(X_i)} w \geq 0$
for all $0< \sigma < \rho < 3r$.
Fixing $\rho = r$ and letting $\sigma \to 0$ we find
\begin{equation}
	\lim_{\sigma \to 0} \Big\{  (\omega_n\sigma^n)^{-1}
	\int_{\reg G \cap B_\sigma(X_i)} w \Big \}
	\leq (\omega_n r^n)^{-1} \int_{\reg G \cap B_{r}(X_i)} w.
\end{equation}
If $X_i$ is regular, then the limit on the right-hand side is $w_i(0)$
and if $X_i$ is a singularity then it is equal $w_1(0) + w_2(0)$.
Translate the graph so that $0 \in G$ and $w$ the larger of its two values
there. In both cases
\begin{equation}
	\label{eq_max_w_bounded_by_integral}
	2 \max \{ w_1(0),w_2(0) \} \leq( \omega_n r^n)^{-1}  \int_{\reg G \cap B_{r}} w.
\end{equation}
This allows us to reduce the proof of Lemma~\ref{lem_int_grad_bounds_form_in_small_ball}
to an estimation of the integral on the right-hand side of%
~\eqref{eq_max_w_bounded_by_integral}.

\begin{claim}\label{claim_int_bound_for_gradient_estimate}
	There is a constant $C = C(n)$ so that if $u \in C^{1,\alpha}(D_{3r};\calA_2)$
	is a two-valued minimal graph with $\orig \in G$ then
	\begin{equation}
		r^{-n} \int_{\reg G \cap B_r} w \leq C ( 1 + r^{-1} \sup_{D_{2r}} \norm{u}).
	\end{equation}
\end{claim}
\begin{proof}
We split the integral $\int_{\reg G \cap B_r} w
= \int_{\abs{x}^2 + u^2 < r^2} wv$ into the sum of
$\int_{\abs{x}^2 + u^2 < r^2} w$ and
$\int_{\abs{x}^2 + u^2 < r^2} \frac{\abs{Du}^2}{v} w$.
The former is easier to estimate, as $w \leq v$. Therefore
$\int_{\abs{x}^2 + u^2 < r^2} w \leq \calH^n(G \cap B_r) \leq C(n) r^n$,
using the area bounds of Proposition~\ref{prop_area_estimate}.

We estimate the latter under the weaker restriction that $\abs{x} < r, \abs{u} < r$,
for notational convenience. Let $\eta \in C_c^1(D_{2r})$ be a standard
cutoff function with $\eta \equiv 1$ on $D_r$ and $\abs{D \eta} \leq 2 r^{-1}$.
We write 
\begin{equation}
\Phi(x,z,p) = \frac{1}{2} \log(1 + \abs{p}^2) (z_{r} + r) \eta(x),
\end{equation}
where $z_r$ is defined as in~\eqref{eq_def_zr}, meaning that here
\begin{equation}
z_r + r =
\begin{cases}
	2r &\text{ if $z > r$}, \\
	z + r &\text{ if $-r \leq z \leq r$}, \\
	0 & \text{ if $z < -r$}.
\end{cases}
\end{equation}
One may then check that $\Phi$ satisfies the required hypotheses
laid out in Proposition~\ref{prop_test_function}
to justify its use as a test function in the
equation~\eqref{eq_weak_MSE}.
The only hypothesis we check here is~\eqref{eq_condition_weak_twovalued_MSE},
\begin{equation}
	\label{eq_boundedness_condition_for_weak_MSE}
	\int_{D_{3r} \setminus \calB_u} \abs{D (\Phi(x,u,Du))} < +\infty.
\end{equation}
As $\Phi(x,u,Du) = (u_r + r) \eta w$ we calculate its derivative as
\begin{equation}
	\label{eq_expression_derivative_Phi}
D \Phi(x,u,Du) = Du \indic_{\abs{x} < r, \abs{u} < r} \eta w + D \eta (u_r + r) w
+ Dw (u_r + r) \eta.
\end{equation}
Only the last term is not locally bounded. Instead we find the following
integral bound. (The main observation underlying 
this bound, as proved in Claim~\ref{claim_bound_relating_Dw_and_nabla_w},
is the pointwise inequality~\eqref{eq_identity_grad_derivative}.
The unwieldiness of the proof is caused by the possibility of branch points in $D_{3r}$.)

\begin{claim}
	\label{claim_bound_relating_Dw_and_nabla_w}
For all compact $K \subset D_{3r}$ there is a constant $C_K > 0$
so that \begin{equation}
	\label{eq_bound_for_Dw}
	\int_{K \setminus \calB_u} \abs{Dw}
	\leq \int_{\reg G \cap K \times \bR} \abs{\nabla_G w}	
	\leq C_K.
\end{equation}

\end{claim}
\begin{proof}
This essentially follows from the integral estimates for $\abs{\nabla_G w}$
in \eqref{eq_local_bounds_for_w_and_nabla_w}. 
To justify this rigorously we proceed as follows.
First we extend the two-valued function
\begin{equation}
	w = \{ w_1,w_2 \}
	= \frac{1}{2} \{ \log (1 + \abs{Du_1}^2), \log(1 + \abs{Du_2}^2) \}
\end{equation}
to a two-valued function defined on the cylinder $D_{3r} \times \bR$ by setting
it constant in the vertical variable,
\begin{equation}
	\label{eq_vertical_extension}
	w(x,X^{n+1}) = w(x) = \{ w_1(x) , w_2(x) \}
\end{equation}
for all $(x,X^{n+1}) \in D_{3r} \times \bR$.

Near points $x \in D_{3r} \setminus \calB_u$ we can make a local selection for
$u$ by say $u_1^{x},u_{2}^x \in C^\infty(D_\sigma(x))$
so that $u = \{ u_1^{x}, u_2^{x} \}$ in $D_\sigma(x)$.
This selection is also valid for $w$, meaning that also $w = \{ w_1^{x},w_2^{x} \}$.
Of course this relation also holds for the extension of $w$ to $D_\sigma(x) \times \bR$.

In the cylinder $D_\sigma(x) \times \bR$ we split the
graph into the transverse union
\begin{equation}
G \cap D_\sigma(x) \times \bR = \graph u_1^x \cup \graph u_2^x.
\end{equation}
Abbreviating $G_i^x = \graph u_i^x$ for $i = 1,2$ we can separately
calculate
\begin{equation}
	\abs{\nabla_{G_i^x} w_i^x(y,Y^{n+1})}^2
	= \abs{D w_i^x(y,Y^{n+1})}^2 - \abs{\langle \nu_{G_i^x} , Dw_i^x(y,Y^{n+1}) \rangle}^2
\end{equation}
at all $Y = (y,Y^{n+1}) \in D_\sigma(x) \times \bR$.
As the $w_i^x$ are both independent of the vertical variable we obtain
that for $i = 1,2$,
\begin{equation}
	\label{eq_identity_grad_derivative}
	\abs{\langle \nu_{G_i^x} , e_{n+1} \rangle } \abs{D w_i^x}
	\leq \abs{\nabla_{G_i^x} w_i^x} 
	\text{ on $D_\sigma(x) \times \bR$}.
\end{equation}
Next we split the integral $\int_{D_\sigma(x)} \abs{Dw}
= \int_{D_\sigma(x)} \abs{D w_1^x} + \int_{D_\sigma(x)} \abs{D w_2^x}$,
and separately bound the two terms using \eqref{eq_identity_grad_derivative},
that is for $i = 1,2$,
\begin{equation}
	\int_{D_\sigma(x)} \abs{Dw_i^x}
	=\int_{G_i^x \cap D_\sigma(x) \times \bR} \frac{ \abs{Dw_i^x}}{v_i^x}
	\leq \int_{G_i^x \cap D_\sigma(x) \times \bR} \abs{\nabla_{G_i^x} w_i^x}.
\end{equation}
where recall $v_i^x = (1 + \abs{Du_i^x}^2)^{1/2}
=  \langle \nu_{G_i^x} , e_{n+1} \rangle^{-1}$.
Thus we obtain
\begin{equation}
	\int_{D_\sigma(x)} \abs{Dw}
	\leq \int_{\reg G \cap D_\sigma(x)} \abs{\nabla_G w},
\end{equation}
which is finite by~\eqref{eq_local_bounds_for_w_and_nabla_w} for example.

We may now return to the original problem of estimating the integral
$\int_{K  \setminus \calB_u} \abs{Dw}$. We may take a countable
cover of the set $K  \setminus \calB_u$ by some collection of
discs $(D_{\sigma_j}(x_j) \mid j \in \bN)$
centered at points $x_j \in K  \setminus \calB_u$
with
\begin{equation}
	D_{\sigma_j}(x_j) \subset D_{3r} \setminus \calB_u,
\end{equation}
and let $(\rho_j \mid j \in \bN)$ be a partition of unity subordinate to the cover.
Justifying the permutation of the sum and the integral by monotone
convergence for example, we can decompose $\int_{K \cap D_{3r} \setminus \calB_u} \abs{Dw}$
like
\begin{equation}
	\label{eq_integral_gradient_Dw_comparison}
	\int_{K \cap D_{3r} \setminus \calB_u} \abs{Dw}
	= \sum_{j \in \bN} \int_{K \cap D_{3r} \setminus \calB_u}
	\abs{Dw} \rho_j.
\end{equation}
Arguing as in our derivation of the inequality~\eqref{eq_integral_gradient_Dw_comparison}
we can bound each of the integrals on the right-hand side by
\begin{equation}
\int_{K \setminus \calB_u} 	\abs{Dw} \rho_j
\leq \int_{\reg G \cap D_{\sigma_j(x_j)  \times \bR}}
\abs{\nabla_G w} \rho_j,
\end{equation}
where we extend $\rho_j$ to a function defined on $D_{\sigma_j(x_j)} \times \bR$
by setting it constant in the vertical variable, as we did for $w$
in~\eqref{eq_vertical_extension}.
Finally we may use monotone convergence once again to permute sums with
integrals, yielding the desired
\begin{equation}
	\int_{K  \setminus \calB_u} \abs{Dw}
	\leq \int_{\reg G \cap K \times \bR} \abs{\nabla_G w}
	\leq C_K,
\end{equation}
where the last inequality comes from~\eqref{eq_local_bounds_for_w_and_nabla_w}.
\end{proof}

As we imposed that $\spt \eta \subset D_{2r}$, the
inequality~\eqref{eq_bound_for_Dw} from the claim applied with $K = \clos{D}_{2r}$ 
confirms the boundedness required by~\eqref{eq_boundedness_condition_for_weak_MSE}.
We may thus substitute the expression~\eqref{eq_expression_derivative_Phi}
we calculated for $D \Phi(x,u,Du)$ into~\eqref{eq_weak_MSE} to get
\begin{equation}
	\int_{D_{3r}} \big \langle \frac{Du}{v} ,
	Du \indic_{\abs{x} < r, \abs{u} < r} \eta w 
	+ D \eta (u_r + r) w 
	+ Dw (u_r + r) \eta \big \rangle = 0,
\end{equation}
so that
\begin{equation}
	\int_{\abs{x} < r,  \abs{u} < r  } \frac{\abs{Du}^2}{v} w 
	\leq 2r \int_{\abs{x} < 2r , u > -r } \abs{D \eta} w + \abs{Dw} \eta.
\end{equation}
To justify the expression on the left-hand side, simply note that for $i = 1,2$,
wherever $u_i > - r$ we have $(u_i)_r + r \leq 2r$, and  if
$u_i \leq -r$ then $(u_i)_r + r = 0$.
The proof then boils down to separately estimating the two integrals
\begin{equation}
	\int_{\abs{x} < 2r, u > -r } \abs{D \eta} w
\quad \text{and} \quad
\int_{\abs{x} < 2r,  u > -r } \abs{Dw} \eta.
\end{equation}
The first integral is easier. Indeed using the fact that $w \leq v$ and
the area bounds of Lemma~\ref{lem_cylindrical_area_estimates} we get
\begin{equation}
	\int_{\abs{x} < 2r, u > -r}
	\abs{D \eta} w \leq 2 r^{-1} \calH^n(G \cap D_{2r} \times \bR)
\leq C(n) r^{n-1}(1 + M r^{-1}),
\end{equation}
where we set $M = \sup_{D_{2r}} \norm{u}$.

For the second integral $\int_{u > -r} \abs{Dw} \eta$ we may start by arguing
as in the proof of Claim~\ref{claim_bound_relating_Dw_and_nabla_w} 
to justify that 
\begin{equation}
	\label{eq_ineq_with_Dw_eta_and_nabla_w}
\int_{u > -r} \abs{Dw} \eta \leq
\int_{\reg G \cap D_{2r}  \times (-r,\infty)} \abs{\nabla_G w} \eta.
\end{equation}
To estimate this, recall the helpful
integral inequality~\eqref{eq_gradient_estimate_for_nabla_w} we used
in the proof of Lemma~\ref{lem_weak_PDE_for_w}.
To use it in the present context, we extend $\eta$ to a test function compactly
supported in the cylinder
$D_{2r} \times \bR$ by multiplying by a cutoff function $\tau \in C_c^1(\bR)$
in the vertical direction. We further impose that $\tau \equiv 1$ on
$(-r,M)$ and $\spt \tau \subset (-2r,M + r)$ with $\abs{\tau'} \leq 2 r^{-1}$.
Then from \eqref{eq_gradient_estimate_for_nabla_w} we obtain
\begin{equation}
	\label{eq_conseq_of_stability_type_bound}
	\int_{\reg G} \abs{\nabla_G w}^2 \tau^2 \eta^2
	\leq 8 \int_{\reg G} \abs{\nabla_G \tau}^2 \eta^2 + \tau^2 \abs{\nabla_G \eta}^2
	\leq 64 r^{-2} \calH^n(G \cap \spt \phi).
\end{equation}
Estimate the right-hand side
of~\eqref{eq_ineq_with_Dw_eta_and_nabla_w} with H\"{o}lder's inequality,
\begin{multline}
	\int_{\reg G \cap D_{2r} \times (-r,\infty)} \abs{\nabla_G w} \eta \\
\leq (\calH^n(G \cap D_{2r} 
\times (-2r,M + r))^{1/2}
\Big(\int_{G \cap D_{2r} \times (-r,M)} \abs{\nabla_G w}^2 \eta^2\Big)^{1/2},
\end{multline}
and combining this with~\eqref{eq_conseq_of_stability_type_bound},
\begin{equation}
	\int_{\reg G \cap D_{2r} \times (-r,\infty)} \abs{\nabla_G w} \eta
\leq  8r^{-1} \calH^n(G \cap D_{2r} \times \bR).
\end{equation}
This in turn can be bounded using the area estimates from
Lemma~\ref{lem_cylindrical_area_estimates}:
$\calH^n(G \cap D_{2r} \times \bR)
\leq C(n)r^n (1 + Mr^{-1})$.
This yields the desired bound for $\int_{u > -r} \abs{Dw} \eta$
via~\eqref{eq_ineq_with_Dw_eta_and_nabla_w}, and concludes the proof of the claim.
\end{proof}

\section{A regularity lemma for two-valued minimal graphs}

\label{sec_regularity_lemma_for_two_valued_minimal_graphs}

\subsection{A maximum principle near branch point singularities}
\label{subsec_max_princ_near_branch_point}

\begin{lem}
\label{lem_max_princ_for_nu_dot_e}
Let $\alpha \in (0,1)$ and $u \in C^{1,\alpha}(D_2;\calA_2)$
be a two-valued minimal graph. Suppose that at the origin
\begin{equation}
	u(0) = \{0,0\} \text{ and } Du(0) = \{ 0 , 0 \}.
\end{equation}
Let $e \in \bR^n \times \{ 0 \}$ be a fixed unit vector.
If $\langle Du(X) , e \rangle \leq 0$ for all $X \in \reg G \cap D_2 \times \bR$
then it vanishes identically.
\end{lem}
\begin{proof}
Let $X \in \reg G \cap D_2 \times \bR$, 
and note that $\langle Du(X),e \rangle \leq 0$ if and only the upward-pointing unit
normal to $\graph u$ has $\langle \nu(X) , e \rangle \geq 0$.
We argue by contradiction, assuming that $\langle \nu , e \rangle$ is non-negative
but does not vanish identically.
It is well-known that the function $\langle \nu , e \rangle$ is a Jacobi field
for $G$, that is it satisfies the equation $\Delta \langle \nu , e \rangle 
+ \abs{A_G}^2 \langle \nu , e \rangle = 0$ both pointwise on $\reg G$ and weakly
through singularities of $G$. (The justification for this can be made in essentially
the same way as when working with $\langle \nu,e_{n+1} \rangle$ in the above.)
Moreover, by the standard strong maximum principle, we know that
$\langle \nu , e \rangle > 0$ on $\reg G$, meaning that we can
define a smooth function $w_e \in C^2(\reg G)$ by
\begin{equation}
w_e(X) = - \log \langle \nu(X) , e \rangle
\text{ at all $X \in \reg G$}.
\end{equation}

The hypotheses of the claim ensure that $\langle \nu(0), e \rangle = 0$ at
the origin, which also means that $w$ diverges there.
To avoid technical difficulties related to this, we perturb the vector
$e$ slightly.
Let $\theta > 0$ be a small angle, through which we rotate $e$ in the
two-dimensional plane $\linspan \{ e,e_{n+1} \}$, yielding the vector
\begin{equation}
	e_\theta = (\cos \theta) e + (\sin \theta) e_{n+1}.
\end{equation}

Unless we are in the pathological case where $u$ diverges near the boundary
of $D_2$, the function $\langle \nu , e_{n+1} \rangle = \frac{1}{\sqrt{1 + \abs{Du}^2}}$
is positive and bounded below, say 
\begin{equation}
	\langle \nu , e_{n+1} \rangle \geq \alpha > 0
	\text{ on $\reg G \cap B_1$}.
\end{equation}
Should $u$ in fact diverge near the boundary, we can rescale it around the
origin by a factor $\lambda > 1$ close to one, and restrict the resulting
function to $D_2$ to reduce to the case where $u$ is bounded near the origin.

As a consequence the rotated vector $e_\theta$ also has
\begin{equation}
	\langle \nu , e_\theta \rangle \geq \alpha \sin \theta  > 0
	\text{ on $\reg G \cap B_1$}.
\end{equation}
We can then define the function
\begin{equation}
	w_\theta(X) = -\log \langle \nu(X) , e_\theta \rangle
	\text{ at all $X \in \reg G$}
\end{equation}
without running the risk of it diverging anywhere.
This function has, locally for all compact $K \subset D_2 \times \bR$ that
\begin{equation}
	\int_{K \cap \reg G} w_\theta^2 + \abs{\nabla w_\theta}^2 < +\infty,
\end{equation}
and it satisfies the equation
\begin{equation}
	\Delta w_{\theta} - \abs{\nabla w_\theta}^2 - \abs{A_G}^2 = 0
\end{equation}
weakly on $G \cap D_2 \times \bR$.
(These facts can be checked in much the same way as we did for the function
$- \log \langle \nu , e_{n+1} \rangle$ in Section~\ref{sec_grad_estimates},
see Lemma~\ref{lem_weak_PDE_for_w}.)

In particular the function $w_\theta$ is weakly subharmonic on $G$,
and thus by the mean-value inequality, we get that for all
$0 < r < s < 1$ and all points $X \in B_1$,
\begin{equation}
\frac{1}{\omega_n r^n} \int_{\reg G \cap B_r(X)} w_\theta
\leq \frac{1}{\omega_n s^n} \int_{\reg G \cap B_s(X)} w_\theta.
\end{equation}
Applying this at the origin and letting $r \to 0$
(as is justified by Fatou's lemma), we obtain that for all $0 < s < 1$,
\begin{equation}
	\label{eq_mean_value_ineq_for_wtheta}
	2 w_\theta(0) \leq \frac{1}{\omega_n s^n} \int_{\reg G \cap B_s} w_\theta.
\end{equation}

For all $\theta > 0$ we write
\begin{equation}
	M_\theta = \sup_{\reg G \cap B_1} w_\theta 
	\text{ and }
	m_\theta = \inf_{\reg G \cap B_1} \langle \nu , e_\theta \rangle,
\end{equation}
which are related by $M_\theta = - \log m_\theta$.
Then 
\begin{equation}
w_\theta(0) \geq M_\theta + \log \alpha,
\end{equation}
where recall $\alpha = \inf \langle \nu , e_{n+1} \rangle.$
Indeed at the origin $\langle \nu(0),e_\theta \rangle = \sin \theta$,
whereas $m_\theta \geq \alpha \sin \theta$. Therefore
$\langle \nu(0) , e_\theta \rangle \leq  m_\theta / \alpha$.
Translating this to $w_\theta$ we obtain
$w_\theta(0) \geq - \log (m_\theta / \alpha) = - \log m_\theta + \log \alpha$.
Let $\delta > 0$ be given.
As $M_\theta \to \infty$ as $\theta \to 0$, we may choose a small value
$\theta_0 > 0$ in terms of $\alpha$ so that for all $\theta \in (0,\theta_0)$ we also have
\begin{equation}
w_\theta(0) \geq (1-\delta/2) M_\theta.
\end{equation}
Substitute this into the mean-value inequality~\eqref{eq_mean_value_ineq_for_wtheta}
with radius $s = 1$, obtaining that
\begin{equation}
	(2-\delta) M_\theta \leq \frac{1}{\omega_n} \int_{\reg G \cap B_1} w_\theta.
\end{equation}

Let $\lambda > 0$ be a parameter whose value we will fix later. Then we may
split the integral on the right-hand side by conditioning on the event that
$\{ w_\theta \geq \lambda \}$, obtaining two integrals which we can separately
bound by
\begin{equation}
	 \int_{\reg G \cap B_1  \cap \{ w_\theta > \lambda \}}
	w_\theta
	\leq M_\theta \calH^n(\reg G \cap B_1 \cap
	\{ w_\theta > \lambda \})
\end{equation}
and
\begin{equation}
	 \int_{\reg G \cap B_1 \cap \{ w_\theta \leq \lambda \} }
	w_\theta
	\leq \lambda  \calH^n(\reg G \cap B_1 \cap \{ w_\theta \leq \lambda \}).
\end{equation}

Recall that $w$ is large only near $\{ \langle \nu , e \rangle = 0 \}
\subset \sing G \cap B_1$, and hence is bounded away from the singular set.
Since the functions $w_\theta$ converge to $w$ pointwise on $\reg G \cap \clos{B}_1$,
and uniformly in compact subsets $K \subset \reg G \cap \clos{B}_1$, we
obtain uniform bounds for the sequence too: 
for all compact $K \subset \reg G \cap \clos{B}_1$
there exist $\theta_K > 0$ and $D_K > 0$ so that for all $\theta \in (0,\theta_K)$,
\begin{equation}
	\label{eq_unif_bound_w_theta_away_from_sing}
	w_\theta \leq D_K \text{ on $K$}.
\end{equation}
Working in the larger ball $B_{3/2}$, we see that 
$\calH^n(\sing G \cap \clos{B}_{3/2}) = 0$, which also means that
small neighbourhoods of this set have arbitrarily small Hausdorff measure.
That is, given any $\eps > 0$ we may find a finite open cover of
$\sing G \cap \clos{B}_{3/2}$ by balls $B_{r_1}(X_1),\dots,B_{r_N}(X_N)$
with $\sum_{k=1}^N r_k^n \leq \eps$.
Perhaps after slightly increasing the radii of the balls 
in the cover, we may arrange for $\reg G \cap B_1 \setminus \cup_{k=1}^N B_{r_k}(X_k)$
to lie a positive distance away from the singular set.
Using the bound above in~\eqref{eq_unif_bound_w_theta_away_from_sing}, 
we see that there must be a constant
$D > 0$ so that for all $\theta \in (0,\theta_0)$
\begin{equation}
w_\theta(X) \leq D
\text{ at all $X \in \reg G \cap B_1 \setminus \cup_{k=1}^N B_{r_k}(X_k)$},
\end{equation}
after adjusting $\theta_0$ to a smaller value if necessary.
We may then set $\lambda = D$, and see that for all $\theta \in (0,\theta_0)$,
\begin{equation}
	\label{eq_area_bound_splitting_lemma}
	\calH^n(\reg G \cap B_1 \cap \{ w_\theta > D \})
	\leq \sum_{k=1}^N \omega_n r_k^n \leq \omega_n \eps.
\end{equation}

The divergence of $M_\theta$ as $\theta \to 0$ additionally lets us
impose that $\theta$ be small enough that $M_\theta \geq C(n) D$,
where $C(n)$ is a constant so that 
$\calH^n(\reg G \cap B_1) \leq  \omega_n C(n)$,
available via the area bounds of Proposition~\ref{prop_area_estimate}.
With $\theta$ as small as this, we get
\begin{equation}
	\label{eq_small_values_bound_splitting_lemma}
	D/\omega_n \calH^n(\reg G \cap B_1 \cap \{ w_\theta \leq D \}  )
	\leq M_\theta.
\end{equation}

Substituting  \eqref{eq_area_bound_splitting_lemma} and
\eqref{eq_small_values_bound_splitting_lemma} into our decomposition
for $1 / \omega_n \int_{\reg G \cap B_1} w_\theta$ we obtain the inequality
\begin{equation}
	(2-\delta) M_\theta
	\leq \frac{1}{\omega_n} \int_{\reg G \cap B_1} w_\theta
	\leq \big(\eps  + 1 \big) M_\theta,
\end{equation}
which is absurd provided $\delta,\eps$ are small enough.
\end{proof}

The analogous statement is a lot easier to prove for single-valued,
smooth minimal graphs. In fact, this is an immediate application of
the classical, strong maximum principle, and indeed we used this in the proof
above to deduce that $\langle \nu(X) , e \rangle > 0$ for all regular
points $X \in \reg G \cap D_2 \times \bR$.

\subsection{Regularity by a geometric argument}

Here we show that Lipschitz two-valued minimal graphs are 
automatically regular This is somewhat well-known among experts
in the field, although an explicit proof is absent from the literature.
Here we follow a strategy suggested to us by S.\ Becker-Kahn, using the
results developed in his thesis~\cite{Spencer_Two_valued_graphs_arbitrary_codimension}.
Our aim is to prove the following two results simultaneously, using an
inductive argument on the dimension $n$.

\begin{thm}
\label{thm_regularity_lipschitz_two_valued_graphs}
Let $u \in \mathrm{\Lip}(D_2;\calA_2)$ be a two-valued
minimal graph with Lipschitz constant $L$.
Then there is $\alpha = \alpha(L,n) \in (0,1)$ so that
\begin{equation}
	u \in C^{1,\alpha}(D_2;\calA_2).
\end{equation}
\end{thm}

This turns out to be equivalent to the following, seemingly weaker lemma.

\begin{lem}
\label{lem_reg_lipschitz_graph_cone}
Let $u \in \mathrm{\Lip}(\bR^n;\calA_2)$ be a two-valued
minimal graph with Lipschitz constant $L$.
If additionally $u$ is homogeneous,
\begin{equation}
	\label{eq_u_homogeneous}
u(\lambda x) = \lambda u(x)
\text{ for all $\lambda > 0, x \in \bR^n$},
\end{equation}
then $u$ is linear.
\end{lem}

\begin{proof}
Before moving on to the inductive argument, let us explain how
Lemma~\ref{lem_reg_lipschitz_graph_cone} implies
Theorem~\ref{thm_regularity_lipschitz_two_valued_graphs}.
Let $u \in \Lip(D_2;\calA_2)$ be a two-valued minimal graph with 
Lipschitz constant $L$. This is smooth away from away from $\calB_u$,
so consider an arbitrary $x \in \calB_u  \cap D_2$. 
Write $X = (x,X^{n+1}) \in \calB(G)$ be the corresponding point in the
graph, at height $X^{n+1} = u_1(x) = u_2(x)$. For any tangent cone
$\bC_X \in \vartan(\abs{G},X)$---\emph{a priori} these are not unique---,
there is a two-valued Lipschitz function $U_X \in \Lip(\bR^n;\calA_2)$
with the same Lipschitz constant, so that $\bC_X = \abs{\graph U_X}$.
This function $U_X$ is homogeneous as in~\eqref{eq_u_homogeneous},
and thus by Lemma~\ref{lem_reg_lipschitz_graph_cone} this must be linear.
In other words there is $\Pi_X \in \Gr(n,n+1)$ so that
$\bC_X = 2 \abs{\Pi_X}$. 
Then by~\cite{Spencer_Two_valued_graphs_arbitrary_codimension} there
is $0 < \gamma = \gamma(n,L) < 1$ so that for some $\rho > 0$, 
$u \in C^{1,\gamma}(D_\rho(x);\calA_2)$. 
As the branch point $x$ was chosen arbitrarily and $\alpha(n,L) := \gamma$ does not
depend on it, we get $u \in C^{1,\gamma}(D_2;\calA_2)$.

The base of the inductive argument is simple, as when $n = 1$ then two-valued minimal graphs
are automatically linear.
For the induction step, assume that Theorem~\ref{thm_regularity_lipschitz_two_valued_graphs}
holds in dimension $n-1 \geq 1$.
We prove Lemma~\ref{lem_reg_lipschitz_graph_cone} in dimension $n$,
and for that purpose consider an arbitrary minimal graph
$u \in \Lip(\bR^n;\calA_2)$, homogeneous as in~\eqref{eq_u_homogeneous}.
We claim that there is $\gamma = \gamma(n,L) > 0$ so that
$u \in C^{1,\gamma}(\bR^n \setminus \{ 0 \};\calA_2)$. To see this,
let $X = (x,X^{n+1}) \neq 0 \in \sing G \cap D_2 \times \bR$.
Every tangent cone $\bC_X \in \vartan(G,X)$ is the graph
of a two-valued function $U_X \in \Lip(\bR^n;\calA_2)$ with the same Lipschitz constant.
By a standard dimension reduction argument, $\bC_X$ is invariant under
translation by $tX$ for all $t \in \bR$.
By the induction hypothesis $U_X$ is linear, and $\bC_X$ is
\begin{enumerate}
\item either a sum of two multiplicity one planes, $\bC_X
	= \abs{\Pi_1^X} + \abs{\Pi_2^X}$,
\item or a single multiplicity two plane, $\bC_X = 2 \abs{\Pi_1^X}$.
\end{enumerate}
By~\cite{Spencer_Two_valued_graphs_arbitrary_codimension} there exists
$0 < \gamma = \gamma(n,L) < 1$ so that for some $0 < \rho < \abs{x}$,
$u \in C^{1,\gamma}(D_\rho(x);\calA_2)$ regardless of whether $X = (x,X^{n+1})$
is a classical singularity or a branch point.
As $x$ is arbitrary and $\gamma$ can be chosen independently of it,
we get $u \in C^{1,\gamma}(\bR^n \setminus \{ 0 \};\calA_2)$.

To extend this across the origin, define a function $w$ on the regular set,
\begin{equation}
	\label{eq_defn_function_w}
	w(X) = -\log \langle \nu(X) , e_{n+1} \rangle
	\text{ for all $X \in \reg G$.}
\end{equation}
This is non-negative and bounded, and we may let
\begin{equation}
	\label{eq_defn_M}
	M = \sup_{ \reg G \cap \bdary B_1} w > 0,
\end{equation}
and consider a sequence of points $X_i \in \reg G \cap \bdary B_1$
with $X_i \to Z \in G \cap \bdary B_1$ and $w(X_i) \to M$ as $i \to \infty$.
\begin{claim}
	\label{claim_limit_point_unbranched}
If $Z = (z,Z^{n+1}) \notin \calB(G) \cap \bdary B_1$
then $u$ is locally linear near $Z$ in the sense that there is $\rho > 0$
and a smooth selection $u_1,u_2 \in C^\infty(D_\rho(z))$ with
$u_1$ linear and $Z \in \graph u_1$.
\end{claim}
\begin{proof}
As the graph $\abs{G}$ is invariant under homotheties, so is $w$,
whence if $X \in \reg G$ then $\lambda X \in \reg G$ 
and $w(\lambda X) = w(X)$ for all $\lambda > 0$.
Therefore~\eqref{eq_defn_M} also means $M = \sup_{\reg G \cap B_1} w.$

When $Z \in \reg G \cap \bdary B_1$ then by the classical strong maximum
principle $w$ is locally constant near $Z$, and thus so is so $\abs{Du}$.
Pick a small radius $\rho > 0$ so that a
smooth selection $\{ u_1, u_2 \}$ can be made for $u$ on $D_\rho(z)$.
We arrange for  $Z \in \graph u_1$. As $\abs{Du_1}$ is
constant in $D_\rho(z)$, it is harmonic by inspection.
By the Bochner formula, $\abs{D^2 u_1}^2 = \Delta \abs{Du_1}^2 \equiv 0$
on $D_\rho(z)$.
Thus $u_1$ is affine linear, and the homothety-invariance of $G$ means that
it must in fact be linear.

The argument is similar when $Z \in \calC(G) \cap \bdary B_1$.
Make a smooth selection $\{ u_1,u_2 \}$ for $u$ on $D_\rho(z)$ and
write $G_i = \graph u_i$.
Define the two functions $w_1,w_2$ on $G_1,G_2$ respectively, using the
analogue of~\eqref{eq_defn_function_w}. 
Without loss of generality assume that $G_1$ contains infinitely
many points of $\{ X_i \mid i \in \bN \}$.
As $w_1$ is continuous at the point $Z$, we get $w_1(Z) = M$.
From then on, one can argue in the same way as when $Z$ is regular.
\end{proof}

Now suppose $Z \in \calB(G) \cap \bdary B_1$.
This argument is a bit more involved, but revolves around the same idea.
Because $Du_1(z) = Du_2(z)$, the function $w$ can be continuously extended to $Z$.
(The same is true for all branch points.)
Write $\nu(Z) \in \bR^{n+1}$ for the unit normal, and $2 \abs{\Pi_Z}
\in \vartan(\abs{G},Z)$ for the tangent plane to $G$ at $Z$.
As $w(Z) = M > 0$, this plane is not horizontal, and $\nu(Z) \neq e_{n+1}$.
Let $P = \linspan \{ e_{n+1},\nu(Z) \} \subset \bR^{n+1}$. This intersects
$\Pi_Z$ in a one-dimensional line, from which we pick a vector $e \in \Pi_Z \cap P$
with $\langle e , e_{n+1} \rangle > 0$.
Write $e = a e_{n+1} + b \nu(Z)$ for $a,b \in \bR$, which are constrained by
$0 = \langle e , \nu(Z) \rangle = a \langle e_{n+1} , \nu(Z) \rangle  + b$,
or equivalently $b = - a \langle e_{n+1} ,\nu(Z) \rangle$.
Let $X \in \reg G$ be an arbitrary regular point near $Z$. Then
\begin{align}
	\label{eq_expression_nu_dot_e}
\langle \nu(X) , e \rangle &= \langle \nu(X) , a e_{n+1} - a \langle e_{n+1},\nu(Z) \rangle 
\nu(Z) \rangle \\
&= a  \langle \nu(X),e_{n+1} \rangle
- a \langle e_{n+1},\nu(Z) \rangle  \langle \nu(X),\nu(Z) \rangle.
\end{align}
Take $X$ close enough to $Z$ that $0 < \langle \nu(Z),\nu(X) \rangle \leq 1$,
say this holds for $X \in \reg G \cap B_\rho(Z)$ for example.
Moreover by construction $\langle e_{n+1} , \nu(Z) \rangle \leq \langle e_{n+1},
\nu(X) \rangle$, and hence
\begin{equation}
	\label{eq_nu_dot_e_non_negative}
\langle \nu(X) , e \rangle \geq 0
\text{ for all $X \in \reg G \cap B_\rho(Z)$}.
\end{equation}
Upon decreasing $\rho > 0$ there is
$U_Z \in C^{1,\gamma}( B_\rho(Z) \cap (Z + \Pi_Z);\Pi_Z^\perp)$ so that
$G \cap  B_\rho(Z) \subset  \graph U_Z$, by Wickramasekera's
Theorem~\ref{thm_wic_mult_two_allard}.
By construction $e \in \Pi_Z$, and we can apply Lemma~\ref{lem_max_princ_for_nu_dot_e}
to $U_Z$ to deduce from~\eqref{eq_nu_dot_e_non_negative}
that $\langle \nu(X) , e \rangle = 0$ for all $X \in \reg G \cap B_\rho(Z)$.
Returning to~\eqref{eq_expression_nu_dot_e} we see that this is
only possible if for these points we have both $\langle \nu(X) , e_{n+1} \rangle
= \langle \nu(Z),e_{n+1} \rangle$ and $\langle \nu(Z) , \nu(X) \rangle = 1$.
Either would suffice to conclude that
$w(X) = w(Z)$ for all $X \in \reg G \cap B_\rho(Z)$.
Once we have derived this, we may reason as in the proof of
Claim~\ref{claim_limit_point_unbranched} to draw the analogous conclusion.

We use this to show that $G$ must be a union of planes, 
using an argument similar to that used to prove Lemma~\ref{lem_decomposition_cone}.
Let $\calR$ be the set of connected components of $\reg G$, of which there
are at most countably many. Among them we write $\calR_f \subset \calR$ for
those $\Sigma \in \calR$ that are flat in the sense that
$\abs{A_\Sigma} \equiv 0$.
The rest is denoted $\calR_c = \calR \setminus \calR_f$.
We decompose $\abs{G} = \bC_f + \bC_c \in \IV_n(\bR^{n+1})$, respectively
defined by $\bC_f = \sum_{\Sigma \in \calR_f} \Theta_\Sigma \abs{\Sigma}$
and $\bC_c = \sum_{\Gamma \in \calR_c} \Theta_\Gamma \abs{\Gamma}$.
Here given $\Sigma \in \calR$ we write $\Theta_\Sigma \in \bZ_{> 0}$
for its multiplicity, which is constant by~\cite[Thm.~41.1]{Simon84}.
Both $\bC_f,\bC_c$ are invariant under homotheties, and stationary.
To justify the latter, it suffices to prove that $\bC_f,\bC_c$ are stationary
near points in $\calC(G)$, as the other singularities do not contribute to the first variation.
Pick some point $X \in \spt \norm{\bC_f} \cap \norm{\bC_c} \cap \calC(G)$,
and let $\rho > 0$ be so that we can decompose
$G \cap B_\rho(X) = \Sigma_1 \cup \Sigma_2$ into a union of two surfaces embedded
in $B_\rho(X)$, which meet transversely along $\sing G \cap B_\rho(X)$.
By a unique continuation argument we may arrange for
$\Sigma_1 \subset \spt  \norm{\bC_f}$ and $\Sigma_2 \subset \spt \norm{\bC_c}$.
Both $\Sigma_i$ are stationary in $B_\rho(X)$, whence $\bC_f,\bC_c$ are stationary
inside $B_\rho(X)$ too. As $X$ was arbitrary, they are stationary in $\bR^{n+1}$.
The argument above shows that $\calR_f \neq \emptyset$ and $\bC_f \neq 0$,
and by Lemma~\ref{lem_GMT_immersed_flat_union_of_planes} it is
supported in a union of planes, say $\spt \norm{\bC_f} = \Pi_1 \cup \cdots \cup \Pi_D$
with $D \leq 2$.
If $\bC_c = 0$ then we are done, otherwise $D = 1$ and $\bC_f = \abs{\Pi_1}$.
In this case too one ultimately finds that $\abs{G} = \abs{\Pi_1} + \abs{\Pi_2}$,
for instance using~\cite{Simon_Erasable_Singularity_Result}.
\end{proof}

\section{A Jenkins--Serrin type lemma for single-valued minimal graphs}
\label{sec_single_valued_minimal_graphs}

We make a brief excursion to single-valued minimal graphs, with the aim of
proving two basic technical results that will turn out essential in the classification
of vertical limit cones; see Section~\ref{sec_classical_cones_vertical} and 
Lemma~\ref{lem_single_sheet_lemma} in particular.

Throughout this section, $\Omega \subset \bR^n$ be a bounded, convex domain
with Lipschitz-regular boundary. Let $u \in C^2(\Omega) \cap C^1(\clos{\Omega})$
be a single-valued function, whose graph $G \in \I_n(\Omega \times \bR)$ is
minimal.
Using the convexity of $\Omega$, it is well-known that the current $\cur{G}$
is area-minimising: any current $T \in \I_n(\bR^{n+1})$ with $\bdary T = \bdary \cur{G}$
has larger area: $\calH^n(G) \leq \norm{T}(\bR^{n+1})$.
Next, given $A < B \in \bR$ let $G_{A,B} = G \cap \Omega \times (A,B)$.
As the set $\Omega \times (A,B) \subset \bR^{n+1}$ is convex, this retains
the area-minimising property: any current $T \in \I_n(\bR^{n+1})$ with $\bdary T = 
\bdary \cur{G_{A,B}}$ has $\calH^n(G \cap \Omega \times (A,B)) \leq \norm{T}(\bR^{n+1})$.

For the remainder we work in a more specialised, polyhedral setting.
Let $\Pi_1,\dots,\Pi_N \in \Gr(n,n+1)$ be vertical planes, that is
$\Pi_i = \Pi_i^0 \times \bR e_{n+1}$, consider points $p_1,\dots,p_N \in \bR^n$
and the affine planes $p_1 + \Pi_1,\dots,p_N + \Pi_N$. 
In terms of some choice of unit normals $n_1,\dots,n_N \in \bR^n$ these can
be written $p_i + \Pi_i^0 = \{ x \in \bR^n \mid \langle x, n_i \rangle = a_i \}$,
where we set $a_i = \langle p_i , n_i \rangle$.
Let $\Omega \subset \bR^n$ be a convex polyhedral domain bounded by these planes,
and assume that the normals point into $\Omega$:
$\Omega = \cap_{i=1}^N \{ x \in \bR^n \mid \langle x , n_i \rangle > a_i \}$.
The boundary of $\Omega$ contains what we call the \emph{faces}
$F_j = \{ x \in \bdary \Omega \mid \langle x , n_i \rangle = a_i \}$
and the \emph{edges} $E_{ij} = F_j \cap F_j$, which have
$\calH^{n-1}(E_{ij}) = 0$.

Let $T = T_0 \times \bR e_{n+1} \in \I_n(\bR^{n+1})$ be a vertical current
so that the support of $T_0$ is equal to a union of some subcollection of
the faces in $\bdary \Omega$. 
Write $F_T$ for the faces and, likewise $E_T$ for the edges contained in $\spt \norm{T_0}$.
Let $(\Omega_j \mid j \in \bN)$ be a sequence of bounded domains with
$\dist_{\calH}(\Omega_j,\Omega) \to 0$ as $j \to \infty$; these
need neither be convex, nor polyhedral or have piecewise smooth boundary.
For each $j \in \bN$ there is a regular, single-valued $u_j \in C^2(\Omega_j)$
whose graph $G_j$ is minimal.
We prove the following lemma for their limit, inspired by the work
of Jenkins--Serrin in dimension three~\cite{JenkinsSerrin_Dirichlet_problem_infinite_data,
JenkinsSerrin_Variational_Problems_II}.

\begin{thm}
\label{thm_no_folding}
Let $\Omega, \Omega_j \subset \bR^n$ and $u_j \in C^2(\Omega_j)$ be as above.
Suppose that $\dist_{\calH}(\Omega_j,\Omega) \to 0$, $\cur{G_j} \to T$
and that for all $0 < \sigma < 1 < A$ there is $J = J(\sigma,A) \in \bN$ 
so that $\spt \bdary \cur{G_j \cap \Omega_j \times (-A,A)} \subset (E_T)_\sigma$
for all $j \geq J$.
Then $\norm{T_0} \leq \calH^{n-1}(\bdary \Omega) / 2.$
\end{thm}

\begin{proof}
We obtain the conclusion by constructing a comparison surface. We begin
the construction by making the following two assumptions:
\begin{enumerate}[label = (\arabic*)]
	\item \label{item_domain_contained_in_sequence}
		$\Omega \subset \Omega_j$ for all $j$,
\item \label{item_divergence_away_from_spt_T_0}
	and for all $0 < \tau < 1 < A$ there is $J = J(\tau,A) \in \bN$ so that
	$\abs{u_j} >2 A$ on $\Omega \setminus (F_T)_\tau$ for all
	$j \geq J$.
\end{enumerate}
We explain at the end of the proof why this can be done without restricting
the generality of the argument.

Let $0 < \tau < \sigma < 1 <  A $ be given, with the eventual aim of letting
$\sigma,\tau \to 0$ and $A \to \infty$.
We may perturb these by a small amount to guarantee that the level
sets $\{ u_{j} = \pm A \}$ are regular inside $\Omega_j$, justifying this by Sard's lemma;
whenever we adjust the values of $\tau,\sigma$ or $A$ we do so in a way which preserves
this property.

First we adjust $\tau$ in terms of $\sigma$ so that
$(F_T)_\tau \cap \bdary \Omega \subset F_T \cup (E_T)_\sigma$ and 
the $\tau$-tubular neighbourhoods of any two distinct faces are either
disjoint or meet inside $(E_T)_\sigma$.
Hence we can decompose $(F_T)_\tau \setminus (E_T)_\sigma$ into a disjoint
union, with one connected component for each face.
We define the open subset $\Omega_{j,\tau,\sigma} \subset \Omega_j$
by $\Omega_{j,\tau,\sigma} =  \Omega_j \cap \{\Omega \cup (F_T)_\tau \}
\setminus (E_T)_\sigma$ and take $j \geq J(\tau,\sigma,A)$ large enough
that $\abs{u_j} > 2A$ on $\Omega_{j,\tau,\sigma}$.

Consider the open set $W_{j,\tau,\sigma,A} \subset \Omega_{j,\tau,\sigma}
\times (-A,A)$ defined by
\begin{equation}
	W_{j,\tau,\sigma,A}
	= \{ (x,X^{n+1}) \in \Omega_{j,\tau,\sigma} \times (-A,A)
	\mid u_j(x) < X^{n+1} \}.
\end{equation}
This is a Caccioppoli set, whose current boundary we decompose into
\begin{align}
	-\bdary \cur{W_{j,\tau,\sigma,A}}
	& = \cur{(\{ u_j < A \} \cap \Omega_{j,\tau,\sigma} ) \times \{A \}} \\
	& + \cur{(\{ u_j < -A \}) \cap \Omega_{j,\tau,\sigma}) \times \{ -A \}} \\
	& + \cur{\bdary \Omega \setminus ([E_T]_\sigma \cup F_T) \times (-A,A)}\\
	& + \cur{\bdary W_{j,\tau,\sigma,A} \cap \bdary (E_T)_\sigma \times (-A,A) } \\
	&+ \cur{G_j \cap \Omega_{j,\tau,\sigma} \times (-A,A)}.
\end{align}
We estimate the areas of all the summands separately. Let a small $\delta > 0$ be given.
For the first two we respectively find
\begin{equation}
	\calH^n( (\{ u_j < \pm A \} \cap \Omega_{j,\tau,\sigma}) \times \{ \pm A \})
	\leq \calH^n(\Omega_{j,\tau,\sigma}) \leq \calH^n(\Omega_j).
\end{equation}
Using the convergence $\Omega_j \to \Omega$ in the Hausdorff distance,
we find that for large $j \geq J(\tau,\sigma,A,\delta)$ so that
$\calH^n(\Omega_{j}) \leq \calH^n(\Omega) + \delta$.
The third term, $\cur{\bdary \Omega \setminus ([E_T]_\sigma \cup F_T) \times (-A,A)}$
does not depend on $j$, and taking $\sigma,\tau$ small
enough in terms of $\delta$ we get
\begin{multline}
\calH^n(\bdary \Omega  \setminus ([E_T]_\sigma \cup F_T) \times (-A,A)) \\
\leq 2A(\calH^{n-1}(\bdary \Omega) - \calH^{n-1}(\spt T_0)) + 2A \delta.
\end{multline}
For the fourth, we bound 
\begin{align}
	\calH^n( \bdary W_{j,\tau,\sigma,A} \cap \bdary (E_T)_\sigma \times (-A,A)) 
	&\leq 2A \calH^{n-1}( \bdary (E_T)_\sigma ) \\
	&\leq 4\pi A \sigma \calH^{n-2}(E_T) + 2A \delta.
\end{align}
Next, the convergence assumed in the statement allows the lower bound
\begin{align}
	\calH^n(G_j \cap \Omega_{j,\tau,\sigma} &\times [-A,A]) \\
	&\geq 2A ( \calH^{n-1}(F_T) - \calH^{n-1}((E_T)_\sigma) )
	- 2 A \delta \\
	&\geq 2A (\calH^{n-1}(F_T) - 2 \pi \sigma \calH^{n-2}(E_T) - \delta)
	- 2 A \delta,
\end{align}
provided we change $\sigma$ to a suitably small value in terms of $\delta$.

We compare the graph $\cur{G_j \cap \Omega_{j,\tau,\sigma} \times (-A,A)}$
to the integral current $T_{j,A} = -\bdary \cur{W_{j,\tau,\sigma,A}}
+ \cur{G_j \cap \Omega_{j,\tau,\sigma} \times (-A,A)}$.
These have the same boundary, whence
$\calH^n(G_j \cap \Omega_{j,\tau,\sigma} \times (-A,A))	\leq \norm{T_{j,A}}(\bR^{n+1}).$
Substituting our term-by-term calculations into this inequality we find
\begin{multline}
	2A (\calH^{n-1}(F_T) - 2 \pi \sigma \calH^{n-2}(E_T) - \delta)
	- 2 A \delta \\
	\leq 2 \calH^n(\Omega) + 2 \delta 
	+ 2A(\calH^{n-1}(\bdary \Omega) - \calH^{n-1}(F_T))
	+ 2A \delta
	+ 4 \pi A \sigma \calH^{n-2}(E_T) + 2A \delta,
\end{multline}
whence after dividing by $2A$,
\begin{multline}
	\calH^{n-1}(F_T)
	- 2 \sigma \pi \calH^{n-2}(E_T)
	-2 \delta  \\
	\leq \calH^n(\Omega) / A 
	+ \calH^{n-1}(\bdary \Omega) - \calH^{n-1}(F_T)
	 + 2 \sigma \pi\calH^{n-2}(E_T) + \delta(2 + 1/A).
\end{multline}
This simplifies to 
\begin{equation}
	2 \calH^{n-1}(F_T)
	\leq \calH^n(\Omega) / A
	+ \calH^{n-1}(\bdary \Omega)
	+ 4 \sigma \pi \calH^{n-2}(E_T)
	+ \delta (4 + 1/A).
\end{equation}
The desired inequality follows after letting $A \to \infty$,
$\delta,\sigma,\tau \to 0$ and $j \geq J(\sigma,\tau,A,\delta) \to \infty$.
To conclude it only remains to justify
the two assumptions~\ref{item_domain_contained_in_sequence} and
\ref{item_divergence_away_from_spt_T_0}.

\ref{item_domain_contained_in_sequence}
After translating $\Omega$ we may assume that it contains the origin.
The convexity of $\Omega$ and the piecewise regularity of its boundary
that for  all $\tau > 0$ there is $\delta > 0$ so that
$\eta_{0,(1 + \delta)\#} \Omega \subset \Omega \setminus (\bdary \Omega)_\tau$
and $(\Omega)_\tau \subset \eta_{0,(1 + \delta)^{-1} \#} \Omega$.
Moreover as $\tau \to 0$ we may impose that $\delta \to 0$ also.
Given $\tau > 0$ take $J(\tau) \in \bN$ so that
$\Omega \setminus (\bdary \Omega)_\tau \subset \Omega_j \subset (\Omega)_\tau$.
Taking $\delta > 0$ as above we get that
$\eta_{0,1+\delta \#} \Omega \subset \Omega_j \subset \eta_{0,(1 + \delta)^{-1} \#} \Omega$.
After rescaling we find a sequence $(\eta_{0,1+\delta \#} \cur{G_j} \mid j \geq J(\tau))$
of single-valued minimal graphs respectively defined over
$\eta_{0,1+\delta \#} \Omega_j$.
Moreover $\eta_{0,1+\delta \#} \cur{G_j} \to \eta_{0,1+\delta \#}T$ as $j \to \infty$ and
$\eta_{0,1+\delta \#} T \to T$ as $\delta \to 0$. We may then diagonally
extract a subsequence of graphs $\cur{G_{j'}}$ and
find a sequence of positive scalars with $\delta_{j'} \to 0$ as $j' \to \infty$,
so that $\Omega \subset \eta_{0,1 + \delta_{j'}} \Omega_{j'}$ for all $j'$
and $\eta_{0,1+\delta_{j'} \#} \cur{G_{j'}} \to T$ as $j' \to \infty$.
Upon replacing our original sequence by this rescaled subsequence, we
may assume throughout that $\Omega \subset \Omega_j$ without restriction
of generality.

\ref{item_divergence_away_from_spt_T_0}
Let any $0 < \tau < \delta < 1 < A$ be given,
and consider the open set $\Omega' = \Omega
\setminus [\bdary \Omega \setminus \spt T_0]_\delta$.
By construction $\Omega' \setminus [\spt T_0]_\tau$ lies a distance at
least $\tau > 0$ away from $\bdary \Omega$, and thus also $\bdary \Omega_j$
for all $j$.
Were it not guaranteed that $\abs{u_j} > 2A$ in $\Omega' \setminus [\spt T_0]_\tau$
for large enough $j$, then there would exist a sequence of points
$X_{j'} \in G_{j'} \cap \Omega'  \times [-2A,2A] \setminus [\spt T]_\tau$
belonging to a subsequence of the two-valued graphs.
From this we can extract yet another subsequence so that
$X_{j''} \to X \in \clos{\Omega'} \times [-2A,2A] \setminus (\spt T)_\tau$
as $j'' \to \infty$.
By the monotonicity formula and upper semicontinuity of density $X \in \spt T$,
but that is manifestly absurd.
Replacing the original domain by $\Omega'$ and following the reasoning in the proof
yields $\calH^n(\spt( T_0 \mres \Omega')) \leq 1/2 \calH^n(\bdary \Omega')$.
Letting $\delta \to 0$ one finds the desired conclusion in terms of $\Omega$.
\end{proof}

The following special case is of particular importance in what follows.
Let $\pi = \pi_0 \times \bR e_{n+1},\pi' = \pi_0' \times \bR e_{n+1}$ be
two $n$-dimensional half-planes meeting along an axis $L = L_0 \times \bR e_{n+1}$
at which they form a positive angle $0 < \theta < \pi$, taken in the counterclockwise
direction.
Let $N,N'$ be their respective unit normals, which we both take pointing
in the counterclockwise direction.
Further let $p,p' \in L^\perp$ be the two unit vectors so that
$\pi = \{ Y + t p \mid Y \in L, t \geq 0 \}$ and 
$\pi' = \{ Y + t p' \mid Y \in L, t \geq 0 \}$.
 Any point in $\bR^n$ can be written $x = y + z = y + tp + t'p$
 with $y \in L_0, z \in L_0^\perp$.
Define $Q = \{  y + z \in \bR^n \mid \abs{y} < 1, \abs{z} < 1 \}$ and
the wedge-shaped region $V = \{ x \in Q \mid \langle x , N \rangle >0,
\langle x , N' \rangle < 0 \}$.

\begin{lem}
\label{lem_no_folding_two_half_planes}
Let $\pi,\pi'$ and $V \subset \bR^n$ be as above,
and let the current $T = \cur{(\pi \cup \pi') \cap \bdary V \times \bR}$
be oriented inward.

Then there does not exist a sequence of minimal graphs $G_j = \graph u_j$
defined over domains $\Omega_j$ with $\dist_{\calH}(\Omega_j,V) \to 0$,
$\cur{G_j} \to 0$ and so that for all $0 < \sigma < 1 < A$ there is $J = J(\sigma,A) \in \bN$
so that $\spt \bdary \cur{G_j \mres \Omega_j \times (-A,A)} \subset ((\pi \cup \pi')
\cap \bdary V)_\sigma$ for all $j \geq J$.
\end{lem}

\begin{proof}
This is essentially a direct consequence of~Theorem~\ref{thm_no_folding}, 
although we need to construct a subdomain $\Delta_a \subset V$ to make the
area comparison work, where $0 < a < 1/2$ is a small parameter whose value we leave
undetermined for now.
Let
$\Delta_a = \{ y + tp + t' p' \in V \mid \abs{y} < 1/2, t,t' > 0, t + t' < a \}$,
which is convex and has piecewise smooth boundary.
When we fix $y_0 \in L_0$ with $\abs{y_0} < 1/2$ then $\Delta_a \cap \{ x = y_0 + z \}$
is an isosceles triangle with two sides of length $a$.
Moreover, apart from $\pi \cap \bdary \Delta_{a}$ and $\pi' \cap \bdary \Delta_{a}$,
the boundary of $\Delta_{a}$ contains only two subsets $\Gamma_1,\Gamma_2$
with positive area,
namely  $\Gamma_1 = \{ \abs{y} = 1/2 \} $ and $\Gamma_2 = \{ \abs{y} < 1/2, t + t' = a\}$.
On the one hand
\begin{align}
	\calH^{n-1}(\Gamma_1) &= a^2 (n-2) \omega_{n-2} \sin( \theta) 2^{-n+2}, \\
	\calH^{n-1}(\Gamma_2) &= a\omega_{n-2} 2^{-n+3} \sin (\theta/2),
\end{align}
and on the other hand
\begin{equation}
	\calH^{n-1}(\pi \cap \bdary \Delta_{a}) = a \omega_{n-2} 2^{-n+2}
	= \calH^{n-1}(\pi' \cap \bdary \Delta_{a}).
\end{equation}
Comparing the two we have
\begin{multline}
	\calH^{n-1}(\Gamma_1 \cup \Gamma_2) =
	a \omega_{n-2} 2^{-n+3} ( (n-2) a \sin (\theta) / 2 +  \sin (\theta/2)) \\
	< a \omega_{n-2} 2^{-n+3} 
	= \calH^{n-1}(\pi \cup \pi' \cap \bdary \Delta_{a})
\end{multline}
provided $a$ is small enough.
If there were a sequence of minimal graphs of $u_j \in C^2(\Omega_j)$ as
in the statement, then by restricting them to $\Omega_j \cap \Delta_a$
and letting $j \to \infty$ we would obtain a contradiction to Theorem~\ref{thm_no_folding}.
\end{proof}


\section{Multiplicity and branch points of limit cones}
\label{sec_multiplicity_limit_cones}

Let $\alpha \in (0,1)$ and $(u_j \mid j \in \bN)$ be a sequence of
two-valued minimal graphs with $u_j \in C^{1,\alpha}(D_2;\calA_2)$.
Here we examine the situation in which these graphs converge to a plane
weakly in the varifold topology, $\abs{G_j} \to m \abs{\Pi}$ as $j \to \infty$,
where $\Pi \in \Gr(n,n+1)$ and $m \in \bZ_{>0}$.

\subsection{An a priori multiplicity bound}
\label{subsec_a_priori_multiplicity_bound}

\begin{lem}
\label{lem_multiplicity_bound}
If $\abs{G_j} = \abs{\graph{u_j}} \to m \abs{\Pi}$ then $m \leq 2$.
\end{lem}
\begin{proof}
The proof is slightly easier when the plane $\Pi \in \Gr(n,n+1)$ is vertical,
that is of the form $\Pi = \Pi_0 \times \bR e_{n+1}$.
Let $\eps > 0$ be a given, arbitrarily small constant.
Arguing as in the proof of Proposition~\ref{prop_area_estimate} we may take
$J(\eps) \in \bN$ so that $\int_{D_1} \frac{1}{v_j}
\indic_{\abs{u_j} < 1} \leq \eps$ when $j \geq J(\eps)$.
Thus for large $j$, $\calH^n(G_j \cap D_1 \times (-1,1))
\leq \eps + \int_{D_1} \frac{\abs{Du_j}^2}{v_j} \indic_{\abs{u_j} < 1}$.
Let $U \subset D_1$ be a tubular neighbourhood of $\Pi_0$ inside the disc,
narrow enough that $\Per(U) \leq 2 \calH^{n-1}(\Pi_0 \cap D_1) + \delta$.
We may find a function
$\eta \in C_c^1(D_2)$ with $\eta \equiv 1$ on $U$ and $\int_{D_2} \eta
\leq \Per(U) + \delta \leq 2 \calH^{n-1}(\Pi_0 \cap D_1) + 2 \delta$.
Updating $ j \geq J(\eps,\delta$ so that $G_j \cap D_1 \times (-1,1) \subset
U \times (-1,1)$ we get $\calH^n(G_j  \cap D_1 \times (-1,1))
\leq \eps + 4( \calH^{n-1}(\Pi_0 \cap D_1) + \delta)$.
Letting $\eps,\delta \to 0$ and $j \geq J(\tau,\eps) \to \infty$ we find
$m \calH^n(\Pi \cap D_1 \times (-1,1)) \leq 4 \calH^{n-1}(\Pi_0)$,
which gives the desired conclusion.

Now for the case where the plane $\Pi$ is not vertical.
There nothing to prove if $m = 1$, so we may assume that $m \geq 2$.
Let a small constant $0 < \tau < 1$ be given, and take $j \geq J(\tau)$
large enough that inside the cylinder $G_j \cap D_2 \times \bR \cap (\Pi)_1
\subset (\Pi)_\tau$.
In fact we have the same control over $G_j$ in the whole cylinder,
that is $G_j \cap D_2 \times \bR \subset (\Pi)_\tau$.
Indeed, if this were to fail then $G_j \cap D_2 \times \bR$ would be
disconnected, and we could write $G_j \cap D_2 \times \bR = \Gamma_{j,1} \cup \Gamma_{j,2}$
where
$\Gamma_{j,1} \cap (\Pi)_\tau \neq \emptyset$ and
$\Gamma_{j,2} \cap  (\Pi)_{1} = \emptyset$.
But then $\calH^n(G_j \cap D_2 \times \bR \cap (\Pi_1)_1) \leq \calH^n(\Gamma_{j,1})$,
and taking limits as $j \to \infty$ would yield $m = 1$.
This is absurd as we initally assumed that $m$ is at least two, and hence
we have confirmed that  $G_j \cap D_2 \times \bR \subset (\Pi)_\tau$.
Let $L = \max \{ X^{n+1} \mid X = (x,X^{n+1}) \in \Pi \cap D_2 \times \bR \}$,
then $\norm{u_j(x)} \leq 2 (\tau + L)$ for all $x \in D_2$. Using the interior
gradient estimates, we find that there is a constant $C = C(n,L)$ so that
eventually $\norm{u_j}_{1;D_1} \leq C$.
Up to extracting a subsequence we find that the $u_j$ converge to a two-valued
Lipschitz graph defined on $D_1$. As by assumption $\abs{G_j} \to m \abs{\Pi}$
we can conclude that $m = 2$.
\end{proof}

\subsection{Multiplicity in limit varifolds}

Combining the previous lemma with a diagonal extraction argument, we
obtain the following result.

\begin{cor}
Let $\alpha \in (0,1)$, and let $(u_j \mid j \in \bN)$ be a sequence of
two-valued minimal graphs in $C^{1,\alpha}(D_2;\calA_2)$.
Suppose that there are half-planes $\pi_i$ and $m_i \in \bZ_{>0}$ so that
$\abs{G_j} \to \sum_i m_i \abs{\pi_i}$. Then $m_i \leq 2$.
\end{cor}

Similarly, though in a more general context, we can combine the estimate
from Lemma~\ref{lem_multiplicity_bound} with the work of
Krummel--Wickramasekera~\cite{KrumWic_FinePropsMinGraphs}.

\begin{cor}
\label{cor_conseq_multiplicity_bound}
Let $G_j = \graph u_j$ be a sequence of two-valued minimal graphs, where
$u_j \in C^{1,\alpha}(D_2;\calA_2)$ for all $j$ for some $\alpha \in (0,1)$.
Suppose that $\abs{G_j} \to V \in \IV_n(D_2 \times \bR)$ weakly in the
topology of varifolds. Then for all $Z \in \reg V$,
$\Theta(\norm{V},Z) \leq 2$.
If $Z \in \calB(V)$ then $\Theta(\norm{V},Z) = 2$, and the branch set
is countably $n-2$-rectifiable.
\end{cor}

\subsection{Local description near vertical planes}

We return to the situation where $\abs{G_j} \to 2 \abs{\Pi}$ to some
vertical plane $\Pi = \Pi_0 \times \bR e_{n+1} \in \Gr(n,n+1)$.
The limit in the current topology is supported in the same plane,
$\cur{G_j} \to l \cur{\Pi}$ for some non-negative $l \in \bZ$
with $l \leq 2$.
By Allard's regularity theorem the multiplicity is either $l = 0$ or $2$.
The following result considers the case where the mass of the currents vanishes
in the limit.

\begin{lem}
\label{lem_no_branch_points_if_cancellation}
Let $\alpha \in (0,1)$, $u_j \in C^{1,\alpha}(D_2;\calA_2)$ be a
sequence of two-valued minimal graphs, and let $\Pi = \Pi_0 \times \bR e_{n+1}$
be a vertical plane.
Suppose $\abs{G_j} \to 2 \abs{\Pi}$ and $\cur{G_j} \to 0$ as $j \to \infty$.
Then $\calB({G_j}) \cap B_1 = \emptyset$ for large $j$.
\end{lem}
\begin{proof}
Let $N$ be either unit normal to $\Pi$, and define $f_j = \langle \nu_j, N \rangle$
on $\reg G_j$. Given any $\delta > 0$, $f_j$ restricted to $\reg G_j \cap B_{3/2}$
takes values in $[-1,-1+\delta) \cup (1-\delta,1]$ by~\cite{Wic_MultTwoAllard},
at least for large enough $j \geq J(\delta)$.
Write $\calR_j$ for the connected components of $\reg G_j \cap B_{3/2}$, which
we further divide into $\calR_{j}^\pm$ according to the sign of $f_j$. 
Accordingly we may decompose $\cur{G_j} = T_j^+ + T_j^-$ into the sum of the
two currents obtained by integrating over $\calR_j^\pm$ respectively, in a
way that $\cur{G_j} = T_j^+ + T_j^-$ and $\abs{G_j} = \abs{T_j^+} + \abs{T_j^-}$.
The two currents only meet along classical, immersed singularities of $G_j$, where
they moreover intersect transversely. Therefore they are both separately stationary
with $\bdary T_j^\pm = 0$ in $B_{3/2}$.
By assumption $T_j^+ + T_j^- \to 0$ and $\abs{T_j^+} + \abs{T_j^-} \to 2 \abs{\Pi} \mres
B_{3/2}$ as $j \to \infty$ in the current and varifold topologies respectively.
Moreover by Federer--Fleming compactness $T_j^\pm \to T^\pm$ separately as $j \to \infty$.
The limit currents satisfy $T^+ + T^- = 0$, and thus they are equal to the plane
$\Pi$ with multiplicity, but with opposite orientations.
By Allard regularity both $T_j^+ \mres B_1$ and $T_j^- \mres B_1$ can be
written as smooth graphs defined on $\Pi$, and thus do not support any branch points.
\end{proof}

Return to the general case,  where $\abs{G_j} \to m \abs{\Pi}$ and
$\cur{G_j} \to l \cur{\Pi}$ for some vertical plane $\Pi = \Pi_0 \times \bR e_{n+1}$.
If $l \neq 0$ then we let $N$ be the unit normal to $\Pi$ corresponding to
the orientation induced on the plane by $\cur{G_i}$, and if $l = 0$ then
we pick our orientation arbitrarily.
Thus we can divide $D_1 \setminus \Pi_0 \subset \bR^n$ into two connected components
$D_1^{\pm} = \{ x \in D_1 \mid \pm \langle x , N \rangle > 0 \}$.
For each $j$ define a function $F_j: D_1 \to \{ 0, 1,2 \}$ by
\begin{equation}
	\label{eq_defn_number_of_sheets_function}
	F_j(x)
	= \sum_{\substack{X \in P_0^{-1}(\{ x \})\\ X^{n+1} < -1}} \Theta(\norm{G_j},X).
\end{equation}
This returns the number of points in $G_j$ which lie below $x \in D_1$, counted
with multiplicity. (We could equally well have worked with a function counting
the points lying above $x \in D_1$, although formulas
such as~\eqref{eq_difference_number_of_sheets} would have the opposite sign.)
These functions are eventually locally constant away from the plane $\Pi_0$,
in the following sense.
\begin{claim}
\label{claim_function_F_j_loc_constant}
Let $\tau > 0$ be arbitrary. Then there is $j \geq J(\tau)$
so that $F_j$ is constant on the two components $D_1^\pm \setminus (\Pi_0)_\tau$.
\end{claim}
\begin{proof}
The proof is identical for both components, so we just work with $D_1^+$.
By the convergence of the graphs $G_j$ in the Hausdorff distance, we may
take $j \geq J(\tau)$ large enough that
$G_j \cap D_1 \times (-1,1) \subset (\Pi)_\tau$.
Hence eventually $G_j \cap D_1^+ \times \bR \setminus (\Pi)_\tau
\subset \{\abs{ X^{n+1}} > 1 \}$. 

There are three possibilities:
\begin{enumerate}
	\item either $G_j \cap D_1^+ \times \bR \setminus (\Pi)_\tau \cap \{ X^{n+1} < -1 \}
	= \emptyset$,
\item or $G_j \cap D_1^+ \times \bR \setminus (\Pi)_\tau \subset \{ X^{n+1} < -1\}$,
\item or $G_j \cap D_1^+ \times \bR \setminus (\Pi)_\tau$ contains points
	with positive and negative values for $X^{n+1}$.
\end{enumerate}
Going through these cases in the same order we find that at all points $x \in D_1^+
\setminus (\Pi)_\tau$ the function $F_j(x)$ takes the values $0,2$ or $1$.
\end{proof}

\begin{lem}
\label{lem_description_near_planes}
Let $\alpha \in (0,1)$, $u_j \in C^{1,\alpha}(D_2;\calA_2)$ be
a sequence of two-valued minimal graphs, and
$\Pi = \Pi_0 \times \bR e_{n+1}$ be a vertical plane.
Suppose that $\abs{G_j} \to m \abs{\Pi}$ and $\cur{G_j} \to l \cur{\Pi}$
as $j \to \infty$.
Then the $F_j$ are eventually constant away from $\Pi_0$, taking the
values $F^\pm$ on $D_1^\pm$ respectively, and
\begin{equation}
	\label{eq_difference_number_of_sheets}
	F^+ - F^-  = l.
\end{equation}
\end{lem}
\begin{proof}
There are three possible cases:
\begin{enumerate}
	\item \label{item_mult_two_no_mass_cancel}
	either $m = 2$ and $l = 2$,
	\item \label{item_mult_two_mass_cancel}
	or $m = 2$ and $l = 0$,
	\item \label{item_mult_one}
	or $m = 1$ and $l = 1$.
\end{enumerate}
The proof is basically the same in all three cases, so we only consider
the first. Consider the line $l \subset \bR^n \times \{ 0 \}$ directed
by $N$ and passing through the origin.
Identify $l$ with $\bR$ via a unit-speed parametrisation.
Then there exist two functions $u_{j,l}^1,u_{j,l}^2 \in C^1(\bR)$
so that $u_j(tN) = \{ u_{j,l}^1(tN),u_{j,l}^2(tN) \}$ for all $t \in \bR$.
Moreover, as $l = 2$ we get that $u_{j,l}^1(t) \wedge u_{j,l}^2(t) < - 1$
on $(1/2,1)$ and $u_{j,l}^2(t) \vee u_{j,l}^2(t) > 1$ for $t \in (-1,-1/2)$,
provided $j$ is large enough.

Now let $0 < \tau < 1$ be an arbitrary small constant.
By Claim~\ref{claim_function_F_j_loc_constant}, the function $F_j$
is constant on the two components of $D_1 \setminus (\Pi)_\tau$,
at least provided $j \geq J(\tau)$ is chosen large enough.
Combining this with our calculations above, we find that 
for points $x \in D_1 \setminus (\Pi)_{\tau}$, $F_j(x) = 2$ if $x \in D_1^+$
and $F_j(x) = 0$ if $x \in D_1^-$.
These values do not change with large values of $j$, and we may set
$F^+ = 2, F^- = 0$, which confirms that indeed $F^+ - F^- = 2 = l$.
As explained above, the other cases can be argued similarly.
\end{proof}

\section{Classical limit cones: initial analysis}
\label{sec_classical_limit_cones}

Let $\alpha \in (0,1)$, and $(u_j \mid j \in \bN)$ be a sequence of two-valued
minimal graphs, with $u_j \in C^{1,\alpha}(D_2;\calA_2)$ for all $j$.
We assume that they converge to a \emph{classical cone} in the varifold
topology, say $\abs{G_j} \to \bP$.
By this we mean that there are $n$-dimensional half-planes $\pi_1,\dots,\pi_N$
meeting along a commmon $n-1$--dimensional axis $L \in \Gr(n-1,n+1)$ and
integers $m_1,\dots,m_N \in \bZ_{>0}$ so that $\bP = \sum_i m_i \abs{\pi_i}$.
By the graphs are endowed with the orientation corresponding to their
upward-pointing unit normal we obtain a sequence of currents which we may
also assume convergent, say $\cur{G_j} \to T \in \I_n(D_2 \times \bR)$,
extracting a subsequence if necessary.
This limit too has a similar form to the above, namely $T = \sum_i l_i \cur{\pi_i}$
where $0 \leq l_i \leq m_i$. Here the half-planes are given the orientations
induced by $T$, where we pick an arbitrary orientation for those $\pi_i$ which
have $l_i = 0$.
Our main theorem in this section is the following.

\begin{thm}
\label{thm_classical_cone_support_as_desired}
 Let $\alpha \in (0,1)$, and let $(u_j \mid j \in \bN)$ be a sequence of two-valued
minimal graphs with $u_j \in C^{1,\alpha}(D_2;\calA_2)$.
Suppose that $\abs{G_j} \to \bP$ and $\cur{G_j} \to T$ as $j \to \infty$,
where $\bP$ and $T$ are classical cones.
Then there exist planes $\Pi_1,\dots,\Pi_D \in \Gr(n,n+1)$
and integers $0 \leq l_i\leq m_i \leq 2$ so that
\begin{align}
	\label{eq_limit_cone_actual_form}
	\bP = \sum_{i=1}^D m_i \abs{\Pi_i} \quad \text{ and } \quad
	T = \sum_{i=1}^D l_i \cur{\Pi_i}.
\end{align}
\end{thm}

Our proof follows an approach adapted from the work of Schoen--Simon
in~\cite{SchoenSimon81}. 
In preparation for this, we introduce an orthogonal decomposition of $\bR^{n+1}$
adapted to $L$, writing $\bR^{n+1} = \{ Y + Z \mid Y \in L, Z \in L^\perp \}$.
In keeping with this, we call the two-dimensional affine space
$\{ Y + Z \mid Z \in L^\perp \}$ obtained by fixing a point $Y \in L$ a \emph{slice}
through $Y$.
In this notation the cone $\bP$ is suported in
$\{ Y + \sum_{i=1}^N t_i p_i \mid Y \in L, t_i \geq 0, i=1,\dots,N \}$,
where the $p_i \in L^\perp$ are unit vectors so that
$\pi_i = \{ Y + t p_i \mid Y \in L, t \geq 0 \}$.
Letting $R_i$ be the ray generated by $p_i$, we have
$\bP = L \times \sum_{i=1}^N m_i \abs{R_i}$.

Let a small $\sigma > 0$ be given. For sufficiently large $j \geq J(\sigma)$,
\begin{equation}
G_j \cap \{ Y + Z \mid Z \in L^\perp, \sigma/2 < \abs{Z} < \sigma \}
= \cup_{k=1}^M \gamma_{j,Y}^k
\end{equation}
for each point $Y \in L$ with $\abs{Y} \leq 1$, where $M = \sum_i m_i$.
The $\gamma_{j,Y}^k$ are $C^{1,\alpha}$ embedded Jordan arcs with endpoints in
$\{ Y  + Z  \mid Z \in L^\perp, \abs{Z} = \sigma/2 \text{ or } \sigma \}$.
Moreover as $j \to \infty$ we have the uniform limits
\begin{equation}
	\label{eq_uniform_convergence_hausdorff_distance}
	\dist_{\calH}(\cup_k \gamma_{j,Y}^k,
	\cup_{i=1}^N \{ Y + t p_i \mid \sigma/2 < t < \sigma \}) \to 0
\end{equation}
and for every $k$,
\begin{equation}
	\label{eq_uniform_convergence_of_the_normals}
	\min_{i \in \{ 1, \dots, N \}} \sup_{X \in \gamma_{j,Y}^k}
	 \langle \nu_j(X), N_i \rangle  \to 0,
\end{equation}
for all $Y \in L$ with $\abs{Y} \leq 1$, where the $N_i$ are the normal
vectors to the $\pi_i$. The latter is justified using either Allard regularity
or the branched sheeting theorem of Wickramasekera,
quoted in Theorem~\ref{thm_wic_mult_two_allard}, depending on the multiplicity of the ray.

Let $0 < \tau < 1$ be chosen small enough in terms of $\sigma$
that the tubular neighbourhoods
$\big( \{ t p_i \mid \sigma/2 < t < \sigma \} \big)_\tau$
are two-by-two disjoint in $L^\perp$.
Next taking $j \geq J(\tau,\sigma)$ large enough that
$\dist_{\calH}(\cup_{k=1}^M \gamma_{j,Y}^k, \cup_{i=1}^N \{ Y + t_i p_i
\mid \sigma/2 < t_i < \sigma \}) < \tau$ we ensure that there lie
precisely $m_i$ Jordan arcs near every line segment
$\{ t p_i \mid \sigma/2 < t < \sigma \}$.

By~\cite{KrumWic_FinePropsMinGraphs}---whose results we quote in
Theorem~\ref{thm_krumwic_finepropsmingraphs}---the branch set of $G_j$
is countably $(n-2)$-rectifiable for all $j$.
Let $P_L$ be the orthogonal projection onto $L$. Then the projection of
$B_{G_j} \cap \{
	\abs{Y} \leq 1, \abs{Z} \leq \sigma \}$
is a compact subset of $L$ with
\begin{equation}
	\label{eq_unbranched_points_full_measure}
	\calH^{n-1}(P_{L}(\calB_{G_j} \cap \{
		\abs{Y} \leq 1 , \abs{Z} \leq \sigma\})) = 0.
\end{equation}
Together with Sard's theorem, we find that there is an open subset $\calU_{j,\sigma}
\subset L \cap \{ \abs{Y} \leq 1 \}$ with full $\calH^{n-1}$-measure of points
$Y \in \calU_{j,\sigma}$ we call \emph{unbranched} for which
\begin{equation}
	G_j \cap \{ Y + Z \mid \abs{Z} < \sigma \} 
	= \cup_{k=1}^P \Upsilon_{j,Y}^k \cup \cup_{l=1}^Q \Delta_{j,Y}^l,
\end{equation}
where the $\Upsilon_{j,Y}^k$ are smooth properly embedded Jordan arcs with
endpoints in $\{ Y + Z \mid \abs{Z} =  \sigma \}$ and
the $\Delta_{j,Y}^l$ are smooth properly embedded Jordan curves.
%
%
Given a small $\kappa > 0$, we also single out those $Y \in \calU_{j,\sigma}$ for which
$\int_{G_j \cap \{ Z + Y \mid \abs{Z} < \sigma \}} \abs{A_{G_j}} \intdiff \calH^1 < \kappa$;
these form a set denoted $\calU_{j,\sigma}(\kappa)$.
The stability of the $G_j$ forbids curvature concentration near the axis $L$,
from whence one obtains the following; this is a point made in~\cite{SchoenSimon81}.
We use a similar result at a later stage, Lemma~\ref{lem_Uj_kappa_complement_goes_to_zero},
valid under slightly different hypotheses. There we also give a detailed argument,
whose steps are easily transcribed to the present context.

\begin{lem}
	\label{lem_orthogonal_slicing_U_j_and_U_j_kappa}
For all $\kappa > 0$, $\calH^{n-1}(\calU_{j,\sigma} \setminus \calU_{j,\sigma}(\kappa))
\to 0$ as $\sigma \to 0$ and $j \geq J(\sigma) \to \infty$.
\end{lem}

This is the key ingredient in the proof of
Theorem~\ref{thm_classical_cone_support_as_desired}, which we give now.
\begin{proof}
We list four ways in which the cones $\bP$ and $T$ can fail to conform
to the conclusion of the theorem, each assuming the negation of the preceding:
\begin{enumerate}[label = (\arabic*)]
\item \label{item_first_spt_P_not_a_union_of_planes}
$\bP$ is not supported in a union of planes,

\item \label{item_second_P_not_sum_of_planes}
$\bP$ is not a sum of planes,

\item \label{item_third_T_is_not_supported_in_a_union_of_planes}
	$T$ is not supported in a union of planes,
\item \label{item_fourth_current_T_is_not_a_sum_of_planes}
	$T$ is not a sum planes.
\end{enumerate}

In each of the four cases we prove the existence of constants
$\varphi_1,\varphi_2,\varphi_3 > 0$ so that
$U_{j,\sigma}(\varphi_i)$ is empty,
independently of $\sigma$ and $j \geq J(\sigma)$.
Specifically we show that for all $Y \in \calU_{j,\sigma}$
there is a curve $\Upsilon \in \{ \Upsilon_{j,Y}^1,\dots,
\Upsilon_{j,Y}^P \}$ along which the unit normal $\nu_j$ varies by a positive
amount, say $2\varphi_i \leq \sup_{X,Y \in \Upsilon} \abs{ \nu_j(X) - \nu_j(Y) }$.
Next let $\varphi$ be the minimum of $\varphi_1,\varphi_2,\varphi_3$. Identifying 
$\Upsilon$ with a smooth parametrisation on $[0,1]$, there are $0 < s_j < t_j < 1$
so that $ \varphi  \leq \abs{ \nu_j(\Upsilon(s_j)) - \nu_j(\Upsilon(t_j))}$.
Then 
\begin{align}
	\varphi
	&\leq \int_{s_j}^{t_j} \abs{ (\nu_j \circ \Upsilon)'(t)} \intdiff t
	= \int_{s_j}^{t_j} \abs{ \langle \nabla_{G_j} \nu_j(\Upsilon(t)),
	\Upsilon'(t) \rangle } \intdiff t \\
	&\leq \int_{s_j}^{t_j} \abs{A_{G_j}}(\Upsilon(t)) \abs{\Upsilon'(t)} \intdiff t 
	\leq \int_{\reg G_j \cap \{ Y + Z \mid Z \in L^\perp, \abs{Z} < \sigma\}}
	\abs{A_{G_j}} \intdiff \calH^1.
\end{align}
Hence $Y \in \calU_{j,\sigma} \setminus \calU_{j,\sigma}(\varphi)$, and
as this was chosen arbitrarily we find that $\calU_{j,\sigma}(\varphi)$ is empty.
We may then let $\sigma \to 0$ and $j \geq J(\sigma) \to \infty$ to obtain
a contradiction with Lemma~\ref{lem_orthogonal_slicing_U_j_and_U_j_kappa}.

\ref{item_first_spt_P_not_a_union_of_planes}
If the cone $\bP$ is no supported in a union of planes, then we can relabel
the rays to arrange $-p_1 \in \{ p_1,\dots,p_M \}$.
Let $\Upsilon \in \{ \Upsilon_{j,Y}^1,\dots,\Upsilon_{j,Y}^P \}$ be an arc
with at least one endpoint near the ray $R_1$. If its other endpoint lies near
another ray $R_i$, then picking a point $X$ near one endpoint and $Y$ near another
we find $2 \varphi_1 \leq \sup_{X,Y \in \Upsilon} \abs{\nu_j(X) - \nu_j(Y)}$, where
$\varphi_1$ depends only on the angle between the two rays. Similarly if both endpoints
lie near $R_0$  then $1 \leq \sup_{X,Y \in \Upsilon} \abs{\nu_j(X) - \nu_j(Y) }$.

\ref{item_second_P_not_sum_of_planes}
Thus we may assume that $\spt \norm{\bP} = \cup_{i=1}^D \Pi_i$,
and we can relabel the half-planes so that $\Pi_i = \pi_i \cup \pi_{i+D}$ for all $i$,
taking the indices module $2D = N$.
The only way for $\bP$ not to be a sum of these planes (with some multiplicities)
is if $m_i \neq m_{i+D}$ for some $i$, say $m_1 \neq m_{D+1}$ after relabelling.
A pigeonhole argument demonstrates the existence of an arc
$\Upsilon \in \{ \Upsilon_{j,Y}^1,\dots,\Upsilon_{j,Y}^P \}$ satisfying one of
the following.
Either $\Upsilon$ has both endpoints near $R_1$ or $R_{D+1}$, or it has one
endpoint one of $R_1,R_{D+1}$ and the other near a third ray $R_i$.
Combining the two cases we find $\min \{ 1,2\varphi_2 \}
\leq \sup_{X,Y} \abs{\nu_j(X) - \nu_j(Y) }$, where $\varphi_2$ depends only
on the respective angles that $R_i$ forms with $R_1,R_{D+1}$.

\ref{item_third_T_is_not_supported_in_a_union_of_planes}
Hence we may take $\bP = \sum_{i=1}^D m_i \abs{\Pi_i}$, and if
$\abs{T} = \sum_{i=1}^D (l_i \abs{\pi_i} + l_{i+D} \abs{\pi_{i+D}})$
is not a sum of planes then $l_{i} \neq l_{i+D}$. Relabelling the half-planes
once again if necessary, we may assume that $l_1 = 0$ and $l_{D+1} = 2$,
inequivalent cases being excluded by Allard regularity.
Then $m_1 = 2 = m_{D+1}$, and for large enough $j$ we can argue as above
to prove the existence of two arcs $\Upsilon_-,\Upsilon_+
\in \{ \Upsilon_{j,Y}^1,\dots,\Upsilon_{j,Y}^P \}$ lying near the rays $R_{1},R_{D+1}$,
lest $\nu_j$ vary by a positive amount along the arcs containing the
portions of $G_j$ near them.
We label these so that $\pm \langle \nu_j, N_1 \rangle$ is positive on
$\Upsilon_{\pm}$ near $R_1$. This therefore changes sign along $\Upsilon_-$,
and $1 \leq \sup_{X,Y \in \Upsilon_-} \abs{\nu_j(X) - \nu_j(Y)}$.

\ref{item_fourth_current_T_is_not_a_sum_of_planes}
For the last step we may assume that $T = \sum_{i=1}^D l_i (\cur{\pi_i}
+ \cur{\pi_{i+D}})$ and it remains to show that the pairs $\pi_i,\pi_{i+D}$
are oriented in a consistent fashion.
Were this to fail for $\pi_1,\pi_{D+1}$ for example then we could once again
find an arc $\Upsilon \in \{ \Upsilon_{j,Y}^1,\dots,\Upsilon_{j,Y}^P \}$
with either one endpoint near one of $R_{1},R_{D+1}$
and the other near a third ray, or both endpoints near $R_{1},R_{D+1}$ but
along which $\langle \nu_j , N_1 \rangle$ changes sign. This concludes the proof.
\end{proof}

\section{Classical limit cones: non-vertical cones}
\label{sec_non_vertical_cone}

Let $\alpha \in (0,1)$ and $(u_j \mid j \in \bN)$ be a sequence of two-valued
minimal graphs with $u_j \in C^{1,\alpha}(D_2;\calA_2)$. Let $D \in \bZ_{> 0}$
and for $i = 1,\dots,D$ let $\Pi_i \in \Gr(n,n+1)$ be $n$-dimensional planes
which meet along a common axis $L \in \Gr(n-1,n+1)$.
Recall that we call a plane $\Pi$ vertical if it is of the form
$\Pi = \Pi_0 \times \bR e_{n+1}$.
Here we consider the situation where they are not all vertical, 
and we may relabel them to arrange for $\Pi_1$ to be non-vertical.
Assume that $\abs{G_j} \to \sum_{i=1}^D m_i \abs{\Pi_i} = \bP$ and
$\cur{G_j} \to \sum_{i=1}^D l_i \cur{\Pi_i} = T$ as $j \to \infty$,
where $0 \leq l_i \leq m_i$.

\subsection{Slicing at an acute angle}

Broadly speaking we again use an approach based on~\cite{SchoenSimon81}.
However from a technical standpoint the arguments from the previous section are
maladapted to the current situation, as we want to exploit the two-valued
graphicality of the $G_j$.
Instead of taking slices orthogonal to the axis $L$ of the cone, we proceed
as follows.
Let $v \in \bR^{n+1}$ be a unit vector with $\langle v , e_{n+1} \rangle  = 0$,
so that $v, e_{n+1}, L$ span $\bR^{n+1}$. Let $V = \linspan \{ v , e_{n+1} \}$,
and write $Z = tv + z e_{n+1} \in V$.
Then $\bR^{n+1} = L + V$, and every point $X \in \bR^{n+1}$ can uniquely
by written $X = Y + Z = y + Y^{n+1} e_{n+1} + tv + z e_{n+1}$.
We emphasise that in general $L$ and $V$ do not meet at a right angle, 
and this decomposition is not orthogonal. (The exception being the case
where $L = \bR^n \times \{ 0 \}$.)
The slices we take are adapted to this decomposition,
using sets of the form $\{ Y + Z \mid Z \in V, \abs{Z} < \sigma \}$,
where $0 < \sigma < 1$ and $Y = (y,Y^{n+1}) \in L$ is a fixed point
with $\abs{y} < 1$.
We abbreviate this $\{ Y+ Z \mid \abs{Z} < \sigma \}$, and in the
same vein write for example $\{ \abs{y} < 1 , \abs{Z} < \sigma \}
= \{ Y + Z \mid Y \in L, Z \in V, \abs{y} < 1 , \abs{Z} < \sigma \}$.

The corresponding projection map is written $Q: X = Y + Z \in \bR^{n+1}
\mapsto Y$. Analogous to~\eqref{eq_unbranched_points_full_measure} we have
$\calH^{n-1}\big(Q (\calB(G_j) \cap \{ \abs{y} \leq 1, \abs{Z} \leq 1 \})\big) = 0$.
Let $\sigma > 0$ be given. Together with Sard's theorem we find that for all $j$
there is an open subset $\calV_{j,\sigma} \subset L \cap \{ \abs{y} \leq 1\}
\setminus Q (\calB(G_j) \cap \{ \abs{y} \leq 1, \abs{Z} \leq 1 \})$ of full measure,
so that for all $Y \in \calV_{j,\sigma}$,
$G_j \cap \{ Z + Y \mid  \abs{Z} < \sigma \}= \cup_{k=1}^P \Upsilon_{j,Y}^k$
where the  $\Upsilon_{j,Y}^k$ are smooth properly embedded Jordan arcs with endpoints
in the set $\{ Y + Z \mid  \abs{Z} = \sigma \}$.
(This cannot contain any closed curves because of the graphicality of $G_j$.)
In the same vein, given $\kappa > 0$ we write
$\calV_{j,\sigma}(\kappa) \subset \calV_{j,\sigma}$
for the measurable subset of $\calV_{j,\sigma}$ formed by those points
$Y \in \calV_{j,\sigma}$ with
$\int_{\reg G_j \cap \{ Z + Y \mid \abs{Z} < \sigma \} }
\abs{A_{G_j}} \intdiff \calH^1 < \kappa.$

\begin{lem}
	\label{lem_Uj_kappa_complement_goes_to_zero}
	For all $\kappa > 0$, $\calH^{n-1}(\calV_{j,\sigma}
	\setminus \calV_{j,\sigma}(\kappa)) \to 0$ as $\sigma \to 0$ and
	$j \geq J(\sigma) \to \infty$.
\end{lem}

\begin{proof}
By definition of $\calV_{j,\sigma}(\kappa)$, we can integrate over points
$Z \in \calV_{j,\sigma} \setminus \calV_{j,\sigma}(\kappa)$ to obtain the inequality
\begin{multline}
	\label{eq_lower_bound_kappa}
	\calH^{n-1}(\calV_{j,\sigma} \setminus \calV_{j,\sigma}(\kappa)) \kappa \\
	\leq \int_{\calV_{j,\sigma} \setminus \calV_{j,\sigma}(\kappa)}
	\Big\{ \int_{\reg G_j \cap \{ Z + Y \mid  \abs{Z} < \sigma \}}
	\abs{A_{G_j}} \intdiff \calH^1 \Big \}  \intdiff \calH^{n-1}(Y).
\end{multline}
On the right-hand side we may increase the domain of integration to now
be over the set
$\{  \abs{Y} < 1, \abs{Z} < \sigma \}$.
By the Cauchy--Schwarz inequality, this larger integral is bounded like
\begin{multline}
	\label{eq_conseq_cauchy_schwarz}
	\int_{\reg G_j \cap \{ \abs{Y}< 1, \abs{Z} < \sigma\}}
	\abs{A_{G_j}} \intdiff \calH^n \\
	\leq \calH^n(\reg G_j \cap \{ \abs{Y} < 1, \abs{Z} < \sigma \})^{1/2}
	\Big(
	\int_{\reg G_j \cap \{ \abs{Y} < 1, \abs{Z} < \sigma \}}
	\abs{A_{G_j}}^2 \intdiff \calH^n
	\Big)^{1/2}
\end{multline}

The integral on the right-hand side can be bounded by the stability inequality.
Let $\phi \in C_c^1(\bR^{n+1})$ be a test function with $0 \leq \phi \leq 1$,
$\phi \equiv 1$ on $\{  \abs{Y} < (\sqrt{7} - 1) / 2, \abs{Z} < 1/4 \}$,
$\phi \equiv 0$ outside $\{  \abs{Y} < \sqrt{7} / 2, \abs{Z} < 1/2 \}$
and $\abs{D \phi} \leq 8$.
Taking $\sigma \in (0,1/4)$ and $j \geq J(\sigma)$ we find that
\begin{align*}
	\int_{\reg G_j \cap \{  \abs{Y} < 1, \abs{Z} < \sigma   \} }
	\abs{A_{G_j}}^2 \intdiff \calH^n 
	& \leq \int_{\reg G_j \cap \{  \abs{Y} < (\sqrt{7} - 1)/2,
	\abs{Z} < 1/4 \}} \abs{A_{G_j}}^2 \phi^2 \intdiff \calH^n \\
	& \leq \int_{\reg G_j \cap \{ \abs{Y} < \sqrt{7} / 2, \abs{Z} < 1/2 \} }
	\abs{\nabla_{G_j} \phi}^2 \intdiff \calH^n \\
	& \leq 64 \calH^n(\reg G_j \cap \{ \abs{Y} < \sqrt{7}/2,
	\abs{Z} < 1/2 \}) \\
	& \leq 128 \norm{\bP}(\{ \abs{Y} < \sqrt{7}/2, \abs{Z} < 1/2 \}),
\end{align*}
the last inequality being guaranteed to hold for large $j$.
From this we retain only that there is a constant $B > 0$ independent of
$\sigma$ so that
\begin{equation}
\int_{\reg G_j \cap \{ \abs{Y} < 1, \abs{Z} < \sigma \})}
\abs{A_{G_j}}^2 \intdiff \calH^n \leq B
\end{equation}
for large enough $j \geq J(\sigma)$.
Increasing the value of this constant if necessary, we additionally find
$\calH^n(\reg G_j \cap \{ \abs{Y} < 1, \abs{Z} < \sigma \}) \leq B \sigma$,
and combining the two with~\eqref{eq_lower_bound_kappa} and%
~\eqref{eq_conseq_cauchy_schwarz} yields
\begin{equation}
	\calH^{n-1}(\calV_{j,\sigma} \setminus \calV_{j,\sigma}(\kappa)) \kappa
	\leq	\int_{\reg G_j \cap \{ \abs{Y} < 1, \abs{Z} < \sigma \}}
	\abs{A_{G_j}} \intdiff \calH^n
	\leq B \sigma^{1/2},
\end{equation}
at least for $j \geq J(\sigma)$ large enough.
To force $\calH^{n-1}(\calV_{j,\sigma} \setminus \calV_{j,\sigma}(\kappa)) \to 0$, it suffices
to let $\sigma \to 0$ and $j \geq J(\sigma) \to \infty$.
\end{proof}

\subsection{Initial reduction}
If at least one of the planes in $\spt \norm{\bP}$ is non-vertical, then
the axis $L$ along which the planes meet cannot be vertical either. 
However, any vertical $\Pi \in \{ \Pi_1,\dots,\Pi_D \}$ must contain $L$,
and thus be of the form $\Pi = \linspan \{ L, e_{n+1} \}$.
This uniquely determines the plane, and thus at most one plane in the support of
$\bP$ is vertical.

\begin{lem}
\label{lem_initial_reduction_planes}
Suppose $G_j \to \sum_i m_i \abs{\Pi_i}$, and that $\Pi_1$ is not vertical. 
\begin{enumerate}[label = (\roman*), font=\upshape]
	\item \label{item_second_plane_vertical_number_planes_support}
		If $\Pi_2$ is vertical then $\bP =  \abs{\Pi_1} + m_2 \abs{\Pi_2}$.
	\item \label{item_second_plane_horizontal_conclusion}
		If $\Pi_2$ is not vertical then $\bP = \abs{\Pi_1} + \abs{\Pi_2}$.
\end{enumerate}
\end{lem}
\begin{proof}
We make a preliminary observation: arguing as in the proof of
Theorem~\ref{thm_classical_cone_support_as_desired} one proves the existence
of a constant $\varphi > 0$ so that for all $0 < \tau < \sigma$ and
$j \geq J(\tau,\sigma)$ the following is true. If $Y \in \calV_{j,\sigma}(\varphi)$
then for every $\Upsilon \in \{ \Upsilon_{j,Y}^1,\dots,\Upsilon_{j,Y}^P \}$ 
there is a plane $\Pi \in \{ \Pi_1,\dots,\Pi_D \}$ so that
$\Upsilon \subset (\Pi)_\tau$.

Using this, we first  show that 
\begin{equation}
	\label{eq_initial_conclusion_cone_sum_of_two_planes}
	\bP = m_1 \abs{\Pi_1} + m_2 \abs{\Pi_2}
\end{equation}
regardless of whether or not $\Pi_2$ is vertical, arguing by contradiction.
Suppose that there are at least three distinct planes $\Pi_1,\Pi_2,\Pi_3
\subset \spt \norm{\bP}$.
Let the constant $\varphi > 0$ be as above, 
let $\tau > 0$ be small and take $j \geq J(\tau,\sigma)$ large enough that
$\calH^{n-1}(\calV_{j,\sigma}(\varphi)) > 0$ as
per Lemma~\ref{lem_Uj_kappa_complement_goes_to_zero}.
We may thus take any point $Y \in \calV_{j,\sigma}(\varphi)$ 
and decompose $G_j \cap \{ Z + Y \mid 
\abs{Z} < \sigma \}$ as above.
Next let 
$\Upsilon_1,\Upsilon_2,\Upsilon_3 \in \{ \Upsilon_{j,Y}^1,\dots,\Upsilon_{j,Y}^P \}$
be three of the curves, lying respectively near $\Pi_1,\Pi_2,\Pi_3$,
in the sense that $\Upsilon_i \subset  (\Pi_i)_\tau$.
Possibly after extending $\Upsilon_1$ slightly beyond its endpoints, 
we have that $\{ t v \mid \abs{t} < \sigma \} \subset P_0(\Upsilon_1)$.
Now if $\tau > 0$ is small enough in terms of $\sigma > 0$ then $\Upsilon_2$
and $\Upsilon_3$ intersect in at least one point, say
$ Y + Z_0 = Y + t_0 v + z_0 e_{n+1}$ with $\abs{Z_0} < \sigma$.
This is absurd because the density of $\norm{G_j}$ at such a point is
at least two, and hence writing $Y = (y,Y^{n+1})$ we get
\begin{multline}
	\sum_{X \in P_0^{-1}(\{ y + t_0 v \})} \Theta(\norm{G_j},X) \\
	\geq \sum_{X \in P_0^{-1}(\{ y +  t_0 v \}) \cap \Upsilon_1 }
	\Theta(\norm{G_j},X)
	+ \Theta(\norm{G_j},Z_0 + Y) \geq 3.
\end{multline}
This is impossible for a two-valued graph, which proves%
~\eqref{eq_initial_conclusion_cone_sum_of_two_planes}.

Next we prove~\ref{item_second_plane_vertical_number_planes_support}.
First pick two arcs
$\Upsilon_1,\Upsilon_2 \in \{ \Upsilon_{j,Y}^1,\dots,\Upsilon_{j,Y}^P \}$
so that $\Upsilon_1,\Upsilon_2 \subset (\Pi_1)_\tau$.
Arguing as above, we may extend the two arcs slightly beyond their respective
endpoints and get that $\{y+  tv \mid \abs{t} < \sigma \} \subset P_0(\Upsilon_1)
\cap P_0(\Upsilon_2)$, where $Y = (y,Y^{n+1})$.
If there were another curve $\Upsilon_3$ say in that slice, then at
any point $Z_0 + Y = tv_0 + z_0 e_{n+1} + Y
\in \Upsilon_3 \setminus (\Upsilon_1 \cup \Upsilon_2)$
we would obtain a contradiction, as
\begin{multline}
	\sum_{X \in P_0^{-1}(\{ y + tv_0 \})} \Theta(\norm{G_j},X)  \\
	\geq \sum_{X \in P_0^{-1}(\{ y + tv_0 \}) \cap (\Upsilon_1 \cup \Upsilon_2)}
	\Theta(\norm{G_j},X)
	+ \Theta(\norm{G_j},Z_0 + Y) \geq 3.
\end{multline}

It remains to prove~\ref{item_second_plane_horizontal_conclusion},
where we do not need the decomposition of $G_j$ in the slice.
Pick any point $x \in D_1 \setminus P_0(L)$. Then there exist two distinct
points $X_1,X_2$ in $\Pi_1,\Pi_2$ respectively so that $P_0(X_1) = x = P_0(X_2)$.
As both planes have multiplicity two at most, we can apply the regularity
theory of Wickramasekera if either of the two planes has multiplicity two,
and Allard regularity otherwise, to guarantee that when $j$ is large enough,
then
\begin{equation}
	\label{eq_lower_bound_sum_multiplicities}
	\sum_{X \in P_0^{-1}(\{ x \})} \Theta(\norm{G_j},X)
	\geq m_1 + m_2 \geq 2
\end{equation}
where $m_1, m_2$ are the respective multiplicities of $\Pi_1,\Pi_2$.
Anything but equality in~\eqref{eq_lower_bound_sum_multiplicities}
would be absurd, whence $m_1 = 1 = m_2$.
\end{proof}

\subsection{Horizontal multiplicity one}
\label{subsec_second_case_non_vertical_limit_cones}

We introduce some useful additional notation. As the arguments only simplify
when $L \subset \bR^n \times \{ 0 \}$, we assume throughout that this is
not the case.
For an arbitrary point $X \in \bR^{n+1}$ we write
$X = Y + sf + z e_{n+1}$ where $Y \in L$ and $s,z \in \bR$.
Sometimes it is also convenient to write $Y = (y,Y^{n+1})$ where $y = P_0(Y)$.

We define open domains $Q,Q_\tau \subset \bR^n \times \{ 0 \}$ by
$Q = \{ y + sf \mid \abs{y} < 1, s^2 < 1 \}$
and $Q_\tau = Q \cap (\Pi_2)_\tau = \{ y + sf \mid \abs{y} < 1, s^2 < \tau^2 \}$.
A point $X = Y + sf + z e_{n+1} \in Q \times \bR$ is said to lie
\emph{north} of $\Pi_1$ if $z > 1$ and \emph{south} of $\Pi_1$ if $z < -1$.
In the same vein we say that a set $E \subset Q \times \bR$ lies \emph{north}
(resp.\ \emph{south}) of $\Pi_1$ if all points in $E$ lie north
(resp. south) of $\Pi_1$.

\begin{lem}
\label{lem_graph_north_south}
Let $\Pi_1,\Pi_2 \in \Gr(n,n+1)$ be so that $\Pi_1$ is not vertical, but $\Pi_2$ is. 
Suppose $\abs{G_j} \to \abs{\Pi_1} + 2 \abs{\Pi_2}$. For all $\tau > 0$ there
is $J(\tau) \in \bN$ so that for all $j \geq J(\tau)$, 
$\sing G_j \cap Q \times \bR \subset Q_\tau \times \bR$
and we can decompose
$G_j \cap (Q \setminus \clos{Q}_\tau) \times \bR
=  \Sigma_{j,-}^1 \cup \Sigma_{j,+}^1 \cup \Sigma_{j,-}^2 \cup \Sigma_{j,+}^2$,
into four embedded connected surfaces with
 \begin{equation}
	 \label{eq_support_four_connected_components}
	 \Sigma_{j,\pm}^1 \subset (\Pi_1)_\tau \text{ and }
	 \Sigma_{j,-}^2 \cup \Sigma_{j,+}^2 \subset \{ z > 1 \} \text{ or } \{ z < - 1\}.
 \end{equation}
 Moreover if $\Sigma_{j,-}^2 \cup \Sigma_{j,+}^2 \subset \{ z > 1 \}$ then
\begin{equation}
	\label{eq_singular_points_close_to_axis}
	\sing G_j \cap (Q \times \bR) \subset \{ s^2 < \tau^2, z \geq -1 \},
\end{equation}
and likewise if $\Sigma_{j,\pm}^2 \subset \{ z < -1 \}$.
\end{lem}

Before we give a proof of the lemma, we use its conclusions to derive the following
corollary. 
\begin{cor}
Let the two planes $\Pi_1,\Pi_2 \in \Gr(n,n+1)$ be so that $\Pi_1$ is not vertical,
but $\Pi_2$ is. Then $\abs{G_j} \to \abs{\Pi_1} + 2 \abs{\Pi_2}$ is impossible.
\end{cor}
\begin{proof}
Fix a small value for $\tau > 0$, depending only on $\Pi_1,\Pi_2$ and
a corresponding $J(\tau) \in \bN$ so that without loss of generality,
$G_j \cap Q \times \bR \cap \{ z < - 1 \}$ is a non-empty subset of
$Q_\tau \times \bR \cap \{ z < - 1\}$ and 
$\sing G_j \cap Q_\tau \times \bR \cap \{ z < - 1 \} = \emptyset$.
Then $G_j \mres Q \times \bR \cap \{ z < -1 \}$ is equal to the graph of a 
single-valued, smooth function $u_{j,S}$ defined on some subset $\Omega_{j,S}
\subset Q_\tau$.
From this we only retain that the current $\cur{G_j} \mres Q \times \bR 
\cap \{ z < - 1 \}$ is area-minimising. 
As $j \to \infty$ we get that
$\abs{G_j} \mres  Q \times \bR \cap \{ z < - 1\}
\to 2 \abs{\Pi_2} \mres Q \times \bR \cap \{ z < - 1 \}$ in the
varifold topology.
At the same time, by inspection one finds that in the current topology
$\cur{G_j} \mres Q \times \bR \cap \{ z < - 1 \} \to 0$ as $j \to \infty$.
This mass cancellation is absurd in light of the well-known compactness for
area-minimising currents.
\end{proof}

Now for the proof of Lemma~\ref{lem_graph_north_south}.

\begin{proof}
Note first that every singular point $X = Y + sf + z e_{n+1}$
in $\sing G_j \cap Q \times \bR$ automatically belongs to $Q_\tau \times \bR$,
that is has $s^2 < \tau^2$.
Indeed the Allard regularity theorem can be applied near $\Pi_1$ because
it has multiplicity one in the limit, which guarantees that away from $L$
the $G_j$ converge to the plane like smooth single-valued graphs.
Counting the pre-images of points $x \in Q \setminus \clos{Q}_\tau$ we find that
$\sing u_j \cap Q \setminus \clos{Q}_\tau = \emptyset$,
or equivalently $\sing G_j \cap Q \times \bR \subset \clos{Q}_\tau \times \bR$.

Although $Q \setminus \clos{Q}_\tau$ is not simply connected, its two connected components,
which lie on either side of $\Pi_2$, both are.
We may thus make a smooth selection $u_{j,\tau}^1,u_{j,\tau}^2 \in
C^\infty(Q \setminus \clos{Q}_\tau)$ for $u_j$, arranging for the graph of
$u_{j,\tau}^1$ to lie near $\Pi_1$.
Both graphs are disconnected, and we write $\graph u_{j,\tau}^1 = \Sigma_{j,-}^1
\cup \Sigma_{j,+}^1$ and $\graph u_{j,\tau}^2 = \Sigma_{j,-}^2 \cup \Sigma_{j,+}^2$.
As the graphs $G_j$ locally converge to $\Pi_1 \cup \Pi_2$ in the Hausdorff distance,
we may take an even larger $j \geq J(\tau)$ to get
\begin{equation}
G_j \cap \clos{Q} \times \bR \cap \{ z^2 \leq 1 \}
\subset (\Pi_1 \cup \Pi_2)_{\tau}.
\end{equation}
Thus $\Sigma_{j,\pm}^1 \subset \{ z^2 < ( 1 - \langle e , e_{n+1} \rangle^2)^{-1} \tau^2 \}$
and 
\begin{equation}
	\Sigma_{j,\pm}^2 \subset \{ z^2 > 1 \} = \{ z > 1 \} \cup \{ z < -1 \}.
\end{equation}

We show that in fact either $\Sigma_{j,\pm}^2 \subset \{ z > 1 \}$ or 
$\Sigma_{j,\pm}^2 \subset \{ z < - 1 \}$.
Recall from our initial analysis that using Sard's theorem 
one finds an open subset of `unbranched' points $\calU_j \subset L \cap \{ \abs{y }< 1 \}$
with $\calH^{n-1}(L \cap \{ \abs{y} < 1 \} \setminus \calU_j) = 0$ so that for all
$Y  \in \calU_j$,
\begin{equation}
	G_j \cap \{ Y + sf + z e_{n+1} \mid s^2 < 1, z^2 < 1 \}
	\cap \calB_{G_j}= \emptyset,
\end{equation}
and in fact can be decomposed into a union of three, smooth properly embedded
Jordan arcs $\Upsilon_{j,Y}^1,\Upsilon_{j,Y}^2,\Upsilon_{j,Y}^3$
with endpoints in $\{ Y + sf + z e_{n+1} \mid s^2 = 1\text{ or } z^2 = 1 \}$.

Given $\kappa > 0$ we define the subset $\calU_j(\kappa) \subset \calU_j$ by 
\begin{equation}
	\calU_j(\kappa) = \Big\{ Y \in \calU_j \Big\vert
		\int_{\reg G_j \cap \{ Y + sf + ze_{n+1} \mid s^2,z^2 < \tau^2 \}}
		\abs{A_{G_j}} \intdiff \calH^n < \kappa \Big \}.
\end{equation}
Arguing as in Lemma~\ref{lem_Uj_kappa_complement_goes_to_zero} we can
show that here we can take $\tau > 0$ small enough
and $j \geq J(\tau)$ to make $\calH^{n-1}(\calU_j \setminus
\calU_j(\kappa))$ as small as we like.
For our purposes we may take for example $\kappa = 1/2$, $\tau > 0$ small
and $j \geq J(\tau)$ large enough that $\calH^{n-1}(\calU_j \setminus
\calU_j(\kappa)) < \omega_{n-1}$, as then automatically $\calU_j(\kappa) \neq \emptyset$.
If we then take a point $Y \in \calU_j(\kappa)$, then we may relabel the curves
so that $\Upsilon_{j,Y}^1 \subset (\Pi_1)_\tau$, and $\Upsilon_{j,Y}^2 \cup \Upsilon_{j,Y}^3
\subset (\Pi_2)_\tau$.

Consider $Y = (y,Y^{n+1}) \in \calU_j$ and let $l_y = \{ y + sf \mid s^2 < 1 \}
\subset \bR^n \times \{ 0 \}$.
On this line segment we can make a smooth selection $u_{j,y}^1,u_{j,y}^2
\in C^\infty(-1,1)$, where we identify $l_y$ with $(-1,1) \subset \bR$.
Then $G_j \cap \{ Y + sf + z e_{n+1} \mid s^2 < 1, z^2 < 1 \}
\subset \graph u_{j,y}^1 \cup \graph u_{j,y}^2$.
We may furthermore make our selection in such a way that
$\Upsilon_{j,Y}^1 \subset \graph u_{j,y}^1 \subset (\Pi_1)_{\tau}$ and 
$\Upsilon_{j,Y}^2 \cup \Upsilon_{j,Y}^3 \subset \graph u_{j,y}^2$.
In what follows we also write $l_y(a,b) = \{ y + sf \mid a < s < b \}$, where
$-1 < a < b < 1$.
The graph of $u_{j,y}^2$ restricted to the short segment $l_y(-2\tau,2\tau)$
is a single smooth curve, which by inspection has both its endpoints lying
on the same side of $\Pi_1$, that is either both lie north or both lie south.
But $\graph u_{j,y}^2 \cap l_y(-2\tau,-\tau) \times \bR \subset \Sigma_{j,-}^2$
and likewise $\graph u_{j,y}^2 \cap l_y(\tau,2\tau) \times \bR \subset \Sigma_{j,+}^2$
whence we find that $\Sigma_{j,-}^2$ and $\Sigma_{j,+}^2$ too must lie on
the same side of $\Pi_1$---that is either both lie north or both lie south of $\Pi_1$.
This concludes the proof of~\eqref{eq_support_four_connected_components}.

To  prove~\eqref{eq_singular_points_close_to_axis} we start by making
a few general observations. First,
by~\eqref{eq_support_four_connected_components} we may assume without loss
of generality that $\Sigma_{j,\pm} \subset \{ z > 1 \}$.
Let $Y = (y,Y^{n+1}) \in L$ be an arbitrary point with
$\abs{y} < 1$, not necessarily in $\calU_j$.
By the two-valued sheeting theorem of~\cite{Wic_MultTwoAllard} applied
in the region $Q \times \bR \cap \{ \tau^2 < s^2 \vee z^2 < 1 \}$,
we find that in this slice the graph can be decomposed into six differentiable curves,
\begin{equation}
	\label{eq_decomp_two_valued_graph}
	G_j \cap \{ Y + sf + z e_{n+1} \mid \tau^2 < s^2 \vee z^2 < 1 \}
	= \cup_{k = 1}^2 \gamma_{j,Y,k}^1
	\cup \cup_{l=1}^4 \gamma_{j,Y,l}^2,
\end{equation}
where $\gamma_{j,Y,k}^1 \subset (\Pi_1)_\tau$ and $\gamma_{j,Y,l}^2 \subset (\Pi_2)_\tau$.
(Were the slice $\{ Y + sf + z e_{n+1} \mid \tau^2 < s^2 \vee z^2 < 1 \}$ 
to contain a branch point of $\calB(G_j)$, then some arbitrary choices
would have to be made in this decomposition, with no impact on the argument.)
These curves taken together have two endpoints $\{ z = -1, s^2 < \tau^2\}$,
counted with multiplicity. Hence
$\# G_j \cap \{ Y + sf - e_{n+1} \mid s^2 < \tau^2 \} \leq 2$ and
\begin{equation}
	\label{eq_multiplicities_on_short_line_segment}
	\sum_{ s^2 < \tau^2  } \Theta(\norm{G_j},Y + sf - e_{n+1}) = 2.
\end{equation}

There exist two functions $u_{j,y}^1,u_{j,y}^2 \in C^1(-1,1)$ so that 
\begin{equation}
	G_j \cap \{ Y + sf + z e_{n+1}\mid
	s^2 < 1,  z \in \bR \} = \graph u_{j,y}^1 \cup \graph u_{j,y}^2.
\end{equation}
Moreover their graphs $\graph u_{j,y}^i$ are two differentiable
curves which do not meet the region south of $\Pi_1$, except
in the thin strip near $\Pi_2$ where $s^2 < \tau^2$.
In other words,
$( \graph u_{j,y}^1 \cup \graph u_{j,y}^2)  \cap \{ s^2 \geq \tau^2, z < -1 \} 
= \emptyset$.

Notice that $\sing G_j \cap \{ Y + sf + z e_{n+1} \mid s^2 < 1, z\in \bR \}
= \graph u_{j,y}^1 \cap \graph u_{j,y}^2$. Hence, if
$\{ Y + sf + z e_{n+1} \mid s^2 < 1, z < -1 \}$ contained a singular point
of $G_j$, then $\graph u_{j,y}^1,\graph u_{j,y}^2$ would both
contain portions lying in that region.
This is impossible, because if both curves intersected that region they
would both need to pass through the set $\{ Y + s^2 - e_{n+1} \mid s^2 < \tau^2 \}$
twice each, because neither meets the set $\{ s^2 \geq \tau^2 , z < -1 \}$.
This in turn would be a contradiction to the decomposition of the graph
as we obtained in~\eqref{eq_decomp_two_valued_graph} above, and%
~\eqref{eq_multiplicities_on_short_line_segment} in particular.
As the point $Y$ was chosen arbitrarily, we have 
$\sing G_j \cap \{ s^2 < 1, z < -1 \} = \emptyset$, which concludes the proof.
\end{proof}

\section{Classical limit cones: vertical cones}
\label{sec_classical_cones_vertical}

Let $\alpha \in (0,1)$ and $(u_j \mid j \in \bN)$
be a sequence of two-valued minimal graphs $u_j \in C^{1,\alpha}(D_2;\calA_2)$.
Let $D \in \bZ_{>0}$ and for $i = 1,\dots,D$, let
$\Pi_i = \Pi_i^0 \times \bR e_{n+1} \in \Gr(n,n+1)$ be vertical planes which meet along
a common $(n-1)$-dimensional axis $L = L_0 \times \bR e_{n+1}$,
and suppose that as $j \to \infty$,
$\abs{G_j} \to \sum_i m_i \abs{\Pi_i}$ and $\cur{G_j} \to \sum_i l_i \cur{\Pi_i}$.
Here $0 \leq l_i \leq m_i \leq 2$ are integers and
we pick arbitrary orientations for those planes with $l_i = 0$.
When this notation is convenient we write $\bP = \sum_i m_i \abs{\Pi_i}$ 
and $T = \sum_i l_i \cur{\Pi_i}$.
To make statements less awkward, we allow the possibility that $\bP$ and $T$ 
are supported in a single plane, although technically these would not be
called classical cones.

Label the planes $\Pi_1,\dots,\Pi_D$ so that they lie in counterclockwise
order around $L$. From now on we consider their indices modulo $D$, and
write $\Pi_i = \pi_i \cup \pi_{i+D}$, where $\pi_i,\pi_{i+D}$ are two 
half-planes which meet along $L$. The indices of the half-planes are considered
modulo $2D$.
For every $\pi_i$ let $N_i$ be its unit normal pointing in the counterclockwise
direction; note that $N_i = - N_{i+D}$. Write $n_i$ for the unit normal induced
on $\Pi_i$ as limits of the $\cur{G_j}$, and let $s_i = \langle n_i , N_i \rangle$,
equal to $\pm 1$ depending on whether or not $n_i$ agrees with the counterclockwise
orientation.
With this notation, $s_i = - s_{i+D}$.
We say that two half-planes $\pi_i,\pi_j$ are oriented in the
same direction if they are both oriented in the clockwise or counterclockwise
direction, or equivalently if $s_i = s_j$.

Let $Q = \{ x \in \bR^n \mid \dist(x,L) < 1, \dist(x,L^\perp) < 1 \}$.
Extend the functions $F^j$ from~\eqref{eq_defn_number_of_sheets_function}
to $Q$ using the same formula,
counting the number of points in $G_j$ lying below $x$ with multiplicity.
These functions are eventually constant away from $\cup_i \Pi_i$,
that is given $\tau > 0$ we can take $j \geq J(\tau)$ large enough that the
$F^j$ are constant on every connected component of $Q \setminus (\cup_i \Pi_i)_\tau$.
Write $Q \setminus \cup_i \Pi_i$ as a disjoint union of
wedge-shaped connected components $V_1,\dots,V_{2D}$. Each
$V_i$ lies between $\pi_i,\pi_{i+1}$,
\begin{equation}
V_i = \{ x \in Q \mid \langle x, N_i \rangle > 0,
\langle x , N_{i+1} \rangle < 0 \}.
\end{equation}
By the above there is $F_i = F(V_i) \in \{ 0 ,1,2\}$  so that
$F^j(x) = F_i$ at all $x \in V_i \setminus (\pi_i \cup \pi_{i+1})_\tau$
provided $j \geq J(\tau)$.
Although the notation is slightly ambiguous, no confusion should arise
between the value $F_i = F(V_i)$ and the functions $F^j$.

\subsection{Results in arbitrary dimensions}

Applying Lemma~\ref{lem_description_near_planes} in the present context,
we can relate consecutive values of $F_i$.
\begin{lem}
\label{lem_consec_number_beans}
For all $i$, $F_i - F_{i-1} = s_i l_i$.
\end{lem}
Similarly using Lemma~\ref{lem_no_folding_two_half_planes},
we obtain the following result.
\begin{lem}
\label{lem_single_sheet_lemma}
If $F_i = 1$ and $\pi_i \neq \pi_{i+1}$ then both half-planes have multiplicity
one and are oriented in the same direction.
\end{lem}
\begin{proof}
Start with the observation that
$1 = F_i = F_{i+1} - s_{i+1} l_{i+1}
= F_{i-1} + s_i l_i$, so $l_i,l_{i+1} \in \{ 0,1 \}$.
The possibilities are as follows:
\begin{enumerate}
	\item $m_i = 1 = m_{i+1}$ and $l_i = 1 = l_{i+1}$,
	\item $m_i = 2, m_{i+1} = 1$ and $l_i = 0$, $l_{i+1} = 1$,
	\item $m_i = 1, m_{i+1} = 2$ and $l_i = 1$, $l_{i+1} = 0$,
	\item $m_i = 2 = m_{i+1}$ and $l_i = 0 = l_{i+1}$.
\end{enumerate}
The proof is similar in all four cases. Every case is argued by contradiction,
eventually reaching a conclusion that
Lemma~\ref{lem_no_folding_two_half_planes} forbids.
We give a detailed proof for $m_i = 1 = m_{i+1}$ and $l_i = 1 = l_{i+1}$ and
explain the necessary modifications for the remaining cases.

Assume without loss of generality that
$\pi_i$ and $\pi_{i+1}$ both point into $V_i$, that is $n_i = N_i$ and
$n_{i+1} = -N_{i+1}$.
Let two small $0 < \tau < \sigma < 1$ and a large $A > 1$ be given, and
define the open subset $U_{j,i} \subset Q \setminus [L]_\sigma$ by
\begin{align}
	\label{eq_defn_region_U_ij}
	U_{j,i}
	&= U_{j,i}(\tau,\sigma,A) \\
	&= (V_i)_\tau \setminus [L]_\sigma  \cap \{ u_j^- < 2A \}  \\
	& = \{ x\in Q \setminus [L]_\sigma \mid 
	\langle N_i, x \rangle > -\tau, \langle N_{i+1} , x \rangle < \tau, 
	u_j^-(x) < 2A \},
\end{align}
where recall $u_j^- = u_j^1 \wedge u_j^2$.
As $F_i = 1$ we know that for large enough $j \geq J(\tau,\sigma,A)$,
\begin{equation}
	V_i \setminus ( [\pi_i]_\tau \cup [\pi_{i+1}]_\tau \cup [L]_\sigma )
	\subset U_{j,i}
\end{equation}
and $\sing G_j \cap U_{j,i} \cap (-\infty,-2A) = \emptyset$.
Hence any singular points of $u_j$ would have to lie in
$U_{j,i} \cap [\pi_i \cup \pi_{i+1}]_\tau$.
As $m_i = 1 = m_{i+1}$ we may use Allard regularity inside 
$Q \times (-9/4 A, 9/4 A) \cap (\pi_i \cup \pi_{i+1})_{2\tau}$
and find that in fact
\begin{equation}
	\sing u_{j} \cap U_{j,i} = \emptyset,
\end{equation}
at least provided $j \geq J(\tau,\sigma,A)$
is large enough.

Let us rename $u_{j,i,S} = u_{j,i}^-$ and $u_{j,i,N} = u_{j,i}^+$.
These two functions give a smooth selection for $u_j$ on $U_{j,i}$.
By construction
\begin{equation}
	\label{eq_southern_part_contains_most_portions}
	G_j \cap (V_i)_\tau \setminus [L]_\sigma \times (-\infty,A)
	\subset \graph u_{j,i,S}.
\end{equation}

From this we obtain a contradiction with Lemma~\ref{lem_no_folding_two_half_planes}.
Indeed as we let $\sigma,\tau \to 0$,
$A \to \infty$ and $j \geq J(\tau,\sigma,A) \to \infty$, we have
$\dist_{\calH}(U_{j,i},V_i) \to 0$ 
and
\begin{equation}
	\label{eq_convergence_to_two_planes}
	\cur{\graph u_{j,i,S}} \to \cur{\pi_i} + \cur{\pi_{i+1}}
	\mres Q \times \bR.
\end{equation}

This concludes the proof in the first case. In the next case 
$m_i = 2$, $m_{i+1} = 1$ and $l_i = 0$, $l_{i+1} = 1$.
As $l_i = 0$ the half-plane $\pi_i$ has no well-defined orientation
induced by $T$. The half-plane $\pi_{i+1}$ may be assumed oriented
in the clockwise direction without loss of generality,
that is $n_{i+1} = - N_{i+1}$.
By Lemmas~\ref{lem_description_near_planes} and~\ref{lem_no_branch_points_if_cancellation}
we may take $j \geq J(\tau,\sigma)$ large enough that
\begin{align}
	\sing u_j \cap ( V_i \cup V_{i-1})
	\setminus ([\pi_{i-1} \cup &\pi_i \cup \pi_{i+1}]_\tau \cup [L]_\sigma)
	= \emptyset, \\
	\calB_{u_j} \cap (V_{i-1} \cup V_i \cup \pi_i)
	\setminus ([\pi_{i-1} &\cup \pi_i]_\tau \cup [L]_\sigma)
	= \emptyset.
\end{align}
Additionally the set $(V_{i-1} \cup V_{i} \cup \pi_i) \setminus 
([\pi_{i-1} \cup \pi_i]_\tau \cup [L]_\sigma)$ is simply connected, so we
can make a smooth selection $\{ u_{j,i,S},u_{j,i,N} \}$ for $u_j$ on it,
arranging the indices in a way that $\graph u_{j,i,S}$ lies south of $V_i$.
(Here $u_{j,i,S}$ is not equal $u_{j,i}^-$ anymore.)
Define the region
$U_{j,i} = (V_i)_\tau \setminus [L]_\sigma \cap \{ u_{j,i,S} < 2A \}
 \subset (V_i)_\tau \setminus [L]_\sigma$.
As above~\eqref{eq_southern_part_contains_most_portions} holds, and
near $\pi_i$ we have that given any $\delta > 0$,
\begin{equation}
	\label{eq_mult_two_case_orientation_goes_in_the_right_direction}
\langle \nu_j(X),N_i \rangle > 1 - \delta 
\end{equation}
for all $X \in \graph u_{j,i,S} \cap \reg G_j \cap (\pi_i)_\tau \cap
\{ \abs{X^{n+1}} < A \}$,
at least after updating $j \geq J(\tau,\sigma,A,\delta)$.
As we let $\tau,\sigma,\delta \to 0$ and $A \to \infty$ 
both $\dist_\calH(U_{j,i},V_i) \to 0$ and%
~\eqref{eq_convergence_to_two_planes} hold as $j \geq J(\tau,\sigma,A,\delta) \to \infty$.
The last inequality~\eqref{eq_mult_two_case_orientation_goes_in_the_right_direction}
guarantees that the limit current has the right orientation to apply
Lemma~\ref{lem_no_folding_two_half_planes}, which immediately yields a contradiction.
The third case, when $m_i = 1, m_{i+1} = 2$ and $l_i = 1, l_{i+1} = 0$ can be
argued in precisely the same way, with reversed roles of $\pi_i$ and $\pi_{i+1}$.

The last remaining case is $m_i = 2 = m_{i+1}$ and $l_i = 0 = l_{i+1}$. 
Arguing as above we find
\begin{align}
	\sing u_j \cap (V_{i-1} \cup V_i \cup V_{i+1})
	\setminus ([\pi_{i-1} \cup &\pi_i \cup \pi_{i+1} \cup \pi_{i+2}]_\tau \cup
	[L]_\sigma) = \emptyset, \\
	\calB_{u_j} \cap (V_{i-1} \cup V_i \cup V_{i+1} \cup \pi_i \cup \pi_{i+1})
	\setminus ([& \pi_{i-1} \cup \pi_{i+2}]_\tau \cup [L]_\sigma)
	= \emptyset.
\end{align}
We may make a smooth selection $\{ u_{j,i,S},u_{j,i,N} \}$ for $u_j$ on the
latter set $(V_{i-1} \cup V_i \cup V_{i+1} \cup \pi_i \cup \pi_{i+1})
\setminus ([\pi_{i-1} \cup \pi_{i+2}]_{\tau} \cup [L]_\sigma)$.
We arrange for $u_{j,i,S}$ to lie south of $V_i$, and moreover 
\begin{equation}
\langle \nu_j(X),N_i \rangle > 1-\delta
\text{ and }
\langle \nu_j(X),-N_{i+1} \rangle > 1 - \delta
\end{equation}
at all points  $X \in \graph u_{j,i,S} \cap \reg G_j \cap (\pi_i)_\tau
\cap \{ \abs{X^{n+1}} < A \}$
and $\graph u_{j,i,S} \cap \reg G_j \cap ( \pi_{i+1})_\tau
\cap \{ \abs{X^{n+1}} < A \}$ respectively.
From that point on we can argue in the same way as above, ultimately 
leading to a contradiction with with Lemma~\ref{lem_no_folding_two_half_planes}.
This exhausts the list of possible cases, and concludes the proof.
\end{proof}

Using this, we can immediately conclude mass cancellation in all but one
case. Indeed $\abs{T} \neq \bP$ if and only if $l_i = 0$ for at least one
plane.
But then $F_{i-1} = 1 = F_i$ by Lemma~\ref{lem_description_near_planes},
which contradicts Lemma~\ref{lem_single_sheet_lemma} unless $\pi_i = \pi_{i+1}$
and $D = 1$.
For the remainder of this section, we may assume that $l_i = m_i$ for all $i$.
As a consequence also $F_{i} -  F_{i-1} = s_i m_i$ for all $i$.

\begin{lem}
	\label{lem_sum_consec}
Let $\abs{G_j} \to \sum_{i=1}^D m_i \abs{\Pi_i}$. Then
\begin{enumerate}[label = (\roman*), font = \upshape]
	\item \label{item_sum_any_consec}
	any consecutive $\pi_{i+1},\dots,\pi_{i + J}$ have
	$ \abs{ \sum_{j = 1}^J s_{i+j} m_{i+j}} \leq 2$,
	\item \label{item_sum_consec_mult_two} if $\pi_i$ has multiplicity two then
	 $\sum_{j = 1}^{D-1} s_{i + j} m_{i + j} = 0$.
\end{enumerate}
\end{lem}
\begin{proof}
\ref{item_sum_any_consec}
Iterating Lemma~\ref{lem_consec_number_beans} we find that
$F_{i+J} - F_i = \sum_{j=1}^J s_{i+j} m_{i+j}$. As $F_{i+J},F_i \in \{ 0 , 1,2 \}$
we get $\abs{\sum_j s_{i+j} m_{i+j}} \leq 2$ as desired.

\ref{item_sum_consec_mult_two} Here we consider the $D-1$ consecutive planes
$\pi_{i+1},\dots,\pi_{i + D - 1}$, which stop just shy of $\pi_i$ and $\pi_{i + D}$.
Because $\pi_i \cup \pi_{i+D} = \Pi$ is a plane, they must have $s_{i} = -s_{i+D}$;
without loss of generality $s_i = 1$.
By Lemma~\ref{lem_consec_number_beans}, $F_i = F_{i-1} + 2$ and $F_{i+D} = F_{i+D-1} - 2$,
so $F_i = 2 = F_{i+D-1}$.
Iterating the same lemma over the half-planes $\pi_{i+1},\dots,\pi_{i+D-1}$
we find $F_{i+D-1} = F_i + \sum_{j=1}^{D-1} s_{i+j} m_{i+j}$,
whence $\sum_{j=1}^{D-1} s_{i+j} m_{i+j} = 0$, as desired.
\end{proof}

\begin{cor}\leavevmode
	\label{cor_dimless_configurations}
Let $\abs{G_j} \to \sum_{i = 1}^D m_i \abs{\Pi_i}$. Then
\begin{enumerate}[label = (\roman*), font = \upshape]
	\item \label{item_limit_confs_all_mult_one}
		if $m_1 = \cdots = m_D = 1$ then $D \equiv 2 \pmod 4$,
	\item \label{item_limit_confs_all_mult_two}
		if $m_1 = \cdots = m_D = 2$ then $D$ is odd,
	\item \label{item_limit_confs_one_mult_two}
		if $m_1 = 2$, $m_2 = \cdots = m_D = 1$ then
		$D \equiv 1 \pmod 4$.
\end{enumerate}
\end{cor}
\begin{proof}
The result is obtained by listing the orientations of the half-planes
$\pi_1,\dots,\pi_{2D}$ weighted by their respective multiplicities,
and excluding certain subsequences from this.
We start by considering three consecutive half-planes $\pi_{i-1},\pi_i,\pi_{i+1}$
with multiplicities $m_{i-1},m_i,m_{i+1} = 1$, and show that then
$(s_{i-1},s_i,s_{i+1}) \in \{ \pm (1,-1,-1), \pm (1,1,-1) \}$.
First note that the two sequences $(s_{i-1},s_i,s_{i+1}) = \pm (1,1,1)$ are excluded
by Lemma~\ref{lem_consec_number_beans}.
For the remaining cases we argue by contradiction,
and assume that $(s_{i-1},s_i,s_{i+1}) = (1,-1,-1)$.
Then on the one hand by Lemma~\ref{lem_single_sheet_lemma},
applied between $\pi_{i-1},\pi_i$ and $\pi_i,\pi_{i+1}$ respectively
we find that $F_{i-1},F_i \neq 1$.
This is absurd, as on the other hand $F_i = F_{i-1} + 1$.
One reasons similarly when $(s_{i-1},s_i,s_{i+1}) = (-1,1,1)$.

\ref{item_limit_confs_all_mult_one}
When the multiplicities are all equal $m_1,\dots,m_D = 1$ 
then the only possibility is that 
$(m_1s_1,\dots,m_{2D}s_{2D}) = (s_1,\dots,s_{2D}) = (1,1,-1,-1,\dots)$
or a cyclic permutation thereof.
Hence $D$ must be even. As $s_1,s_2 = 1$ we get $s_{D+1},s_{D+2} = -1$,
whence $s_{D-1},s_D = 1$ and $D \not \equiv 0 \pmod 4$.

\ref{item_limit_confs_all_mult_two}
From Lemma~\ref{lem_sum_consec}~\ref{item_sum_any_consec}
either $(s_1,\dots,s_{2D}) =  (1,1,-1,-1,\dots)$
or a cyclic permutation thereof, so $D$ must be odd.

\ref{item_limit_confs_one_mult_two}
Without loss of generality $s_1 = 1$, and by Lemma~\ref{lem_consec_number_beans}
the sequence $(m_i s_i)$ starts $(2s_1,s_2,s_3) = (2,-1,-1)$.
The orientations of half-planes with multiplicity one
alternate in pairs, so this continues $(s_4,s_5,s_6,s_7,\dots) = (1,1,-1,-1,\dots)$.
As $2 s_{D+1} = -2$ we get $(s_{D-1},s_D,2 s_{D+1}) = (1,1,-2)$.
Combining the two observations, $D-1 \equiv 0 \pmod 4$.
\end{proof}

\subsection{Classification in dimensions up to seven}

\label{subsec_classification_in_dimensions_up_to_seven}

Here too, as in the previous section, we consider a sequence of two-valued
minimal graphs $u_j \in C^{1,\alpha}(D_2;\calA_2)$ which converge in
the varifold topology, $\abs{G_j} = \abs{\graph u_j} \to \sum_i m_i \abs{\Pi_i}$
as $j \to \infty$.
These planes are assumed to meet along a single $n-1$-dimensional vertical axis
$L = L_0 \times \bR e_{n+1} \in \Gr(n-1,n+1)$.

\begin{cor}
	\label{cor_classical_limit_cone_classification}
Let $\abs{G_j} \to \sum_{i = 1}^D m_i \abs{\Pi_i}$.
If $2 \leq n \leq 6$ then this is either
\begin{equation}
	2 \abs{\Pi_1} 
	\text{ or }
	\abs{\Pi_1} + \abs{\Pi_2}.
\end{equation}
\end{cor}
\begin{proof}
As the graphs have dimension up to six, the area estimates of
Proposition~\ref{prop_improved_estimates_sum_of_planes} give
$\sum_j m_j \leq \lfloor n \omega_n / \omega_{n-1} \rfloor \leq 5$.
Hence $D \leq 5$, and the possibilities for $(m_1,\dots,m_D)$ 
are listed in Table~\ref{table_multiplicities} up to cyclic permutation.
Of these, the only not forbidden by Corollary~\ref{cor_dimless_configurations}
is $(m_1,m_2,m_3) = (2,2,1)$. However $2 s_2 + s_3 \neq 0$ is forbidden by
Lemma~\ref{lem_consec_number_beans}.
\end{proof}

\begin{table}
\caption{The possibilities for the multiplicities afforded by the improved
	area bounds of Corollary~\ref{cor_qualitative_estimates} when $2 \leq n \leq 6$,
	up to cyclic permutations.}
\begin{tabularx}{0.5\textwidth}{c |  Y }
\toprule
$D$ & $(m_1,\dots,m_D)$ \\ \midrule
$2$ & $(2,2),(2,1)$ \\ 
$ 3$ & $  (2,1,1),(1,1,1),(2,2,1)$ \\
$ 4$ & $  (2,1,1,1),(1,1,1,1)$ \\
$ 5$ & $  (1,1,1,1,1)$ \\
\bottomrule
\end{tabularx}
\label{table_multiplicities}
\end{table}

\subsection{Multiplicity and mass cancellation}

\begin{cor}
\label{cor_sing_of_limit_cones}
Suppose $2 \leq n \leq 6$. Let 
\begin{equation}
	\abs{G_j} \to V \neq 0 \in \IV_n(D_1 \times \bR)
\text{ as $j \to \infty$.}
\end{equation}
Then either 
$\Theta(\norm{V},X) = 2$ for $\calH^n$-a.e.\ $X \in \reg V$
and there is a smooth, stable minimal surface $\Sigma$ so that
\begin{equation}
	V = 2 \abs{\Sigma},
\end{equation}
or $\Theta(\norm{V},X) = 1$ for $\calH^n$-a.e.\ $X \in \reg V$ and
\begin{enumerate}[label = (\roman*), font = \upshape]
	\item \label{item_limit_varifold_immersed} $\spt \norm{V}$ is immersed
		near points of $\calS^{n-1}(V) \setminus \calS^{n-2}(V)$,
	\item \label{item_rectifiability_of_lower_strata} the set $\calS^{n-2}(V)
		\cup \calB(V)$ is countably $(n-2)$-rectifiable.
\end{enumerate}
\end{cor}
\begin{proof}
We only need to show that the multiplicity of regular points of $V$ is either one or
two; the conclusion follows by combining our limit cone classification with the
results of~\cite{SchoenSimon81,KrumWic_FinePropsMinGraphs,Wic_MultTwoAllard}.
We may assume without loss of generality that $G_j \cap D_1 \times \bR$ is connected
for all $j$. As the graphs converge locally in $D_1 \times \bR$ with respect to
Hausdorff distance, $\spt \norm{V} \cap D_1 \times \bR$ is also connected.
Let $\calR$ be the set of connected components of $\reg V$, which we group into the
two sets $\calR_1$ and $\calR_2$ according to their respective multiplicities.
(The multiplicities are constant on every component by~\cite[Thm.\ 41.1]{Simon84}.)

We use a contradiction argument to show that one of the two is empty.
Let $V_1 = \sum_{\Sigma \in \calR_1} \abs{\Sigma}$ and $V_2 = \sum_{\Gamma \in \calR_2}
2 \abs{\Gamma}$. These are both stationary in $D_1 \times \bR$ away from
$\spt \norm{V_1} \cap \spt \norm{V_2} \cap D_1 \times \bR$.
By our classification of limit cones, $\calC(V) \cap \spt \norm{V_2} = \emptyset$.
As $\calB(V) \cup \calS^{n-2}(V)$ is countably $(n-2)$-rectifiable, this means
$\calH^{n-1}(\spt \norm{V_1} \cap \spt \norm{V_2} \cap D_1 \times \bR) = 0$,
whence $V_1,V_2$ are in fact stationary in $D_1 \times \bR$ without restrictions.
One argues in the same way to justify their stability in $D_1 \times \bR$.
As the support of $V_2$ contains neither genuine branch points nor classical
singularities,~\cite{SchoenSimon81} implies that there is a smooth embedded minimal
surface $\Sigma_2 \subset D_1 \times \bR$ so that $V_2 = 2 \abs{\Sigma_2}$.
Take $X \in \spt \norm{V_1} \cap \Sigma_2 \cap D_1 \times \bR$, and $\rho > 0$
small enough enough that $B_\rho(X) \setminus \Sigma_2$ has two connected components,
say $U_{\pm}$.
Let $V_1^{\pm}$ be the two varifolds made up of the portions of $V_1$ lying in $U_{\pm}$
respectively. Arguing as above we find that $V_1^{\pm}$ are both stationary in
$B_\rho(X)$. Without loss of generality $X \in \spt \norm{V_1^+}$. As $V_1^+$ lies
above $\Sigma_2$, we have reached a contradiction with the maximum principle
of Solomon--White~\cite{SolomonWhiteMaxPrinc}.
\end{proof}

\begin{cor}[No mass cancellation]
Let $2 \leq n \leq 6$. Suppose that as $j \to \infty$,
$\abs{G_j} \to V \in \IV_n(D_1 \times \bR)$ and  
$\cur{G_j} \to T \in \I_n(D_1 \times \bR)$.
If $T \neq 0$ then $\abs{T} = V$.
\end{cor}
\begin{proof}
The limit current $T$ necessarily has $\abs{T} \ll V$,
with equality if and only if it there is no mass cancellation.
This cannot occur at points of multiplicity one, so by
Corollary~\ref{cor_sing_of_limit_cones} we may assume that
$\Theta(\norm{V},X) = 2$ for $\calH^n$-a.e.\ $X \in \reg V$
and there is a smooth embedded minimal surface $\Sigma$ with
$V = 2\abs{\Sigma}.$
Thus also $\spt \norm{T} \subset \Sigma$,
and by the Constancy Theorem \cite[Thm.~41.1]{Simon84}
there is $m \in \bZ_{> 0}$ so that
$\abs{T} = m \abs{\Sigma}.$
As $m \leq 2$, and $m \neq 1$ as otherwise we could use Allard's regularity
theory, we have that either $m = 2$ and $\abs{T} = V$ or $m = 0$ and $T = 0$.
\end{proof}

\section{Blowdown cones and asymptotic analysis}

\label{sec_blowdown_cones_and_asymptotic_analysis}

\subsection{Entire graphs with bounded growth}

\begin{thm}
\label{thm_grad_bounds_vert_line}
Let $\alpha \in (0,1)$ and $n \geq 2$ be arbitrary.
Let $u \in C^{1,\alpha}(\bR^n;\calA_2)$ be an entire
two-valued minimal graph with $u(0) = 0$. If 
\begin{equation}
	\limsup_{r \to \infty} \big( \norm{u}_{0;D_r} / r \big) < +\infty
\end{equation}
then $u$ is linear.
Otherwise the support of every blowdown cone at infinity contains the half-line
$L_+ = \{ t e_{n+1} \mid t \geq 0 \}$ or its reflection $-L_+$.
\end{thm}
\begin{proof}
Suppose first that $u$ has bounded growth, say 
$\sup_r r^{-1} \norm{u}_{0;D_r} \leq C$ for some $C > 0$. 
Let $(\lambda_j \mid j \in \bN)$ be a sequence of positive scalars
with $\lambda_j \to \infty$, along which we blow down $u$.
For all $j \in \bN$, define $u_j \in C^{1,\alpha}(\bR^n;\calA_2)$ 
by setting $u_j(x) = \lambda_j^{-1} u(\lambda_j x)$ for all $x \in \bR^n$.
Using the interior gradient estimates, we find that for all $r > 0$ there
is a constant $C(r)$ so that $\sup_j \norm{u_j}_{1;D_r} \leq C(r)$.
By the two-valued Lipschitz theorem we can extract a subsequence
so that that there is a two-valued Lipschitz function
$U \in \Lip(\bR^n;\calA_2)$ so that simultaneously $u_{j'} \to U$ locally uniformly
and $\abs{G_{j'}} \to \abs{\graph U}$.
By Lemma~\ref{lem_reg_lipschitz_graph_cone}, $U$ must be linear and its
graph is a union of two possibly equal planes. By the monotonicity formula the same
holds for $u$.
This concludes the proof of the first half of the lemma.

Now assume instead that $\sup_{r > 0} r^{-1} \max_{D_r} \{ u_1,u_2 \} = + \infty$,
without loss of generality. We show that $\{ t e_{n+1} \mid t \geq 0 \}$
is contained in the support of any blowdown cone of $\abs{G}$.
As above, we blow down $\abs{G}$ along the sequence $\lambda_j \to \infty$.
Let $\delta > 0$ be a small parameter, whose value we eventually let tend to zero.
Inside the disc $D_\delta$
the functions $u_j$ have
$\max_{\clos{D}_\delta} \{ u^j_1,u^j_2 \} \to +\infty$ as $j \to \infty$,
so that $\max_{\clos{D}_\delta} \{ u_1^j,u_2^j \} \geq 1$
for large enough $j \geq J(\delta)$.
Hence there exists a sequence
of points $X_j = (x_j,X_j^{n+1}) \in G_j \cap D_\delta \times \bR$
with $X_j^{n+1} > 1$ for all $j$.
As $G_j \cap D_\delta \times \bR$ is connected there
is a continuous path $\gamma_j: [0,1] \to \bR^{n+1}$
with image $\gamma_j([0,1]) \subset G_j \cap D_\delta \times \bR$
and endpoints $\gamma_j(0) = \orig$ and $\gamma_j(1) = X_j$.
This path must cross the solid disc $D_\delta \times \{ 1 \}$
at height one, so that by picking a point in this intersection we can construct
a sequence of points $(Y_{j,\delta} \mid j \geq J(\delta))$ that respectively
belong to $\gamma_j([0,1]) \cap D_\delta \times \{ 1 \}$.

Given any positive sequence $(\delta_m \mid m \in \bN)$ with $\delta_m \to 0$
repeat the argument with $\delta = \delta_m$. Via a diagonal extraction argument
we obtain a subsequence of indices $(j_m \mid m \in \bN)$ and $(Y_m \mid m \in \bN)$
with $Y_m  = 	Y_{j_m,\delta_m} \in G_{j_m} \cap D_{\delta_m} \times \{ 1 \}.$
They converge to the point $(0,1) \in \bR^{n+1}$ at height one, which
thus is in $\spt \norm{\bC}$. As $\bC$ is a cone, this confirms that
$\{ t e_{n+1} \mid t \geq 0 \} \subset \spt \norm{\bC}$.
\end{proof}

\subsection{General results in low dimensions}

Combining the results from Section~\ref{subsec_max_princ_near_branch_point}
with the work of~\cite{Wic_MultTwoAllard} we obtain the following technical
lemma, which will be useful in what follows.

\begin{cor}
	\label{cor_vertical_tangent_cones}
Let $V \in \IV_n(D_2 \times \bR)$ be the limit of a sequence of two-valued
graphs $G_j = \graph u_j$, where there is $\alpha \in (0,1)$ so that
$u_j \in C^{1,\alpha}(D_2;\calA_2)$ for all $j$.
Suppose that at the point $Z \in \spt \norm{V}$, there is a tangent cone of the form
\begin{equation}
	\abs{\Pi_1^0} \times \bR e_{n+1},
2 \abs{\Pi_1^0} \times \bR e_{n+1},
\text{ or } (\abs{\Pi_1^0} + \abs{\Pi_2^0}) \times \bR e_{n+1}
\in \vartan(V,Z),
\end{equation}
where $\Pi_1^0,\Pi_2^0$ are two distinct $n-1$-dimensional planes in $\bR^n$.
Then there is $\rho > 0$ so that
\begin{equation}
	\langle \nu(X) , e_{n+1} \rangle = 0
	\text{ for all $X \in \reg V \cap B_\rho(Z)$}.
\end{equation}
\end{cor}

This holds in arbitrary dimensions, but for the remainder we restrict to the range
$2 \leq n \leq 6$.

\begin{thm} 
\label{thm_degiorgi_splitting_in_low_dimensions}
Let $\alpha \in (0,1)$ and $2 \leq n \leq 6$.
Let $u \in C^{1,\alpha}(\bR^n;\calA_2)$ be
an entire two-valued minimal graph and $\bC$ be a blowdown cone of
$\abs{G}$ at infinity.
Then
\begin{enumerate}[label = (\roman*), font = \upshape]
	\item \label{item_blowdown_cylindrical}
		either $\bC$ is cylindrical of the form
		$\bC = \bC^0 \times \bR e_{n+1}$,
	\item \label{item_blowdown_cylindrical_plus_plane}
		or $\bC = \abs{\Pi} + \bC^0 \times \bR e_{n+1}$
		where $\Pi \in \Gr(n,n+1)$,
	\item \label{item_blowdown_sum_of_planes}
		or $\bC$ is the sum of two possibly equal planes $\Pi_1,\Pi_2 \in \Gr(n,n+1)$,
		$\bC = \abs{\Pi_1} + \abs{\Pi_2}$.
\end{enumerate}
\end{thm}

The remainder is dedicated to proving this theorem, starting by decomposing
the blowdown cone $\bC$ into a vertical and a horizontal part.
We construct this decomposition as follows.  We consider the set $\calR$ of connected
components of $\reg \bC$. (This set has at most countably many elements by a
classical separability argument.)
By~\cite[Thm.~41.1]{Simon84} every $\Sigma \in \calR$ has
has constant multiplicity $\Theta_\Sigma \in \bZ_{>0}$.
We say that $\Sigma$ is \emph{vertical} if $\langle \nu,e_{n+1} \rangle \equiv 0$
on $\Sigma$, and \emph{horizontal} if instead $\langle \nu,e_{n+1} \rangle > 0$.
Thus $\calR = \calR_v \cup \calR_h$, both of which are allowed to be empty.

\begin{lem} 
\label{lem_decomposition_cone}
Let $\alpha \in (0,1)$ and $2 \leq n \leq 6$. Let $u \in C^{1,\alpha}(\bR^n;\calA_2)$
be an entire two-valued minimal graph and $\bC \in \IV_n(\bR^{n+1})$ be
a blowdown cone of $\abs{G}$ at infinity. Then $\bC_v = \bC_v^0 \times \bR e_{n+1}
= \sum_{\Sigma \in \calR_v} \Theta_\Sigma \abs{\Sigma}$ and
$\bC_h = \sum_{\Gamma \in \calR_h} \Theta_\Gamma \abs{\Gamma}$ are stationary
integral varifolds, and 
\begin{equation}
	\label{eq_decomp_cone_sum_vert_horiz}
	\bC = \bC_v + \bC_h
	\in \IV_n(\bR^{n+1}).
\end{equation}
\end{lem}
\begin{proof}
First, the convergence of the two sums can be justified because their
weight measures are bounded by $\norm{\bC}$.
Let us write the argument out explicitly for $\bC_v$. Assume that $\calR_v$
is countably infinite, enumerated by $\calR_v = \{ \Sigma_{i} \mid i \in \bN \}$
say.
For every compact subset $K \subset \bR^{n+1}$,
$\sum_{\Sigma \in \calR_v} \Theta_\Sigma \norm{\Sigma}(K) \leq \norm{\bC}(K) \leq C_K$,
so that the partial sums $\sum_{i = k}^{\infty} \Theta_{\Sigma_i} \abs{\Sigma_{i}} \to 0$
when $k \to \infty$ as varifolds.
Thus the sum $\sum_{i =1}^\infty \Theta_{\Sigma_i} \abs{\Sigma_{i}}$ is convergent,
with limit $\bC_v \in \IV_n(\bR^{n+1})$.
As every $\abs{\Sigma_{i}} \in \IV_n(\bR^{n+1})$ is invariant under homotheties,
the same holds for their limit, which we are thus justified in denoting by $\bC_v$.
Similarly one may check that indeed $\bC_v$ is vertical, meaning it is of
the form $\bC_v = \bC_v^0 \times \bR e_{n+1}$ for some $\bC_v^0 \in \IV_{n-1}(\bR^n)$.

Proceeding similarly one can justify the construction of $\bC_h$, and confirm
that $\bC = \bC_v + \bC_h$ as in~\eqref{eq_decomp_cone_sum_vert_horiz}.
Moreover, the stationarity of $\bC$ means that $\bC_v$ is stationary in
the open set $\bR^{n+1} \setminus \spt \norm{\bC_h}$ and vice-versa for $\bC_h$.
The only way either of the two cones 
could fail to be stationary in $\bR^{n+1}$ is if
\begin{equation}
	\calH^{n-1}(\spt \norm{\bC_v} \cap \spt \norm{\bC_h} \cap B_1) > 0.
\end{equation}
By construction $\spt \norm{\bC_v} \cap \spt \norm{\bC_h} \subset \sing \bC$,
which is stratified like
\begin{equation}
	\calS^0 \subset \cdots \subset \calS^{n-2}
	\subset \calS^{n-1} \subset \calS^n,
\end{equation}
where we abbreviate $\calS^i = \calS^i(\bC)$.
By Corollary~\ref{cor_sing_of_limit_cones}, we further have
$\calH^{n-1}(\calB(\bC) \cup \calS^{n-2}) = 0$,
whence
\begin{equation}
	\calH^{n-1}(\spt \norm{\bC_v} \cap \spt \norm{\bC_h} 
	\setminus ( \calS^{n-1} \setminus \calS^{n-2})) = 0.
\end{equation}

Now assume $\calH^{n-1}(\spt \norm{\bC_v} \cap \spt \norm{\bC_h} \cap B_1) > 0$,
and take a  point
\begin{equation}
	X_0 \in (S^{n-1} \setminus \calS^{n-2})
	\cap \spt \norm{\bC_v} \cap \spt \norm{\bC_h} \cap B_1.
\end{equation}
The classification of classical tangent cones established in the previous
section (e.g.\ see Corollary~\ref{cor_sing_of_limit_cones} again),
and valid for the range of dimensions 
$2 \leq n \leq 6$
prescribed in the hypotheses, implies that $\spt \norm{\bC}$ must be immersed near $X_0$.
Therefore both $\bC_v$ and $\bC_h$ must be embedded near $X_0$, say 
 $B(X_0,\rho_0) \cap \spt \norm{\bC} \subset \reg \bC_h \cup \reg \bC_v$
for some $\rho_0 > 0$, which are transversely intersecting.
Both $\reg \bC_h$ and $\reg \bC_v$ have separately pointwise vanishing
mean curvature, and in particular they are both stationary near $X_0$.
As $X_0$ was chosen arbitrarily, this proves that both $\bC_v$ and $\bC_h$
are stationary as varifolds in $\IV_n(\bR^{n+1})$.
\end{proof}

The three cases~\ref{item_blowdown_cylindrical}, \ref{item_blowdown_cylindrical_plus_plane}
and \ref{item_blowdown_sum_of_planes} listed in
Theorem~\ref{thm_degiorgi_splitting_in_low_dimensions}
correspond to the following situations.
In the first case $\bC_h = 0$, and the conclusion is immediate, while in
the last $\bC_v = 0$ and we conclude using the uniform gradient bounds for $u$,
as demonstrated in Theorem~\ref{thm_grad_bounds_vert_line}.
Probably the most complicated case of the three remains, in which both
$\bC_v \neq 0$ and $\bC_h \neq 0$. 

\begin{lem}
\label{lem_Cv_nonzero_no_classical_singularities}
Let $\alpha \in (0,1)$ and $2 \leq n \leq 6$. Let $u \in C^{1,\alpha}(\bR^n;\calA_2)$
be a two-valued minimal graph and $\bC = \bC_v + \bC_h \in \IV_n(\bR^{n+1})$ be
a blowdown cone of $\abs{G}$ at infinity.
If $\bC_v \neq 0$ and $\bC_h \neq 0$ then
$\bC_h = \abs{\Pi}$ for some plane $\Pi \in \Gr(n,n+1)$.
\end{lem}
\begin{proof}
Define a function $Q: \bR^n \to [0,\infty]$ which for all $y \in \bR^n$
counts the number of points in
$\spt \norm {\bC_h} \cap P_0^{-1}(\{ y \})$ with multiplicity:
$Q(y) = \sum_{Y \in P_0^{-1}(\{ y \})} \Theta(\norm{\bC_h},Y).$
As we are working with cones, this function $Q$ is constant along open rays
through the origin: for all $y \in \bR^n$ and $\lambda > 0$,
$Q(\lambda y) = Q(y)$.
Let $K = P_0(\calS^{n-2}(\bC) \cup \calB(\bC) \cup \sing \bC_v)$.
The convergence $\abs{G_j} \to \bC$ means that at all $y \in \bR^n \setminus K$,
$Q(y) \in \{ 0 , 1, 2 \}$.
The function $Q$ is locally constant on the set $U = \bR^n \setminus K \cap \{ Q > 0 \}$.
We show that $U$ is connected, and that in fact it is equal to $\bR^n \setminus K$.

First off, the set $U$ is open. Let $y \in U$ and $Y = (y,Y^{n+1}) \in \spt \norm{\bC_h}$.
Depending on whether $\Theta(\norm{\bC_h},Y) = 1$ or $2$, the tangent cone to $\bC_h$
at $Y$ is either a single plane $\abs{\Pi_Y}$ with $\Pi_Y \in \Gr(n,n+1)$
or $\abs{\Pi_{Y,1}} + \abs{\Pi_{Y,2}}$ with $\Pi_{Y,1},\Pi_{Y,2} \in \Gr(n,n+1)$.
Moreover by Corollary~\ref{cor_vertical_tangent_cones} these planes cannot be
vertical. Applying either Allard's regularity theorem if $\Theta(\norm{\bC_h},Y) = 1$
or the two-valued branched sheeting theorem of Wickramasekera~\cite{Wic_MultTwoAllard},
Theorem~\ref{thm_wic_mult_two_allard}
if $\Theta(\norm{\bC_h},Y) = 2$
one finds that there is a radius $\rho > 0$
and either a single smooth function $u_Y \in C^\infty(D_\rho(y))$ so that
$\spt \norm{\bC_h} \cap B_\rho(Y) \subset \graph u_Y$, or else two such functions
$u_{Y,1},u_{Y,2} \in C^\infty(D_\rho(y))$ so that
$\spt \norm{\bC_h} \cap B_\rho(Y) \subset \graph u_{Y,1} \cup \graph u_{Y,2}$.
In particular $Q > 0$ in a neighbourhood of $y$.

Next we show that $U \subset \bR^n \setminus K$ is relatively closed.
Let $(y_k \mid k \in \bN)$ be a sequence of points in $U$, with $y_k \to y$ as
$k \to \infty$ for some $y \in \bR^n \setminus K$.
For all $k \in \bN$ write $Y_k = (y_k,Y_k^{n+1}) \in \spt \norm{\bC_h}$,
and let $Z_k = Y_k / \abs{Y_k}$. After extracting a subsequence, which
we do without relabelling, we have $Z_k \to Z$ for some $Z \in \bdary B_1$.
By upper semicontinuity of density, $Q(y) = Q(z) \geq \Theta(\norm{\bC_h},Z) \geq 1$.
Finally, arguing as in the proof of Lemma~A.1 of \cite{SimonWickramasekera16}
one finds that $\bR^n \setminus K$ is connected which means that $U$ is connected
as well. In fact $U = \bR^n \setminus K$, and $Q$ is constant on it,
taking the value $q \in \{ 1 , 2 \}$.

We next show that $q = 1$, by assuming that $q = 2$ instead, and deriving a contradiction
with our tangent cone classification, which we obtain
near points where the supports of $\bC_v,\bC_h$ intersect.
There are two cases to distinguish: 
\begin{equation}
	\text{either }
	\label{eq_two_cases_final_argument}
	\reg \bC_v^0 \setminus P_0(\sing \bC_h)
	\neq \emptyset \quad \text{or } \reg \bC_v^0 \subset P_0(\sing \bC_h).
\end{equation}
We start with the first, and pick an arbitrary point $y \in \reg \bC_v^0 \setminus 
P_0(\sing \bC_h)$. This has two pre-images $Y_i = (y,Y_i^{n+1})$ under $P_0$,
which by assumption are also in $\reg \bC_v$.
By Corollary~\ref{cor_vertical_tangent_cones} the tangent cones to $\bC$ at the two points
are $\abs{\Pi_i^v} + \abs{\Pi_i^h} \in \vartan(\bC,Y_i)$
where $\Pi_i^v,\Pi_i^h \in \Gr(n,n+1)$ and $\Pi_i^v$ is vertical but $\Pi_i^h$ is not.
Because the two planes $\Pi_i^h$ are not vertical, there exists a radius $\rho > 0$
and $u_i \in C^2(D_\rho(y))$ so that $\spt \norm{\bC_h} \cap B_\rho(Y_i) \subset \graph u_i$;
moreover there is $\theta \in (0,1)$ depending only on the inclination of the two planes
so that $\graph u_{i} \cap D_{\theta \rho}(y) \times \bR \subset \spt \norm{\bC_h}
\cap B_\rho(Y_i)$.
From the convergence $G_j \to \bC$ and invoking again Theorem~\ref{thm_wic_mult_two_allard}
of Wickramasekera, we find that, after perhaps slightly adjusting the value of $\rho$,
and taking large enough $j$, there exist functions $u_{ji} \in C^2(D_\rho(y))$ so that
$\graph u_{ji} \cap D_{\theta \rho}(y) \times \bR \subset G_j \cap B_\rho(Y_i)$.
Now notice that in fact the two functions define a selection for the two-valued
$u_j$ on the disc $D_{\theta \rho}(y)$, and in fact $G_j \cap D_{\theta \rho}(y) \times \bR
= ( \graph u_{j1} \cup \graph u_{j2} ) \cap D_{\theta \rho}(y) \times \bR$.
As $\abs{\graph u_{ji}} \mres D_{\theta \rho}(y) \times \bR \to \abs{ \Pi_{i}^h}
\mres D_{\theta \rho}(y) \times \bR$ as $j \to \infty$
we find that $\bC$ cannot have a vertical component,
in contradiction to our original assumption that $\bC_v \neq 0$.

We turn to the second case, where it is assumed that
$\reg \bC_v^0 \subset P_0(\sing \bC_h)$.
As $\calH^{n-1}(P_0(\sing \bC_h \setminus \calC(\bC_h))) = 0$
but $\reg \bC_v^0$ has positive $\calH^{n-1}$-measure,
we can pick a point $y \in \reg \bC_v^0 \setminus
P_0(\sing \bC_h \setminus \calC(\bC_h))$ which further has 
\begin{equation}
	P_0^{-1}(\{ y \}) \cap \spt \norm{\bC_h} \subset \calC(\bC_h).
\end{equation}
In fact more is true, because
\begin{equation}
	\label{eq_projection_singular_sets}
	P_0(\calC(\bC_h)) \cap P_0(\sing \bC_h \setminus \calC(\bC_h)) = 0.
\end{equation}
To see this, take $z \in D_1 \cap P_0(\calC(\bC_h))$ and argue as above to show that
$P_0^{-1}(\{ z \}) \cap \spt \norm{\bC_h} \subset \sing \bC_h$.
Let $Z = (z,Z^{n+1}) \in \calC(\bC_h)$.
Write $\abs{\Pi_{Z,1}} + \abs{\Pi_{Z,2}} \in \vartan(\bC_h,Z)$, neither of which
is vertical by Corollary~\ref{cor_vertical_tangent_cones}. 
By Theorem~\ref{thm_wic_mult_two_allard},
for small $\rho > 0$ 
and large enough $j$, there are two smooth functions
$U_{ji} \in C^\infty(\Pi_{Z,i} \cap B_{2\rho}(Z);\Pi_{Z,i}^\perp)$ so that
$\abs{G_j} \mres B_\rho(Z) = \{ \abs{\graph U_{j1}} + \abs{\graph U_{j2}} \}
\mres B_\rho(Z)$.
There is $\theta > 0$ depending only on the inclination of $\Pi_{Z,1},\Pi_{Z,2}$
for which there is a smooth selection for $u_j$ on $D_{\theta \rho}(z)$:
there are $u_{j,1},u_{j,2} \in C^\infty(D_{\theta \rho}(z))$ with
$\graph u_{ji} \subset \graph U_{ji}$.
Write $Z_1 = Z$. If there were a point $Z_2 \in P_0^{-1}(\{ z \}) \cap \sing \bC_h
\setminus \calC(\bC_h)$ then no matter how small $\sigma > 0$ is chosen we could
pick $j$ large enough that $G_j \cap B_\sigma(Z_2) \neq \emptyset$.
This, however, is absurd if the radius is so small that $\sigma < \theta \rho$
and $B_\rho(Z_1) \cap B_\sigma(Z_2) = \emptyset.$
In short $P_0^{-1}(\{ z \}) \cap \sing \bC_h \subset \calC(\bC_h)$, which
demonstrates the validity of~\eqref{eq_projection_singular_sets}.

Therefore, having chosen the point $y \in \reg \bC_v^0 \setminus
P_0(\sing \bC_h \setminus \calC(\bC_h))$ we can pick a radius $\rho > 0$
small enough that
\begin{equation}
	\reg \bC_v^0 \cap D_\rho(y) \subset P_0(\calC(\bC_h)).
\end{equation}
Next let $Y \in P_0^{-1}(\{ y \}) \cap \calC(\bC_h)$ be the unique 
singular point lying above $y$, and denote its tangent cone
\begin{equation}
	\abs{\Pi_{Y,1}} + \abs{\Pi_{Y,2}} \in \vartan(\bC_h,Y),
\end{equation}
with neither plane $\Pi_{Y,1},\Pi_{Y,2}$ vertical.
By assumption the point $Y$ also belongs to $\reg \bC_v$, with respect to which it
has the tangent cone
\begin{equation}
	\abs{\Pi_Y^v} = \abs{\Pi_{y,0}^v} \times \bR e_{n+1}
	\in \vartan(\bC_v,Y).
\end{equation}
Blowing up $\bC$ at $Y$ one finds
$\abs{\Pi_{Y,1}} + \abs{\Pi_{Y,2}} + \abs{\Pi_y^v} \in \vartan(\bC,Y)$.
Let $L_Y = \Pi_{Y,1} \cap \Pi_{Y,2}$ be the $(n-1)$-dimensional axis along which
the two non-vertical planes. When blowing up the assumed inclusion
$\reg \bC_v^0 \subset P_0(\sing \bC_h)$, one finds $\Pi_{y,0}^v \subset P_0(L_Y)$.
Therefore the three planes meet along this very same axis $L_Y$, and the cone
$\abs{\Pi_{Y,1}} + \abs{\Pi_{Y,2}} + \abs{\Pi_y^v} \in \vartan(\bC,Y)$ is
classical. This is absurd, because it contradicts our classification of classical
limit cones.

We have thus shown, in both cases described in~\eqref{eq_two_cases_final_argument},
that $Q \equiv 1$ on $D_1 \setminus K$. Using Allard regularity for example there is
a smooth, single-valued function $U_h \in C^\infty(D_1 \setminus K)$ so that
$\bC_h \mres  (D_1 \setminus K) \times \bR = \abs{\graph U_h} $. 
As $\calH^{n-1}(K) = 0$ and $K$ is a locally compact subset of $D_1$, we are
precisely in a setting investigated by Simon~\cite{Simon_Erasable_Singularity_Result},
whose results prove that in fact $U_h$ can be extended smoothly across $K$.
It is then a standard fact that $U_h$ is linear, that is $\abs{\graph U_h}
= \abs{\Pi_h} \mres D_1 \times \bR$ for some horizontal plane $\Pi_h \in \Gr(n,n+1)$.
Then also $\abs{\bC_h} = \abs{\Pi_h}$, which is precisely what was to prove.
\end{proof}

\section{The Bernstein theorem in four dimensions}

\label{sec_bernstein_theorem_in_four_dimensions}

In this section we complete the proof of the main result:
the Bernstein theorem in dimension four.

\begin{thm}
	\label{thm_bernstein_dim_four_chapter}
Let $n + 1 = 4$, $\alpha \in (0,1)$, and $u \in C^{1,\alpha}(\bR^n;\calA_2)$
be a two-valued function whose graph $G$ is minimal.
Then $u$ is linear, and there are two three-dimensional planes $\Pi_1,\Pi_2
\in \Gr(3,4)$ so that $\abs{G} = \abs{\Pi_1} + \abs{\Pi_2}$.
\end{thm}

\subsection{Stability and the logarithmic cutoff trick}
\label{sec_GMT_lemma}

By Theorem~\ref{thm_degiorgi_splitting_in_low_dimensions} the cone $\bC$ must take one
of the following three forms:
\begin{equation}
	\label{eq_three_forms_of_cone}
	 \bC^0 \times \bR e_{4},  \abs{\Pi_1} + \bC^0 \times \bR e_{4}
	\text{ or } \abs{\Pi_1} + \abs{\Pi_2},
\end{equation}
where $\Pi_1,\Pi_2$ are two non-vertical, possibly equal planes,
and $\bC^0 \in \IV_2(\bR^3)$ is a stationary integral cone.
The remainder of this section is dedicated to excluding the first two cases,
that is necessarily $\bC = \abs{\Pi_1} + \abs{\Pi_2}$. 
The proof starts with the observation that $\bC^0$ inherits the
ambient stability from $\bC$. Using the so-called logarithmic cutoff trick
one finds that $\bC^0$ has $A_{\bC^0} \equiv 0$ on $\reg \bC^0$.

\begin{lem}
\label{lem_GMT_immersed_flat_union_of_planes}
Suppose $n \geq 2$.
Let $\bC \in \IV_n(\bR^{n+1})$ be a stationary integral cone,
with support immersed outside of the origin and $\abs{A_{\bC}} \equiv 0$
on $ \reg \bC$.
%
Then $\bC$ is supported in a finite union of $n$-dimensional planes.
\end{lem}

\begin{proof}
Let $\Pi \in \Gr(n,n+1)$ be any $n$-dimensional linear plane with $\reg \bC
\cap \Pi \neq \emptyset$.
The set $\Pi \setminus \{ 0 \}$ is connected because $n \geq 2$. As
$\spt \norm{\bC} \cap \Pi \setminus \{ 0 \}$ is a relatively closed subset of
$\Pi \setminus \{ 0 \}$, we only need to show that it is open to obtain
$\Pi \subset \spt \norm{\bC}$.
A point $X \in \spt \norm{\bC} \cap \Pi \setminus \{ 0 \}$ is either regular
or an immersed classical singularity. In both cases the fact that $\abs{A_\bC} \equiv 0$
on $\reg \bC$ means that near $X$ the support of $\bC$ is either a plane
or a union of planes.
By assumption $X \in \Pi$, so one of these planes must be $\Pi$ itself. This shows
that $\spt \norm{\bC} \cap \Pi \setminus \{ 0 \}$ is open inside $\Pi \setminus \{ 0 \}$, 
and thus $\Pi \subset \spt \norm{\bC}$.
Repeating this, we find a finite collection of planes $\Pi_1,\dots,\Pi_D$
so that $\calH^n(\spt \norm{\bC} \setminus \cup_i \Pi_i) = 0$, and hence
$\spt \norm{\bC} \subset \cup_i \Pi_i$ using the monotonicity formula.
\end{proof}

To apply Lemma~\ref{lem_GMT_immersed_flat_union_of_planes} to $\bC^0$
we first need to show that it is immersed outside the origin.
This is essentially a consequence of Corollary~\ref{cor_sing_of_limit_cones},
which offers the following two possibilities.
\begin{enumerate}[label = (\arabic*)]
\item Either $\Theta(\norm{\bC},X) = 2$ for $\calH^3$-a.e.\ $X \in \reg \bC$,
	and then $\bC$ is smooth embedded. This makes it impossible that
	$\bC = \bC^0 \times \bR e_4 + \Pi_1$, and in the case where
	$\bC = \bC^0 \times \bR e_4$ we find that $\bC^0 = 2 \abs{\Pi^0}$
	for some two-dimensional plane $\Pi^0 \subset \bR^3$. 
	Thus $\bC = 2 \abs{\Pi^0 \times \bR e_4}$, and we can conclude by
	the monotonicity formula.

\item The second possibility is that the density of $\bC$ is $\calH^3$-a.e. equal
	one, and thus automatically also $\Theta(\norm{\bC^0},X) = 1$
	for $\calH^2$-a.e.\ $X \in \reg \bC^0$.
	Recall that the singular set of $\bC^0$ is stratified like
	$\calS^0(\bC^0) \subset \calS^1(\bC^0) \subset \calS^2(\bC^0)$.
	Invoking Corollary~\ref{cor_sing_of_limit_cones} again we find that
	that $\spt \norm{\bC^0}$ is immersed near points of $\calS^1(\bC^0)
	\setminus \calS^0(\bC^0)$, and the remaining singularities necessarily
	have $\calS^0(\bC^0) \cup \calB(\bC^0) \subset \{ 0 \}$.
\end{enumerate}
In both cases $\bC^0$ is immersed outside the origin, and by
Lemma~\ref{lem_GMT_immersed_flat_union_of_planes} we find that $\bC^0$ is supported
in a union of planes.
In the remainder we need only consider the second possibility,
where the density of $\bC^0$ at all regular points is one, and
the cone is equal to a sum of planes, all of which are vertical and have multiplicity one.
(Note Lemma~\ref{lem_GMT_immersed_flat_union_of_planes} only gives that $\bC^0$
is supported in a union of planes.)
In the remainder we write $\bC^0 = \bP^0 = \sum_{j=2}^D \abs{\Pi_j^0}$ to reflect this.
At this stage of the proof we have reduced the possible forms of the blowdown
cone given in~\eqref{eq_three_forms_of_cone} to
\begin{equation}
	\bP^0 \times \bR e_4 \text{ or } \abs{\Pi_1} + \bP^0 \times \bR e_4.
\end{equation}
The two require different approaches, and we treat the latter first.

\subsection{Non-vertical blowdown cones}

Write $\bL^0 \subset \bR^3$ for the union of one-dimensional lines along which
the planes in the support of $\bP^0$ meet.
The singularities of the blowdown cone are 
$\sing \bC = (\Pi_1 \cap \spt \norm{\bP^0 \times \bR e_4}) \cup \bL^0 \times \bR e_{4}$,
and those lying in $\Pi_1 \cap \spt \norm{\bC} \setminus (\bL^0 \times \bR e_4)$
are all immersed.
Let $\tau > 0$ be given. Using Allard's regularity theorem near the points
of $\Pi_1 \cap \reg \bC$ and Wickramasekera's stable sheeting theorem near
those in $\Pi_1 \cap \sing \bC \setminus (\bL^0 \times \bR e_4)$ we find the
existence of a smooth function $u_{j,1} \in C^\infty(D_1 \setminus [\bL^0]_\tau)$
with $G_{j,1} = \graph u_{j,1} \subset G_j \cap (\Pi_1)_\tau$, at least provided
$j \geq J(\tau)$ is large enough.
We obtain a smooth selection $u_{j,1},u_{j,2} \in C^\infty(D_1 \setminus [\bL^0]_\tau)$
by picking the remaining value of $u_j$ for $u_{j,2}$ above every point.
(It is not enough to observe that eventually $\calB_{u_j}
\cap D_1 \subset [\bL^0]_\tau$, as the set $D_1 \setminus [\bL^0]_\tau$
is not simply connected regardless of how small $\tau > 0$ is.)
As $\tau \to 0$ and $j \geq J(\tau) \to \infty$ we find that by construction
$\abs{G_{j}^1} \to \abs{\Pi_1}$ and $\abs{G_j^2} \to \bP^0 \times \bR e_4$.
As the $G_j^2$ are all single-valued graphs we find that $\bP^0$ is necessarily
supported in a single plane, say $\bP^0 = \abs{\Pi_2^0}$. (There are various
ways of confirming this in more detail, all boiling down to the fact that
$\bP^0$ cannot be the limit of a sequence of area-minimising currents if
it is supported in more than one plane. To give but one example, revisiting the arguments
used to prove the improved area estimates we obtain that there is $\delta > 0$
so that $\norm{\bP^0 }(D_1) \leq (2 - \delta) \omega_2$.)
Therefore $\abs{G_j} \to \abs{\Pi_1} + \abs{\Pi_2}$, as desired.

\subsection{Vertical blowdown cones: the adjacency graph}

Here the blowdown sequences converges to a vertical cone,
$\abs{G_j} \to \bP = \bP^0 \times \bR e_4 = \sum_{j=1}^D \abs{\Pi_j^0} \times \bR e_4$.
In the current topology $\cur{G_j} \to \sum_{j=1}^D \cur{\Pi_j^0} \times \bR e_4$
where the planes are respectively oriented by unit normals $n_1,\dots,n_D$.
The improved area estimates give $D \leq 3$,
see Corollary~\ref{cor_qualitative_estimates}.

The only problematic value is $D = 3$. We exclude this by a combinatorial argument,
constructing what we call the \emph{adjacency graph} by a kind of dual cellular
decomposition.
The planes $\Pi_1,\Pi_2,\Pi_3$ divide $\bR^3$ into a finite number of connected
components $\Omega_1,\dots,\Omega_N \subset \bR^3$. These are all polyhedral,
with respective boundaries $\bdary \Omega_1,\dots,\bdary \Omega_N
\subset \Pi_1 \cup \Pi_2 \cup \Pi_3$.
We say that two regions $\Omega \neq \Omega'$ are adjacent if they meet along a face.
To every component $\Omega$ we associate a vertex $v$, forming a set $V$.
Connect two distinct vertices $v,v' \in V$ by an edge $e$ if the corresponding regions
$\Omega,\Omega'$ are adjacent.
If $\Omega,\Omega'$ are adjacent then they meet along a single plane $\Pi_i$.
We orient $e$ so that it agrees with the orientation of this plane,
meaning if $n_i$ points away from $\Omega$ and into $\Omega'$
then $e$ is directed from $v$ to $v'$ and vice-versa.
Thus we obtain a set of directed edges denoted $E$. We call the finite, directed
graph $H = (V,E)$ the \emph{adjacency graph} of $\bP$.
Label the vertices of the graph by a function $F: V \to \{ 0 , 1 , 2 \}$
which returns the number $F(v)$ of sheets of $G_j$ eventually lying over the
corresponding region $\Omega \cap D_1$. This is well-defined by
Lemma~\ref{lem_description_near_planes} for example.
Let $v,v' \in V$ be two adjacent vertices, and suppose that
$e$ points from $v$ to $v'$. By Lemma~\ref{lem_consec_number_beans},
$F(v') = F(v) + 1$.
As an immediate consequence we find that $H$ cannot contain directed paths
of length more than two.
Indeed if $H$ contained three edges $e_1,e_2,e_3$ so that $e_i$ points
from $v_i$ to $v_{i+1}$ then $F(v_4) = F(v_1) + 3$, which is absurd.

There are essentially only two ways in which the planes $\Pi_1,\Pi_2,\Pi_3$
can be arranged. 
Let $\Pi_3^1 = \{ x \in \bR^3 \mid \langle x , n_3 \rangle \equiv 1 \} \subset \bR^3$
be the affine plane parallel to $\Pi_3^0$ at height one.
The two planes $\Pi_1^0$ and $\Pi_2^0$ intersect this 
transversely in a pair of affine lines $l_1,l_2$.
If these lines were parallel, then the planes $\Pi_1,\Pi_2,\Pi_3$ would meet
along a common axis, making $\bP = \bP^0 \times \bR e_4$ a classical cone.
As we have already dealt with these, we may assume this is not the case.

Hence we may assume the two lines $l_1,l_2$ intersecting,
and compute the adjacency graph. The set $\Pi_3^1 \setminus (l_1 \cup l_2)$
has four connected components.
Each of these leads to a pair of adjacent vertices in $V$, which correspond
to regions meeting along a face in $\Pi_3$. Thus $H$ contains eight vertices,
arranged as four pairs of vertices lying on either side of $\Pi_3$. 
Additionally the four vertices corresponding to the regions
contained inside $\{ x \in \bR^3 \mid \langle x , n_3 \rangle > 0 \}$
are arranged in a square in $H$, with parallel edges oriented in the same direction.
The same holds for the regions lying in the half-space
$\{ x \in \bR^3 \mid \langle x, n_3 \rangle < 0 \}$.
In short, $H$ is a cube with eight vertices and twelve edges, with
parallel edges pointing in the same direction. As this graph contains a
directed path of length three, we have reached a contradiction.

\appendix



\section{Singularities and regularity of minimal surfaces}

\label{sec_appendix_regularity_and_singularities}

Here we record some results from geometric measure theory,
and define some notation that we rely on in our arguments.
We mainly work with integral varifolds, but we sometimes also require
some results for integer-density rectifiable currents. 
For their basic theory one may consult the book of Simon~\cite{Simon84},
or indeed any of the other standard references.
Here we concentrate on the regularity of stationary integral varifolds,
and quote some results concerning their singularities.
We state all results for codimension one surfaces, although some of them%
---notably those regarding the stratification of their singular sets---%
remain valid in higher codimensions.

\subsection{Stratification of the singular set}
\label{subsec_stratification}

Let $U \subset \bR^{n+1}$ be an open set. We write $\IV_n(U)$ for the space
of $n$-dimensional integer-density rectifiable varifolds in $U$,
and $\I_n(U)$ for the space of integral currents in $U$, namely those integer-density
currents in $U$ whose boundary also defines an integer-density rectifiable current.
(This includes the $n$-dimensional cycles, which have zero boundary by definition.)

Let $V \in \IV_n(U)$ be a stationary varifold.
A point $X \in U \cap \spt \norm{V}$ is called \emph{regular} if there
is a radius $\rho > 0$ so that $B_\rho(X) \cap \spt \norm{V}$ is an embedded
surface. 
A point which is not regular is called \emph{singular}.
We denote the \emph{regular set} $\reg V$ and the \emph{singular set} $\sing V$.

Consider a point $X \in U \cap \sing V$ and a tangent cone $\bC \in \vartan(V,X)$.
Let $V \in \bR^{n+1}$ be a vector, and $\tau_V: X \in \bR^{n+1} \mapsto X - V \in \bR^{n+1}$
be the corresponding translation by $-V$.
The set $\calS(\bC) = \{ V \in \bR^{n+1} \mid \tau_{V\#} \bC = \bC \}$
of vectors that leave $\bC$ invariant forms a vector space called the \emph{spine} of $\bC$.
For $0 \leq m \leq n$ let
$\calS^m(V) = \{ X \in U \cap \sing V \mid \dim \calS(\bC) \leq m \text{ for all }
\bC \in \vartan(V,X) \}$.
If $\dim \calS(\bC) = m$ then there is a stationary cone
$\bC' \in \IV_{n-m}(\bR^{n-m+1})$ so that $\bC = \bC' \times \calS(\bC)$.
 Let $\bC \in \IV_n(\bR^{n+1})$ be a stationary
integral cone. Then for $X \in \spt \norm{\bC} \setminus \{ 0 \}$ every tangent
cone $\bC_X \in \vartan(\bC,X)$ has $\calS(\bC_X) \neq \{ 0 \}$ because
$\bR \cdot X \subset \calS(\bC_X)$. This means that there exists $\bC_X' \in
\IV_{n-1}(\bR^{n})$ so that $\bC_X = \bC_X' \times \bR \cdot X$.
Then we have the so-called \emph{Almgren--Federer stratification theorem}.
\begin{thm}[\cite{Almgren1981}]
Let $U \subset \bR^{n+1}$ be open, and $V \in \IV_n(U)$ be a stationary
integral varifold. Then $\dim_\calH \calS^m(V) \leq m$ for all $0 \leq m \leq n$.
\end{thm}

Under the same hypotheses,
Naber--Valtorta~\cite{NaberValtorta_rectifiability_stationary_varifolds}
improve this to the following.

\begin{thm}[\cite{NaberValtorta_rectifiability_stationary_varifolds}]
Let $U \in \bR^{n+1}$ be open, and $V \in \IV_n(U)$ be a stationary
integral varifold. Then $\calS^m(V)$ is countably $k$-rectifiable
for all $0 \leq m \leq n$.
\end{thm}

In the literature both $\calS^k(V)$ and $(\calS^k \setminus \calS^{k-1})(V)$
are sometimes called \emph{strata} of the singular set of $V$.
The singular set of a stationary $V \in \IV_n(U)$ can be divided into $\sing V
= \calS^{n-2}(V) \cup (\calS^{n} \setminus \calS^{n-2})(V)$.
By Naber--Valtorta's result, the lower strata gathered into $\calS^{n-2}(V)$
can be excised by suitable sequences of test functions,
using a capacity argument.
The top stratum $(\calS^n \setminus \calS^{n-1})(V)$ is also called the
\emph{branch set} of $V$, and denoted $\calB(V)$. A singular point $X$ belongs to $\calB(V)$,
and is called a \emph{branch point} if at least one cone $\bC \in \vartan(V,X)$
is of the form $\bC = Q \abs{\Pi}$ for some $n$-dimensional plane $\Pi \in \Gr(n,n+1)$
with multiplicity $Q \in \bZ_{>0}$.
(By Allard regularity the multiplicity of such a branch point must be at least two.)
Note that a point $X \in \reg V$ near which $V$ coincides with an embedded minimal surface
with multiplicity two would not be considered a branch point according to
this convention. (These points are sometimes called \emph{false branch points}.)

The next stratum $(\calS^{n-1} \setminus \calS^{n-2})(V)$ is formed by
those points $X \in U \cap \sing V$ near which at least one tangent cone
$\bC \in \vartan(V,X)$ is of the form $\bC = \sum_{i=1}^D m_i \abs{\pi_i}$ where
$m_i \in \bZ_{\geq 1}$ and the $\pi_i$ are $n$-dimensional half-planes meeting
along a common $(n-1)$--dimensional axis $L$.
We call these \emph{classical cones}.
This is motivated by the term \emph{classical singularities}, which refers to the points
in $\sing V$ near which the support of $V$ is a union of minimal hypersurfaces
meeting along a common boundary.
We write $\calC(V)$ for the set of classical singularities in $V$.
This is the case for example for singularities that are smoothly immersed.
A classical singularity in $V$ necessary belongs to $(\calS^{n-1} \setminus \calS^{n-2})(V)$,
but in principle not all points in $(\calS^{n-1} \setminus \calS^{n-2})(V)$ are
classical singularities.

An example of particular relevance to us is where $D = 4$, $m_1,\dots,m_4 = 1$
and the half-planes $\pi_1,\dots,\pi_4$ form a union of
two $n$-dimensional planes $\Pi_1,\Pi_2 \in \Gr(n,n+1)$, say $\Pi_1 = \pi_1 \cup \pi_3$
and $\Pi_2 = \pi_2 \cup \pi_4$.
In other words $\bC = \abs{\Pi_1} + \abs{\Pi_2}$. Such tangent cones
arise for example at \emph{immersed singularities}, where there are
two sheets $\Sigma_1,\Sigma_2$ which are separately embedded and
meet transversely along the $(n-1)$--dimensional singular axis $\Sigma_1 \cap \Sigma_2$.
For stable minimal hypersurfaces the converse holds as well, see
Section~\ref{subsec_wic_mult_two_allard} below.

Let $\Omega \subset \bR^n$ be an open domain, and $u \in \Lip(\Omega;\calA_2)$
be a two-valued, locally Lipschitz function. Its graph defines an integral
varifold $\abs{G} \in \IV_n(\Omega \times \bR)$, taken with multiplicity one.
(When endowed with an orientation it yields an integral current
$\cur{G} \in \I_n(\Omega \times \bR)$.)
When in fact $u \in C^1(\Omega;\calA_2)$ then $\abs{G}$ can only have
two kinds of singularities, namely branch points and classical, immersed
singularities, respectively forming the sets $\calB(G)$ and $\calC(G)$.

\subsection{The second variation formula and stability}

Here to start off we assume that $\reg V$ is orientable, and let $\xi \in
C_c^1(U \setminus \sing V;\bR^{n+1})$ be an arbitrary vector field.
Its restriction to the support of $V$ can be written $\varphi N$, where
$N$ is a choice of unit normal to $\reg V$ and $\varphi \in C_c^1(\reg V)$.
Let $K \subset U$ be any compact set containing $\spt \varphi$. Writing $(\Phi_t)$
for the flow generated by $\xi$, the \emph{second variation} of the area of $V$
can be expressed as
\begin{equation}
	\label{eq_second_variation_formula}
	\frac{\diff^2}{\diff t^2} \restr{ \norm{\Phi_{t\#} V}(K)}{t = 0}
	= \int_{K \cap \reg V}
	\abs{\nabla_V \varphi}^2 - \abs{A_V}^2 \varphi^2 \intdiff \norm{V},
\end{equation}
where 
$\nabla_V$ and $A_V$ are the gradient operator and second fundamental form
on $\reg V$ respectively.

A varifold $V \in \IV_n(U)$ is said to have \emph{stable regular part} if
this is non-negative for all perturbations $\varphi \in C_c^1(\reg V)$, that is
$\int \abs{A_V}^2 \varphi^2 \intdiff \norm{V}
\leq \int \abs{\nabla_V \varphi}^2 \intdiff \norm{V}$; this is the 
\emph{stability inequality}.
This automatically gives integral curvature bounds away from
the singular set. Indeed, let $X \in U \cap \reg V$ be a point with
$\dist(X,\sing V) > 2R$. If we let $\varphi \in C_c^1(B_{2R}(X))$ be
a standard cutoff function, with $\varphi = 1$ on $B_R(X)$ and
$\abs{D \varphi} \leq 2/R$ then~\eqref{eq_second_variation_formula} yields
$\int_{\reg V \cap B_R(X)} \abs{A_V}^2 \intdiff \norm{V}
\leq 4R^{-2} \norm{V}(B_{2R}(X))$.
On $\reg V$ we can define the linear, elliptic \emph{Jacobi operator}
$L_V = \Delta_V + \abs{A_V}^2$. Let $W \subset \subset U \setminus \sing V$
be another open set, and write $(\lambda_p(W) \mid p \in \bN)$ for the
spectrum of $L_V$ with zero Dirichlet eigenvalues on $\bdary W \cap \reg V$.
Although this is not necessary, assume for simplicity that the regular part
is connected, so that by the Constancy Theorem~\cite[Thm~41.1]{Simon84}
the density of $V$ is constant on $\reg V$.
If $\reg V$ is stable in $U$, then the stability inequality becomes
$\int_{U \cap \reg V} \abs{A_V}^2 \varphi^2 \intdiff \calH^n
\leq \int_{U \cap \reg V} \abs{\nabla_V \varphi}^2 \intdiff \calH^n$.
After integrating by parts one gets that $\lambda_p(W) \geq 0$ for all $W
\subset \subset U \cap \reg V$ and all $p \in \bN$.
Additionally we point out that the symmetries of $V$ lead to \emph{Jacobi fields},
namely functions $f \in C^2(\reg V)$ which solve $\Delta_{V} f + \abs{A_V}^2 f = 0$
in the classical, pointwise sense.
In particular by translating the varifold in the direction of a vector $v \in \bR^{n+1}$
we find $\Delta_V \langle N , v \rangle + \abs{A_V}^2 \langle N , v \rangle = 0$.

Although we imposed orientability of $\reg V$ in our derivation,
stability is \emph{a posteriori} well-defined regardless of this. 
This remains true for the second notion of stability we introduce,
which we call \emph{ambient stability}.
Suppose that for all compact $K \subset U$,
$\int_{K \cap \reg V} \abs{A_V}^2 \intdiff \norm{V} < \infty$.
Then we say that $V$ is \emph{ambient stable} if the inequality
\begin{equation}
	\label{eq_stab_ineq_varifold_ambient}
	\tag{$S_V$}
	\int_{U \cap \reg V} \abs{A_V}^2 \varphi^2 \intdiff \norm{V}
	\leq \int_{U \cap \reg V} \abs{\nabla_V \varphi}^2 \intdiff \norm{V}
\end{equation}
holds for all $\varphi \in C_c^1(U)$.
We emphasise here that both the local bounds for the curvature and the stability
inequality hold for any compact subsets of $U$, not just those that avoid
the singularities of $V$.
This in turn allows the application of the theory developed by Hutchinson
in~\cite{Hutchinson86}, where it was proved that ambient stability is
preserved under weak convergence of varifolds.
\begin{prop}[\cite{Hutchinson86}]
\label{prop_hutchinson_stability}
Let $U \subset \bR^{n+1}$ be open, and let 
$(V_j \mid j \in \bN)$ be a sequence of stationary varifolds in $U$
satisfying~\eqref{eq_stab_ineq_varifold_ambient}.
Suppose that $V_j \to V \in \IV_n(U)$ weakly in the varifold topology.
Then $V$ is stationary and ambient stable.
\end{prop}

\subsection{Wickramasekera's branched sheeting theorem}

\label{subsec_wic_mult_two_allard}

\begin{thm}[\cite{Wic_MultTwoAllard}]
\label{thm_wic_mult_two_allard}
There is $\eps = \eps(n) > 0$ so that if $V \in \IV_n(B_2)$
is stationary with stable regular part, $\norm{V}(B_2)/(\omega_n 2^n)
< 2 + \eps$, $\Theta(\norm{V},Y) \neq 3/2$ for all $Y \in B_2$ and
\begin{enumerate}[label = (\roman*), font=\upshape]
\item either $\int_{B_2} \dist(X,\Pi)^2 \intdiff \norm{V}(X) < \eps$
	for some $\Pi \in \Gr(n,n+1)$,
\item or $\int_{B_2}  \dist(X,\Pi_1)^2 \wedge \dist(X,\Pi_2)^2 
	\intdiff \norm{V}(X) 
	+ \int_{B_2} \dist(X,\spt \norm{V})^2 \allowbreak
	\intdiff (\norm{\Pi_1} + \norm{\Pi_2})(X)
	< \eps$ for some $\Pi_1 \neq \Pi_2 \in \Gr(n,n+1)$.
\end{enumerate}
Then 
\begin{enumerate}[label = (\roman*), font=\upshape]
\item either $B_1 \cap \spt \norm{V} \subset \reg V \cup \calC(V) \cup \calB(V)$
	and there is $\gamma = \gamma(n,\eps) \in (0,1)$ and 
	a two-valued function
$u \in C^{1,\gamma}(B_1 \cap \Pi;\calA_2(\Pi^\perp))$ so that
$B_1 \cap \spt \norm{V} \subset \graph u \subset  \reg V \cup \calC(V) \cup \calB(V)$,
and there is $C = C(n,\eps) > 0$ so that
$\abs{u}_{1,\gamma;B_1 \cap \Pi}
\leq C \big( \int_{B_2} \dist(X,\Pi)^2 \intdiff \norm{V}(X) \big)^{1/2}.$
\item or $B_1 \cap \spt \norm{V} \subset \reg V \cup \calC(V)$ and there is
$\gamma = \gamma(n,\eps) \in (0,1)$ and
two single-valued functions $u_i \in C^{1,\gamma}(B_1 \cap \Pi_i,
\Pi_i^\perp)$ so that $B_1 \cap \spt \norm{V} \subset \cup_i \graph u_i
\subset \reg V \cup \calC(V)$,
and there is $C = C(n,\eps) > 0$ so that
$\abs{u_i}_{1,\gamma;B_1 \cap \Pi}
\leq C \big( \int_{B_2}  \dist(X,\Pi_1)^2 \wedge \dist(X,\Pi_2)^2  \intdiff \norm{V}(X)
+ \int_{B_2} \dist(X,\spt \norm{V})^{1/2} \intdiff (\norm{\Pi_1} + \norm{\Pi_2})(X)
\big)^{1/2}.$
\end{enumerate}
\end{thm}

Elliptic regularity theory is not available for two-valued minimal graphs.
However, Simon--Wickramasekera~\cite{SimonWickramasekera16} have shown
that necessarily $u \in C^{1,1/2}(B_1 \cap \Pi;\calA_2(\Pi^\perp))$.
Of course, in the second case described above,
where $V$ is close to $\abs{\Pi_1} + \abs{\Pi_2}$ and $V \cap \spt \norm{V}
\subset \reg V \cup \calC(V)$ the two functions $u_i$ are regular.
The usual, single-valued elliptic regularity then
gives that for all $l \in \bZ_{>0}$ there is $C(l) = C(n,l)$ so that
$\abs{u_i}_{l,\gamma;B_1 \cap \Pi_i} \leq 
C(l) \big(
 \int_{B_2} \dist(X,\Pi_1)^2 \wedge \dist(X,\Pi_2)^2  \intdiff \norm{V}(X)
+ \int_{B_2} \dist(X,\spt \norm{V})^{1/2} \intdiff (\norm{\Pi_1} + \norm{\Pi_2})(X).
\big)^{1/2}$

\begin{thm}[\cite{Wic_MultTwoAllard}]
	Let $U \subset \bR^{n+1}$ be open and $V \in \IV_n(U)$ be stationary
	with stable regular part, and $\Theta(\norm{V},Y) \neq 3/2$ for all $Y \in U$.
	Then there is $\eps = \eps(n) \in (0,1)$
	so that if $Y \in U \cap \spt \norm{V}$ has
	$\Theta(\norm{V},Y) < 2 + \eps$ then
	$Y \in \reg V \cup \calB(V) \cup \calC(V)$.
\end{thm}

Let us quickly comment on these three possibilities. Let $V \in \IV_n(U)$
be as above, and let $Y \in U \cap \spt \norm{V}$ have
$\Theta(\norm{V},Y) = 2$. Then either $Y \in \calB(V)$ or else
there is $0 < \rho < \dist(Y,\bdary U)$ so that
\begin{enumerate}
\item either $B_\rho(X) \cap \spt \norm{V} \subset \reg V$ and there
	is a smooth embedded $\Sigma$ so that
	$V \mres B_\rho(X) = 2 \abs{\Sigma}$,
\item or $B_\rho(X) \cap \spt \norm{V} \subset \reg V \cup \calC(V)$
	is immersed and there are two smooth embedded surfaces
	$\Sigma_1 , \Sigma_2$ which meet transversely along an axis
	of immersed, classical singularities, so that
	$V \mres B_\rho(Y) = \abs{\Sigma_1} + \abs{\Sigma_2}$
	and $\sing V \cap B_\rho(Y) = \Sigma_1 \cap \Sigma_2$.
\end{enumerate}

\bibliographystyle{alpha}

\bibliography{min_surfaces_refs.bib}

\end{document}